\documentclass{article}
\usepackage[journal=APDE,lang=british]{ems-journal-jems} 

\usepackage{textcomp}
\usepackage{mathrsfs}

\usepackage{chngcntr}
\usepackage{apptools}
\usepackage{verbatim}
\usepackage{tikz-cd}
\usepackage{subcaption}

\newtheorem{proposition}{Proposition}

\theoremstyle{definition}

\newtheorem{conjecture}{Conjecture}

\theoremstyle{remark}
\newtheorem{remark}{Remark}
\newtheorem{corollary}{Corollary}

\numberwithin{equation}{section}

\newcommand\restr[2]{{
		\left.\kern-\nulldelimiterspace 
		#1 
		\vphantom{\big|} 
		\right|_{#2} 
}}

\numberwithin{equation}{section}

\numberwithin{theorem}{section}
\numberwithin{proposition}{section}
\numberwithin{lemma}{section}
\numberwithin{remark}{section}
\numberwithin{exercise}{section}
\numberwithin{corollary}{section}

\newcounter{desccount}

\begin{document}

\title{Robust upper estimates for topological entropy via nonlinear constrained optimization over adapted metrics}
\titlemark{Nonlinear constrained optimization over adapted metrics}


\emsauthor{1}{
	\givenname{Mikhail}
	\surname{Anikushin}
	\orcid{0000-0002-7781-8551}}{M.M.~Anikushin}
\emsauthor{2}{
	\givenname{Andrey}
	\surname{Romanov}
	\orcid{0009-0008-4886-2351}}{A.O.~Romanov}

\Emsaffil{1}{
	\pretext{}
	\department{Department of Applied Cybernetics}
	\organisation{Faculty of Mathematics and Mechanics, St Petersburg University}
	\rorid{01a2bcd34}
	\address{Universitetskiy prospekt 28}
	\zip{198504}
	\city{Peterhof}
	\country{Russia}
	\posttext{}
	\affemail{demolishka@gmail.com}
	}
\Emsaffil{2}{
	\pretext{}
	\department{Department of Applied Cybernetics}
	\organisation{Faculty of Mathematics and Mechanics, St Petersburg University}
	\rorid{01a2bcd34}
	\address{Universitetskiy prospekt 28}
	\zip{198504}
	\city{Peterhof}
	\country{Russia}
	\posttext{}
	\affemail{romanov.andrey.twai@gmail.com}
}
	

	


\classification[37M22, 90C26]{37M25}

\begin{abstract}
	We present an analytical-numerical method providing robust upper estimates for the topological entropy or, more generally, uniform volume growth exponents of differentiable mappings. By introducing varying metrics, we simplify the analysis at the cost of generally rougher bounds, but keeping the prospect of choosing more relatable metrics to refine the estimates. With any covering of an invariant set by a finite number of cubes, we associate a graph describing overlaps (edges) of the cubes (vertices) under iterates of the mapping. Weighing vertices according to a given metric, we reduce the problem to finding simple cycles with maximal relative weights. Then we develop an algorithm concerned with iterative resolving nonlinear programming problems for optimization of maximal relative weights in general smooth families of metrics which may involve interpolation or neural networks models. We describe applications of the algorithm to compute the largest uniform Lyapunov exponent and uniform Lyapunov dimension for the H\'{e}non and Rabinovich systems justifying the Eden conjecture at stationary and periodic points respectively.
\end{abstract}

\maketitle
\tableofcontents

\section{Introduction}

In the study of dissipative dynamical systems, one is interested in the structure of attractors. There are many known characteristics of attractor dynamics related to dimension, entropy and volume expansion. However, when it comes to computation, especially in the case of applied nonlinear models, we face obstacles related to the lack of rigorous justification for what is being computed in numerical experiments.

One of basic reasons for this is that most of such characteristics are discontinuous w.r.t. the system in adequate generality. So, we need a deep knowledge about a particular system to be able to compute them. Moreover, the possibility of discontinuous oscillations under perturbations of the system casts doubts on the necessity of exact computation of such characteristics or at least demands understanding relevant disturbances.

As an alternative, one may try to bound the characteristics of interest by quantities which are more robust and computable. A classical result of this kind is the Douady-Oesterl\'{e} estimate \cite{DouadyOesterle1980} for the Hausdorff dimension which led to the notion of (uniform) Lyapunov dimension and various generalizations in finite and infinite dimensions (see R.~Temam \cite{Temam1997}; V.V.~Chepyzhov and A.A.~Ilyin \cite{ChepyzhovIlyin2004}; N.V.~Kuznetsov and V.~Reitmann \cite{KuzReit2020}; S.~Zelik \cite{ZelikAttractors2022}). More special characteristics related to the study of fractal dimensions arise in the case of nonconformal repellers (see J.~Chen and Y.~Pesin \cite{ChenPesin2010}) or almost periodic attractors (see our works \cite{Anikushin2019Liouv} and \cite{AnikushinReitRom2019} (joint with V.~Reitmann)).

It was demonstrated by G.A.~Leonov that Lyapunov dimension can be effectively computed by introducing Lyapunov-like functions into the corresponding estimates (see his survey \cite{LeonovBoi1992} joint with V.A.~Boichenko). In particular, he applied the method to analytically compute Lyapunov dimension for invariant sets in H\'{e}non \cite{LeonovForHenonLor2001} and global attractors in Lorenz-like systems \cite{Leonov2016LorenzLike}. After his name, this approach was called the \textit{Leonov method} by N.V.~Kuznetsov \cite{Kuznetsov2016}. By now, an exact Lyapunov dimension formula for the global attractor of the Lorenz system is established for any standard parameters (see G.A.~Leonov et al. \cite{LeoKuzKorKusakin2016}; N.V.~Kuznetsov et al. \cite{KuzMokKuzKud2020}).

By the Kaplan-Yorke formula, one can express the Lyapunov dimension via uniform Lyapunov exponents. From this it is seen the important connection with the classical result of S.B.~Katok \cite{KatokSB1980} providing an upper estimate for topological entropy via the sum of positive uniform Lyapunov exponents. We refer to Section \ref{SEC: EntHistoricalBackground} for precise definitions and more discussions on the estimation of topological entropy.

From the above, it is not surprising that the Leonov method turned out to work well with the estimation of topological entropy. Recent works of A.S.~Matveev and A.Yu.~Pogromsky \cite{MatveevPogromsky2016} and the same authors joint with C.~Kawan \cite{KawanPogromsky2021} explore this area. Besides applications, \cite{KawanPogromsky2021} is concerned with the possibility of theoretical computation via general adapted metrics. Related results were also independently obtained by the first author in \cite{Anikushin2023LyapExp} for infinite-dimensional noninvertible systems (nuances are to be discussed in Section \ref{SEC: EntHistoricalBackground}).

On the level of geometry, the Leonov method is concerned with variations of a (usually space-constant) metric in its conformal class via Lyapunov-like functions. Recently, C.~Kawan, S.~Hafstein and P.~Giesl \cite{KawanHafsteinGiesl2021} and M.~Louzeiro et al. \cite{LouzeiroetAlAdaptedMetricsComp2022} developed subgradient methods resolving the related optimization problems in the class of polynomial Lyapunov-like functions and constant matrices defining the metrics. In these works, the method is based on the optimization of what is called the \textit{maximized exponent} in \cite{Anikushin2023LyapExp} (see our Section \ref{SEC: ULEandAdaptedMetrics}). In the case of \cite{KawanHafsteinGiesl2021,LouzeiroetAlAdaptedMetricsComp2022} it is given by the maximum of certain sums of logarithms of singular values over an invariant region. This is a nonsmooth optimization problem and for effective applications of the subgradient method it demands restricting to geodesically convex subsets in the space of metrics. The latter places essential restrictions on the considered classes of metrics. To the best of our knowledge, geodesic convexity has been established only for the simplest class of metrics coming from the Leonov method. Moreover, one should not expect convexity to be satisfied for general smooth families of metrics which one wishes to use in optimization.

Let us note that known successes of the Leonov method are related to the cases where the Lyapunov dimension (or consecutive sums of uniform Lyapunov exponents) can be realized as the analogous value over an ergodic measure concentrated at a stationary state. From this perspective, it is interesting that neither \cite{KawanHafsteinGiesl2021} nor \cite{LouzeiroetAlAdaptedMetricsComp2022} succeeded to obtain relevant estimates for the largest uniform Lyapunov exponent of the classical H\'{e}nnon attractor (containing only the positive equilibrium); both results significantly differ from the stationary asymptotics which seems to be the true one (this will be justified in Section \ref{SEC: HenonLPestimate}).

It is expected that in greater generality the uniform characteristics are achieved at periodic points (see Section \ref{SEC: RobustAlgoEstimateULE} for more detailed discussions). In the case of Lyapunov dimension, such a statement for particular systems is known as the \textit{Eden conjecture} which comes from the studies of A.~Eden \cite{EdenLocalEstimates1990, Eden1989Thesis}. In this work, we are aimed to develop analytical-numerical optimization machinery to provide a reliable way to justify the Eden conjecture for particular systems.

Namely, we develop a method based on the optimization of \textit{averaged exponents} in terms of \cite{Anikushin2023LyapExp} (see our Section \ref{SEC: ULEandAdaptedMetrics}). In other words, instead of maximizing certain exponents over a region containing (=localizing) a given compact invariant subset, we consider their time-averages over trajectories of the system, maximize such averages and take the limit as time goes to infinity. This limit can be effectively estimated using only short-time calculations and space-discretizations of the mapping.

For this, we cover the localizing region by small subsets and consider the associated graph (called by us a \textit{rough symbolic image}) describing overlaps (edges) of such subsets (vertices) under iterates of the mapping. Similar constructions were previously used in other works related to entropy (see, for example, G.~Froyland, O.~Junge and G.~Ochs \cite{FroylandJungeochs2001}; M.S.~Tomar, C.~Kawan and M.~Zamani \cite{TomarKawanZamani2022}; G.S.~Osipenko and N.B.~Ampilova \cite{OsipenkoAmpilova2019}) and rigorous computations in dynamics based on the Conley index theory (see the survey K.~Mischaikow \cite{Mischaikow2002}, where the related construction is called an outer approximation by a multivalued combinatorial map).
\begin{remark}
	Many natural connections between properties of rigorous symbolic images\footnote{We define a rough symbolic image (see Section \ref{SEC: RobustAlgoEstimateULE}) taking into account possibilities of computing intersections up to a given precision to make computations more reliable. If the intersections are computed exactly, we speak of a \textit{rigorous} symbolic image. In \cite{Osipenko2023, Osipenko2006} and related works of G.S.~Osipenko, the term \textit{symbolic image} is used to denote what we call rigorous symbolic images with equal transition times. However, many of his results are in the spirit of outer approximations and therefore they should be valid in the setting more appropriate for computations.} and dynamics of the corresponding mappings are established in works of G.S.~Osipenko (see \cite{Osipenko2023, Osipenko2006} and references therein).
	
	We also note that a more general construction called an \textit{approximate Markov chain}, where one may also take into account a finite past (so vertices of appropriate graphs correspond to a fixed-length sequences of covering subsets), is studied by S.M.~Pincus \cite{Pincus1992, Pincus1991ApEn}. This approach and the related notion of \textit{approximate entropy} turned out to be very successful in studying experimental data \cite{Pincus1991ApEn}. Although it is straightforward to generalize our constructions to this general case, accurate computations for applied deterministic models become much more complicated. Moreover, a similar effect can be achieved via refinements of paths in the graph through right-resolving presentations as in \cite{Osipenko2006, FroylandJungeochs2001}. \qed
\end{remark}

After this, we weight the vertices according to a given metric by maximizing the exponents only in the corresponding subset. Then elementary arguments show that the problem is reduced to maximizing relative weights over simple cycles in the graph (see Proposition \ref{PROP: GraphPeriodicMeasure}). This procedure (called the \textit{Relative Weights Optimization}) always results in values which are bounded from above by the maximized exponent. In particular, the above mentioned theoretical results of \cite{Anikushin2023LyapExp, MatveevPogromsky2016} give that one can always adapt the metric to approximate the consecutive sum of uniform Lyapunov exponents with a given accuracy for any partition and even when the graph is complete. However, here we are aimed to illustrate how relatively accurate computations of the dynamical graph may help to construct relevant metrics by adding a dynamical guidance to the optimization process.

By means of the Relative Weights Optimization, in Section \ref{SEC: OptimizatiobViaINP} we develop an optimization algorithm in general smooth families of metrics called by us the \textit{Iterative Nonlinear Programming}. It is concerned with iterative resolution of nonlinear programming problems on data (called \textit{reference cycles} and \textit{reference points}) extracted during optimization in the rough symbolic image. A key feature is that weights at these reference points are optimized not individually, but in the context of collected cycles which become highly overlapped during the optimization process. This adds a guidance by dynamics that helps achieving only relevant minima. In case of success, we also obtain a localization of an extreme trajectory in the original system. In other words, the algorithm results in an implicit approximation of the maximizing measure.

Let us also note that, unlike the Relative Weights Optimization which is a computationally robust procedure\footnote{Since the consecutive sum of singular values is globally Lipschitz in the space of matrices, the resulting computations can be potentially made rigorous in many cases.}, justifications of the Iterative Nonlinear Programming involve heuristics particularly related to chaotic dynamics which cannot be rigorously proven. However, we illustrate effectiveness of our approach by means of two systems.

Namely, in Section \ref{SEC: HenonLPestimate}, we justify that the largest uniform Lyapunov exponent of the classical H\'{e}non attractor is achieved at the positive equilibrium. This is done by constructing an adapted metric (see Tab.~\ref{TAB: HenonMetricValues}) providing an upper bound which is accurate up to 5 decimal places to the expected value. Here the Iterative Nonlinear Programming optimization shows a fast convergence.

In Section \ref{SEC: RabinovichLargestULE}, we study a more challenging problem represented by the Rabinovich system (a chaotic Lorenz-like system) with parameters for which it was previously numerically shown that neither the largest uniform Lyapunov exponent nor the Lyapunov dimension are achieved at any equilibrium and found a short periodic orbit which is a candidate for the extreme trajectory. Justifying this is a challenging problem. We construct an adapted metric using a neural network model with $2006$ parameters which results in an estimate accurate up to $3$ decimal places. In this case, the optimization process took 2 months on a personal computer (see the section for details).

Although in this paper we are interested in chaotic systems, it should be mentioned that our approach shall efficiently work for globally stable systems. Here rough symbolic images have simple structure and, consequently, the only problem is to construct them. In particular, by means of our method we can verify the global stability criterion based on dimension estimates (see M.Y.~Li and J.S.~Muldowney \cite{LiMuldowney1996SIAMGlobStab}; R.A.~Smith \cite{Smith1986HD}; our work \cite{AnikushinRomanov2023FreqConds} and Section \ref{SUBSEC: ComputationOfULE} for more discussions).

Let us briefly describe structure of the work. In Section \ref{SEC: EntHistoricalBackground}, we give preliminary definitions and provide a discussion on volume expansion and entropy for differentiable mappings on manifolds. In Section \ref{SEC: ULEandAdaptedMetrics}, we demonstrate how adapted metrics can be used to track volume expansions. In Section \ref{SEC: RobustAlgoEstimateULE}, we expose the Relative Weights Optimization algorithm and discuss related problems. In Section \ref{SEC: OptimizatiobViaINP}, the Iterative Nonlinear Programming optimization is discussed. In Section \ref{SEC: ComputationalTips}, related computational nuances are discussed. In Section \ref{SEC: HenonLPestimate}, we justify the Eden conjecture for the H\'{e}non mapping with classical parameters. In Section \ref{SEC: RabinovichLargestULE}, we justify the Eden conjecture for the Rabinovich system with particular parameters. In Appendix \ref{SEC: LyapunovFloquetMetrics}, we introduce Lyapunov-Floquet metrics and emphasize their role for optimization. In Appendix \ref{SEC: OperatorsOnExteriorProducts}, we briefly discuss operators on exterior powers and their singular values.
\section{Preliminary definitions and historical background}
\label{SEC: EntHistoricalBackground}

Let us start with some rigorous definitions concerning infinitesimal volumes, their evolution and measurements. Suppose $\vartheta$ is a differentiable mapping from an $n$-dimensional smooth manifold $\mathcal{M}$ to itself. Consider the \text{discrete forward-time space}\footnote{In the text, we prefer to use the abstract symbol $\mathbb{T}_{+}$ instead of just $\mathbb{Z}_{+}$ because most of definitions and results work also in the continuous-time case, i.e. $\mathbb{T}_{+} = \mathbb{R}_{+} = [0,+\infty)$.} $\mathbb{T}_{+} := \mathbb{Z}_{+} = \{0,1,2,\ldots\}$. For $t \in \mathbb{T}_{+}$, let $\Xi^{t}(q,\cdot)$ be the differential of $\vartheta^{t}$ at $q \in \mathcal{M}$ which is a linear mapping taking the tangent space $T_{q}\mathcal{M}$ at $q$ to the tangent space $T_{\vartheta^{t}(q)}\mathcal{M}$ at $\vartheta^{t}(q)$. This family will be denoted simply by $\Xi$. It forms a \textit{cocycle} on the tangent bundle $T\mathcal{M} = \bigcup_{q \in \mathcal{M}}T_{q}\mathcal{M}$ over the dynamical system $(\mathcal{M},\vartheta)$, i.e. for any $t,s \in \mathbb{T}_{+}$ and $q \in \mathcal{M}$ we have
\begin{equation}
	\Xi^{t+s}(q,\cdot) = \Xi^{t}(\vartheta^{s}(q),\Xi^{s}(q,\cdot))
\end{equation}
with $\Xi^{0}(q,\cdot)$ being the identity mapping of $T_{q}\mathcal{M}$.

For each $m \in \{1,\ldots,n\}$, we consider the $m$-th exterior power\footnote{We refer to Appendix \ref{SEC: OperatorsOnExteriorProducts} for a brief introduction to operators on exterior powers and singular values.} $T^{\wedge m}_{q}\mathcal{M}$ of $T_{q}\mathcal{M}$ (for $m=1$ this is just $T_{q}\mathcal{M}$ itself). Recall that it is spanned by finite linear combinations of antisymmetric tensors $\xi_{1} \wedge \ldots \wedge \xi_{m}$ with $\xi_{1},\ldots,\xi_{m} \in T_{q}\mathcal{M}$. Then there is a well-defined family of mappings given by
\begin{equation}
	\label{EQ: CompoundCocycleDef}
	\Xi^{t}_{m}(q, \xi_{1} \wedge \ldots \wedge \xi_{m} ) := \Xi^{t}(q,\xi_{1}) \wedge \ldots \wedge \Xi^{t}(q, \xi_{m})
\end{equation}
which also form a cocycle on the $m$-th exterior bundle $T^{\wedge m}\mathcal{M} = \bigcup_{q \in \mathcal{M}}T^{\wedge m}_{q}\mathcal{M}$ over $(\mathcal{M},\vartheta)$. We denote this family by $\Xi_{m}$ and call $\Xi_{m}$ the \textit{$m$-fold (antisymmetric) multiplicative compound} of $\Xi$.

Aimed to measure the growth of volumes under $\Xi_{m}$, we suppose $\mathcal{M}$ is endowed with a Riemannian metric, i.e. an inner product $\langle \cdot, \cdot \rangle_{q}$ in each tangent space $T_{q}\mathcal{M}$ smoothly depending on $q$. It induces a family $\mathfrak{n} = \{ \mathfrak{n}_{q}\}_{q \in \mathcal{M}}$ of norms on $T\mathcal{M}$, where $\mathfrak{n}_{q}(\xi) := \sqrt{\langle \xi, \xi \rangle_{q}}$ for $\xi \in T_{q}\mathcal{M}$ and $q \in \mathcal{M}$. We call $\mathfrak{n}$ a \textit{metric} on $T\mathcal{M}$.

Then on each $T^{\wedge m}_{q}\mathcal{M}$ there is the associated inner product (for convenience defined by the same symbol) $\langle \cdot, \cdot \rangle_{q}$ as in \eqref{EQ: InnerProductExteriorPower}. Analogously, we may consider the family of norms $\mathfrak{n}^{\wedge m} = \{ \mathfrak{n}^{\wedge m}_{q} \}_{q \in \mathcal{M}}$ generated by these inner products. We say that $\mathfrak{n}^{\wedge m}$ is the metric on $T^{\wedge m}\mathcal{M}$ associated with $\mathfrak{n}$.

Let $\mathcal{K} \subset \mathcal{M}$ be a compact subset which is invariant w.r.t. $\vartheta$, i.e. $\vartheta^{t}(\mathcal{K}) = \mathcal{K}$ for all $t \in \mathbb{T}_{+}$. As a motivation for this work, we are interested in upper estimates for the topological entropy $h_{\operatorname{top}}(\vartheta;\mathcal{K})$ of the dynamical system $(\mathcal{K},\vartheta)$. Here we omit definitions of $h_{\operatorname{top}}(\vartheta;\mathcal{K})$ since we will work only with quantities providing its upper estimates.

In the next subsection, we give an extended discussion on the existing approaches for estimation of topological entropy. Although most of it does not directly related to the present work, it explores difficulties in obtaining relevant estimates for topological entropy and therefore places the present study in the broader context.

\subsection{Topological entropy and volume expansion}

Suppose $\mu_{L}$ is the Lebesgue measure on $\mathcal{M}$ associated with the Riemannian metric and let $\mathcal{U}$ be a bounded set with nonempty interior. For $m \in \{1,\ldots,m\}$, we define the \textit{expansion exponent of $m$-volumes}\footnote{To the best of our knowledge, there is no name for the corresponding quantity in the literature.} in $\mathcal{U}$ as
\begin{equation}
	\label{EQ: MeanEntropyOfVolumes}
	\lambda^{vol}_{m}(\Xi;\mathcal{U}) \coloneq \limsup_{t \to +\infty} \frac{\log \int_{\mathcal{U}_{t}} \|\Xi^{t}_{m}(q,\cdot)\|_{\mathfrak{n}^{\wedge m}} d\mu_{L}(q) }{t},
\end{equation}
where $\mathcal{U}_{t} \coloneq \{ q \in \mathcal{U} \ | \ \vartheta^{s}(q) \in \mathcal{U} \text{ for all } 0 \leq s \leq t \}$.
\begin{remark}
	\label{REM: NormsBundleOperatorsConvention}
	Let us emphasize that $\|\Xi^{t}_{m}(q,\cdot)\|_{\mathfrak{n}^{\wedge m}}$ means the norm of $\Xi^{t}_{m}(q,\cdot)$ as an operator between normed spaces $(T^{\wedge m}_{q}\mathcal{M}, \mathfrak{n}^{\wedge m}_{q})$ and $(T^{\wedge m}_{\vartheta^{t}(q)}\mathcal{M}, \mathfrak{n}^{\wedge m}_{\vartheta^{t}(q)})$ respectively.
\end{remark}

Suppose that $\mathcal{U}$ (as above) contains $\mathcal{K}$ in its interior. Then for $\vartheta$ being a $C^{r}$-differentiable mapping with real $r > 1$, it is known that
\begin{equation}
	\label{EQ: TopologicalEntropyEstimate}
	h_{\operatorname{top}}(\vartheta;\mathcal{K}) \leq \max_{1 \leq m \leq n} \lambda^{vol}_{m}(\Xi;\mathcal{U}).
\end{equation}
For diffeomorphisms on compact manifolds with $\mathcal{K}=\mathcal{M}$ and $\mathcal{U}=\mathcal{M}$, this was established by F.~Przytycki \cite{Przytycki1980}. In \cite{Newhouse1988EntropyVol}, S.~Newhouse refined the result as we stated it (see the remark on p.~296 therein). Note that the requirement $r > 1$ is related to the use of nonuniform partial hyperbolicity (Pesin theory) in the proofs, where the additional smoothness is required to construct the corresponding foliations and linearize the mapping near supports of ergodic measures. In fact, the proof shows that \eqref{EQ: TopologicalEntropyEstimate} holds when $\limsup$ is exchanged with $\liminf$ in the definition \eqref{EQ: MeanEntropyOfVolumes}.

For $C^{\infty}$-differentiable mappings on a compact manifold $\mathcal{M}$ we have
\begin{equation}
	\label{EQ: EntropyViaVolumesEquality}
	h_{\operatorname{top}}(\vartheta;\mathcal{M}) = \max_{1 \leq m \leq n} \lambda^{vol}_{m}(\Xi;\mathcal{M}).
\end{equation}
Here, the estimate from below is essentially due to Y.~Yomdin \cite{Yomdin1987}, who studied the expansion of nonlocal volumes (curves, surfaces, etc.) embedded into the manifold in connection with the Shub entropy conjecture. In the form \eqref{EQ: EntropyViaVolumesEquality}, the identity appeared in the work of O.S.~Kozlovski \cite{Kozlovski1998} who systematized previous investigations. For lower regularity $r < \infty$, there indeed may exist gaps between $h_{\operatorname{top}}(\vartheta;\mathcal{M})$ and the right-hand side of \eqref{EQ: EntropyViaVolumesEquality} which can be estimated from above in terms of the largest uniform Lyapunov exponent $\lambda_{1}(\Xi;\mathcal{M})$ (see \eqref{EQ: UniformLyapunovExponentsDefinition}), $n = \dim \mathcal{M}$ and $r$ (see \cite{Yomdin1987}).

Moreover, in the above context of $C^{\infty}$-differentiable mappings, the topological entropy is upper-semicontinuous w.r.t. $\vartheta$ in the $C^{\infty}$-topology and there exists a measure of maximal entropy (see S.~Newhouse \cite{Newhouse1989ContEnt}). However, in lower regularity all of these properties may be violated. 

A direct approach related to \eqref{EQ: TopologicalEntropyEstimate} is presented by S.~Newhouse and T.~Pignataro in \cite{NewhousePignataro1993}, where it is numerically studied the evolution of a typical\footnote{It must be transverse to stable manifolds of the measure of maximal entropy.} nonlocal $m$-volume embedded into $\mathcal{M}$. Analogously, in \cite{Kozlovski1998}, O.S.~Kozlovski showed that $\lambda^{vol}_{m}(\Xi;\mathcal{M})$ can be realized as the growth exponent in the $L_{1}$-norm of a typical field of $m$-vectors or $m$-covectors (a differential form) on $\mathcal{M}$ under the action of $\Xi$. At least in the case of model examples, the approach of \cite{NewhousePignataro1993} shows a relatively fast convergence which is much more regular than in the case of Lyapunov exponents. 

There also exist combinatorial approaches for estimation of topological entropy. For arbitrary finite partitions, the method developed in the works G.~Froyland, O.~Junge and G.~Ochs \cite{FroylandJungeochs2001} and G.~Osipenko \cite{Osipenko2006} provides computation of the combinatorial entropy w.r.t. the partition. Being applied sufficiently accurate, the method leads to lower estimates for the topological entropy. Theoretically one may have an upper estimate, but this requires either the use of arbitrarily small partitions (that is impossible) or at least proving that a given partition is generating. 

Constructing partitions providing dynamically relevant grammars is a challenge task and besides the kneading theory for one-dimensional maps, the Pruning Front Conjecture for two-dimensional systems is promising on this way (see G.~D'Alessandro et al. \cite{Alessandroelal1990}; P.~Cvitanovi\'{c} et al. \cite{CvitanovicChaos2005}). Recently, some rigorous results for H\'{e}non mappings were claimed by J.P.~Boro\'{n}ski and S.~\v{S}timac \cite{BoronskiStimac2023}.

In this paper, we deal with a uniform analog of $\lambda^{vol}_{m}(\Xi;\mathcal{M})$ which is upper semicontinuous in the $C^{1}$-topology and will be defined in the next subsection.

\subsection{Uniform Lyapunov exponents and their computation}
\label{SUBSEC: ComputationOfULE}
Let $\mathcal{Q} \subset \mathcal{M}$ be a positively invariant closed bounded set. Then the \textit{uniform Lyapunov exponents} $\lambda_{1}(\Xi;\mathcal{Q}), \lambda_{2}(\Xi;\mathcal{Q}), \ldots$ of $\Xi$ over $\mathcal{Q}$ are given by induction for $m \in \{1,\ldots, n\}$ from the relations
\begin{equation}
	\label{EQ: UniformLyapunovExponentsDefinition}
	\lambda_{1}(\Xi;\mathcal{Q}) + \ldots + \lambda_{m}(\Xi;\mathcal{Q}) = \lambda_{1}(\Xi_{m};\mathcal{Q}) \coloneq \lim_{t \to +\infty} \frac{\ln \sup_{q \in \mathcal{Q}}\|\Xi^{t}_{m}(q,\cdot)\|_{\mathfrak{n}^{\wedge m}}}{t}.
\end{equation}
Note that $\lambda_{j}$ need not be nonincreasing in $j$ (see \cite{Anikushin2023LyapExp}). It is clear that $\lambda^{vol}_{m}(\Xi;\mathcal{Q}) \leq \lambda_{1}(\Xi_{m};\mathcal{Q})$, although the gap between these quantities, besides trivial cases, is always significant.

It is an immediate corollary of the Margulis-Ruelle inequality and the variational principle for the topological entropy that for any $C^{1}$-differentiable mapping we have
\begin{equation}
	\label{EQ: TopologicalEntropyViaUnExps}
	h_{\operatorname{top}}(\vartheta;\mathcal{K}) \leq \sum_{j \colon \lambda_{j} > 0} \lambda_{j}(\Xi;\mathcal{K}) = \max_{1 \leq m \leq n} \lambda_{1}(\Xi_{m};\mathcal{K}).
\end{equation}
for any invariant compact $\mathcal{K}$. However, first proofs of \eqref{EQ: TopologicalEntropyViaUnExps} utilized more direct approaches related to expansion of volumes (see S.B.~Katok \cite{KatokSB1980}), which were developed later to obtain the more delicate estimate \eqref{EQ: TopologicalEntropyEstimate}.

If $\mathcal{K} = \bigcap_{s \in \mathbb{T}_{+}} \vartheta^{s}(\mathcal{Q})$ for $\mathcal{Q}$ as above, then $\lambda_{j}(\Xi;\mathcal{Q}) = \lambda_{j}(\Xi;\mathcal{K})$ for any $j \geq 0$. It is clear then that \eqref{EQ: TopologicalEntropyViaUnExps} is rougher than \eqref{EQ: TopologicalEntropyEstimate}. However, there are two advantages of the estimate \eqref{EQ: TopologicalEntropyViaUnExps} via uniform quantities $\lambda_{1}(\Xi_{m};\mathcal{K})$.

Firstly, for each $m$, the value $\lambda_{1}(\Xi_{m};\mathcal{K})$ depends upper semicontinuously under $C^{1}$-perturbations of $\vartheta$ which are uniformly small in a neighborhood of $\mathcal{K}$ and preserve the positive invariance of $\mathcal{Q}$ (see Section A.2 in \cite{Anikushin2023LyapExp}). 

Secondly, uniform Lyapunov exponents theoretically admit computation via adapted metrics. This was established for linear cocycles in Hilbert spaces in \cite{Anikushin2023LyapExp} by the first author in connection with the Liouville trace formula, symmetrization procedure and general metrics\footnote{Not necessarily coercive in the case of continuous-time (locally) noninvertible systems. This feature is natural in infinite dimensions.} given by families of norms. In particular, there always exists a metric $\mathfrak{m}$ on $T^{\wedge m}\mathcal{M}$ over $\mathcal{Q}$ such that the maximum of $\alpha^{+}_{\mathfrak{m}}(q) \coloneq \ln \| \Xi^{1}(q,\cdot) \|_{\mathfrak{m}}$ over $q \in \mathcal{Q}$ (see Section \ref{SEC: ULEandAdaptedMetrics}) produces an upper estimate for $\lambda_{1}(\Xi_{m};\mathcal{Q})$ which is close to $\lambda_{1}(\Xi_{m})$ up to a given precision. Note that $\mathfrak{m}$ is constructed as a Lyapunov-like metric in \cite{Anikushin2023LyapExp} and need not be associated with a metric on $T\mathcal{M}$.

For finite-dimensional invertible systems, analogous characterization was obtained by C.~Kawan, A.S.~Matveev and A.Yu.~Pogromsky \cite{KawanPogromsky2021} in connection with the so-called restoration entropy given by the right-hand side of \eqref{EQ: TopologicalEntropyViaUnExps}. Namely, following geometric ideas of J.~Bochi \cite{Bochi2018}, the authors showed that $\mathfrak{m}$ as above can be taken to be associated with a metric on $T\mathcal{M}$ (and, in fact, the metric can be chosen the same for all $m$). Such a metric is constructed by taking the barycenter (in the space of metrics) of pullbacks of a given Riemannian metric. 


Nevertheless, the consideration of metrics defined directly on exterior bundles has advantages for practice. For example, such metrics arise from applications of the Frequency Theorem which produces fields of Lyapunov-like quadratic functionals via an infinite-horizon quadratic optimization procedure. Although the theorem, establishing a parallel between frequency-domain and the Lyapunov direct methods, is well-known for its success in studying nonlinear ODEs (see, e.g., the monographs \cite{LeonovBurkinShep2012, LeonovPonomarenkoSmirnova1996, KuzReit2020}), it lacked several important applications until recent works of the first author related to appropriate extensions of the theorem in infinite dimensions (see \cite{Anikushin2020FreqDelay, Anikushin2020FreqParab}), connections with inertial manifolds theory (see \cite{Anikushin2020Geom, Anikushin2022Semigroups, AnikushinAADyn2021, AnikushinRom2023SS}) and studying spectra of compound cocycles like our $\Xi_{m}$ (see \cite{Anikushin2023Comp, AnikushinRomanov2023FreqConds, AnikushinRomanov2024EffEst}).

Moreover, in connection with the present work, considering smooth families of metrics defined directly on $T^{\wedge m}\mathcal{M}$ may add a flexibility for applications, however at the cost of computational complexity, both in time and memory, since we have to deal with larger matrices.

There is a particular interest in proving that $\lambda_{1}(\Xi_{2}; \mathcal{Q}) < 0$. For a simply connected $\mathcal{Q}$ localizing an attractor $\mathcal{K}$, this allows to apply the generalized Bendixson criterion for attractors (R.A.~Smith \cite{Smith1986HD}; M.Y.~Li and J.S.~Muldowney \cite{LiMuldowney1996SIAMGlobStab, LiMuldowney1995LowBounds}). Since $\lambda_{1}(\Xi_{2}; \mathcal{Q}) < 0$ is preserved under small $C^{1}$-perturbations, this allows to deduce (at least in finite dimensions) the global stability in $\mathcal{Q}$ (i.e. the convergence of any trajectory to an equilibrium). We refer to \cite{AnikushinRomanov2023FreqConds} for more discussions.

\section{Tracking volume expansion via adapted metrics}
\label{SEC: ULEandAdaptedMetrics}
Here we will demonstrate how adapted metrics may simplify tracking the expansion of $m$-volumes. A price to pay for such a simplification is upper estimates which may be sufficiently rough for a given metric. However, this approach has the prospect of searching for more relevant metrics catching proper asymptotics via the estimates.

Choose a bounded subset $\mathcal{U} \subset \mathcal{M}$ with nonempty interior and let $\mathfrak{m} = \{\mathfrak{m}_{q}\}_{q \in \mathcal{U}_{1}}$ be a metric on $T^{\wedge m}\mathcal{Q}$ over $\mathcal{U}_{1} \coloneq \mathcal{U} \cap \vartheta^{-1}(\mathcal{U})$ which is uniformly equivalent to $\mathfrak{n}^{\wedge m}$, i.e. for some constants $C_{1},C_{2} > 0$ we have
\begin{equation}
	\label{EQ: EquivalentMetricOnVolumes}
	C_{1} \mathfrak{m}_{q}(\mathcal{P}) \leq \mathfrak{n}^{\wedge m}_{q}(\mathcal{P}) \leq C_{2} \mathfrak{m}_{q}(\mathcal{P}) \text{ for any } q \in \mathcal{U}_{1} \text{ and } \mathcal{P} \in T^{\wedge m}_{q}\mathcal{M}.
\end{equation}

We define the \textit{growth exponent} $\alpha_{\mathfrak{m}}(q,\mathcal{P})$ of $\Xi_{m}$ at nonzero $\mathcal{P} \in T^{\wedge m}_{q}\mathcal{Q}$ over $q \in \mathcal{U}_{1}$ w.r.t. $\mathfrak{m}$ as (with the convention $\ln 0 := -\infty$)
\begin{equation}
	\alpha_{\mathfrak{m}}(q,\mathcal{P}) \coloneq \ln \frac{\mathfrak{m}_{\vartheta(q)}\left(\Xi^{1}_{m}(q,\mathcal{P})\right)}{\mathfrak{m}_{q}(\mathcal{P})}.
\end{equation}
Clearly, for any $t \in \mathbb{T}_{+}$ we have (with the convention $e^{-\infty} := 0$)
\begin{equation}
	\label{EQ: LiouvilleFormulaDiscrete}
	\mathfrak{m}_{\vartheta^{t}(q)}\left(\Xi^{t}_{m}(q,\mathcal{P}) \right) = \mathfrak{m}_{q}(\mathcal{P}) \exp\left( \sum_{s=0}^{t-1}\alpha_{\mathfrak{m}}(\vartheta^{s}(q),\Xi^{s}_{m}(q,\mathcal{P})) \right)
\end{equation}
provided that $\Xi^{t-1}_{m}(q,\mathcal{P}) \not=0$. 

Note that \eqref{EQ: LiouvilleFormulaDiscrete} is a discrete analog of the Liouville trace formula. For continuous-time systems, one has to differentiate in order to obtain an analog $\alpha_{\mathfrak{m}}(q,\mathcal{P})$ (therefore they are called \textit{infinitesimal growth exponents}) so certain differentiability properties of $\Xi_{m}$ w.r.t. $\mathfrak{m}$ are required (see \cite{Anikushin2023LyapExp}). In this work, to study continuous-time systems we use their finite-time discretizations and work only with growth exponents of the corresponding mappings.

Using the \textit{maximization procedure} in each fiber $q \in \mathcal{U}_{1}$, we get
\begin{equation}
	\label{EQ: AlphaPlusFiberDef}
	\alpha^{+}_{\mathfrak{m}}(q) \coloneq \sup_{\mathcal{P} \not= 0} \alpha_{\mathfrak{m}}(q,\mathcal{P}) = \ln\| \Xi^{1}(q,\cdot) \|_{\mathfrak{m}}.
\end{equation}
\begin{remark}
	From \eqref{EQ: KeyResultExteriorOperators}, it follows that for $\mathfrak{m} = \mathfrak{n}^{\wedge m}$, i.e. metrics associated with a metric $\mathfrak{n}$ on $T\mathcal{M}$ over $\mathcal{U}_{1}$, we have (see \cite{Anikushin2023LyapExp})
	\begin{equation}
		\label{EQ: MaximizedExponentDef}
		\alpha^{+}_{\mathfrak{m}}(q) = \sup_{\xi_{1},\ldots,\xi_{m}} \alpha_{\mathfrak{m}}(q,\xi_{1} \wedge \ldots \wedge \xi_{m}),
	\end{equation}
	where the supremum is taken over all linearly independent vectors $\xi_{1},\ldots,\xi_{m} \in T_{q}\mathcal{M}$.
\end{remark}

Consequently, for $q \in \mathcal{U}_{t} = \bigcap_{s=0}^{t} \vartheta^{-s}(\mathcal{U})$, \eqref{EQ: LiouvilleFormulaDiscrete} gives
\begin{equation}
	\label{EQ: AdaptedMetricNormEstimate}
	\|\Xi^{t}_{m}(q,\cdot)\|_{\mathfrak{m}} \leq \exp\left( \sum_{s=0}^{t-1}\alpha^{+}_{\mathfrak{m}}(\vartheta^{s}(q)) \right).
\end{equation}

If $\mathcal{U} = \mathcal{Q}$ is a positively invariant subset, this gives rise to two kinds of exponents estimating $\lambda_{1}(\Xi_{m};\mathcal{Q})$. Namely, (the limit exists due to the Fekete lemma)
\begin{equation}
	\alpha^{+}_{\mathfrak{m}}(\Xi;\mathcal{Q}) \coloneq \sup_{q \in \mathcal{Q}} \alpha^{+}_{\mathfrak{m}}(q) \text{ and } \overline{\alpha}_{\mathfrak{m}}(\Xi;\mathcal{Q}) \coloneq \lim_{t \to +\infty} \frac{1}{t}\sup_{q \in \mathcal{Q}} \sum_{s=0}^{t-1}\alpha^{+}_{\mathfrak{m}}(\vartheta^{s}(q)).
\end{equation}
Here $\alpha^{+}_{\mathfrak{m}}(\Xi;\mathcal{Q})$ is the \textit{uniform growth exponent} and $\overline{\alpha}_{\mathfrak{m}}(\Xi;\mathcal{Q})$ is the \textit{averaged uniform growth exponent} of $\Xi_{m}$ w.r.t. $\mathfrak{m}$ over $\mathcal{Q}$. Clearly, 
\begin{equation}
	\lambda_{1}(\Xi_{m};\mathcal{Q}) \leq \overline{\alpha}_{\mathfrak{m}}(\Xi;\mathcal{Q}) \leq \alpha^{+}_{\mathfrak{m}}(\Xi;\mathcal{Q}).
\end{equation}
As it was discussed in Section \ref{SEC: EntHistoricalBackground}, the results of \cite{Anikushin2023LyapExp,KawanPogromsky2021} guarantee that there exist a sequence of $\mathfrak{m}$ making $\alpha^{+}_{\mathfrak{m}}(\Xi;\mathcal{Q})$ arbitrarily close to $\lambda_{1}(\Xi_{m};\mathcal{Q})$. 

Uniform growth exponents are convenient to obtain rigorous analytical estimates (see \cite{Anikushin2023LyapExp, KuzReit2020, ZelikAttractors2022}). In this work, we explore advantages of averaged uniform growth exponents $\overline{\alpha}_{\mathfrak{m}}(\Xi;\mathcal{Q})$ for analytical-numerical estimation of $\lambda_{1}(\Xi_{m};\mathcal{Q})$. Moreover, in Section \ref{SEC: RobustAlgoEstimateULE} we will consider a more general situation allowing different transition times (not just unit ones) depending on $q$.

\section{Relative weights optimization for estimation of uniform growth exponents of subadditive families}
\label{SEC: RobustAlgoEstimateULE}

In this section, on the theoretical level, it will be sufficient to work with a continuous self-map $\vartheta$ of a complete metric space $\mathcal{Q}$. Recall $\mathbb{T}_{+} = \mathbb{Z}_{\geq 0}$.

Consider a family $f=\{ f^{t} \}_{t \in \mathbb{T}_{+}}$ of functions $f^{t} \colon \mathcal{Q} \to \mathbb{R} \cup \{-\infty\}$ satisfying
\begin{description}[before=\let\makelabel\descriptionlabel]
	\item[\textbf{(SF1)}\refstepcounter{desccount}\label{DESC: SF1}] $f^{t+s}(q) \leq f^{t}(\vartheta^{s}(q)) + f^{s}(q) \text{ for any } t,s \in \mathbb{T}_{+} \text{ and } q \in \mathcal{Q}$;
	\item[\textbf{(SF2)}\refstepcounter{desccount}\label{DESC: SF2}] $f^{t} \colon \mathcal{Q} \to \mathbb{R} \cup \{-\infty\}$ is upper-semicontinuous for any $t \in \mathbb{T}_{+}$;
	\item[\textbf{(SF3)}\refstepcounter{desccount}\label{DESC: SF3}] $\sup_{t \in [0,1] \cap \mathbb{T}_{+}} \sup_{q \in \mathcal{Q}} f^{t}(q) < +\infty$.
\end{description}
Under these conditions we say that $f$ is a \textit{proper subadditive family} over $(\mathcal{Q},\vartheta)$. Then the \textit{uniform (upper) growth exponent} $\lambda_{f}$ of $f$ is defined as
\begin{equation}
	\label{EQ: UniformUpperGrowthExponentF}
	\lambda_{f} \coloneq \lim_{t \to +\infty} \frac{1}{t}\sup_{q \in \mathcal{Q}}f^{t}(q),
\end{equation}
where the limit exists according to the Fekete lemma.

\begin{remark}
	In terms of the previous section, for an invariant compact $\mathcal{K}$, our basic aim is to estimate the sum of the first $m$ uniform Lyapunov exponents of $\Xi$ over $(\mathcal{K},\vartheta)$, i.e.
	\begin{equation}
		\label{EQ: UniformLyapunovExponents}
		\lambda_{1}(\Xi;\mathcal{K}) + \ldots + \lambda_{m}(\Xi;\mathcal{K}) = \lambda_{1}(\Xi_{m};\mathcal{K}) = \lim_{t \to +\infty} \frac{\ln \sup_{q \in \mathcal{K}}\| \Xi^{t}_{m}(q,\cdot) \|_{\mathfrak{n}^{\wedge m}} }{t}.
	\end{equation}
	Let $\mathfrak{m}$ be any (continuous) metric on $T^{\wedge m}\mathcal{M}$ over $\mathcal{K}$. Clearly, the limit \eqref{EQ: UniformLyapunovExponents} remains the same if we exchange $\mathfrak{n}^{\wedge m}$ with $\mathfrak{m}$. Then clearly, the problem is to estimate $\lambda_{f}$ for the proper subadditive family $f=\{f^{t}\}_{t \in \mathbb{T}_{+}}$ given by
	\begin{equation}
		\label{EQ: SubadditiveFamilyForULEm}
		f^{t}(q) \coloneq \ln \| \Xi^{t}_{m}(q,\cdot) \|_{\mathfrak{m}} \text{ for } t \in \mathbb{T}_{+} \text{ and } q \in \mathcal{K}.
	\end{equation}
	
	Moreover, let $\mathfrak{n}$ be as above and consider the function of singular values $\omega^{(\mathfrak{n})}_{d}(\Xi^{t}(q,\cdot))$ of $\Xi^{t}(q,\cdot)$ in the metric $\mathfrak{n}$ (see \eqref{EQ: SingularValuesFunctionDef} and Remark \ref{REM: NormsBundleOperatorsConvention}). It is not hard to see that 
	\begin{equation}
		\label{EQ: SubadditiveFamilySingVal}
		f^{t}(q) = \ln \omega^{(\mathfrak{n})}_{d}(\Xi^{t}(q,\cdot)) \text{ for } t \in \mathbb{T}_{+} \text{ and } q \in \mathcal{K}
	\end{equation}
	is a proper subadditive family. For $d=m$ and $\mathfrak{m} = \mathfrak{n}^{\wedge m}$ it coincides with \eqref{EQ: SubadditiveFamilyForULEm} according to \eqref{EQ: KeyResultExteriorOperators}. Note also that in this case $\lambda_{f} \eqcolon \ln \overline{\omega}_{d}(\Xi;\mathcal{K})$ depends only on $d \geq 0$ and does not depend on the metric. Then the (uniform) Lyapunov dimension $\dim_{\operatorname{L}}(\Xi;\mathcal{K})$ of $\Xi$ is defined as the infimum over $d \geq 0$ such that $\overline{\omega}_{d}(\Xi;\mathcal{K}) < 1$. Suppose that $\mathcal{K}$ is an invariant compact for $\vartheta$ and $\Xi$ is the cocycle of differentials of $\vartheta$ over $(\mathcal{K},\vartheta)$. Then (see \cite{ChepyzhovIlyin2004})
	\begin{equation}
		\label{EQ: FractalDimEstViaLD}
		\dim_{\operatorname{F}} \mathcal{K} \leq \dim_{\operatorname{L}}(\Xi;\mathcal{K}),
	\end{equation}
	where $\dim_{\operatorname{F}} \mathcal{K}$ is the fractal (=upper box-counting) dimension of $\mathcal{K}$.
	
	Note that in all of the above situations, we in fact deal with parameterized (by a metric) subadditive families $f$ for which $\lambda_{f}$ does not depend on the parameter. However, our geometric methods to estimate $\lambda_{f}$ do depend on the parameter and the above discussed results of \cite{KawanPogromsky2021,Anikushin2023LyapExp} guaranteeing the existence of metrics providing arbitrarily close estimates justify the approach.
	\qed
\end{remark}

In \cite{Morris2013}, I.D.~Morris established the existence of \textit{extreme points} $q \in \mathcal{Q}$ for which $\lambda_{f}$ can be realized as the limit of $f^{t}(q)/t$ as $t \to +\infty$. Such points can always be taken as recurrent and as typical points over a certain ergodic measure. Note also that \cite{Morris2013} deals with the case of compact $\mathcal{Q}$. However, if there is a compact $\mathcal{K} \subset \mathcal{Q}$ which attracts $\mathcal{Q}$, the result still holds since it is not hard to see that exchanging $\mathcal{Q}$ with $\mathcal{K}$ in \eqref{EQ: UniformUpperGrowthExponentF} does not change the limit. This circumstance is important in infinite dimensions (see \cite{Anikushin2023LyapExp}). Moreover, this justifies computations for attractors in practice, where we work in a localizing neighborhood $\mathcal{U}$ of $\mathcal{K}$ rather than $\mathcal{K}$ itself.

In the topological context (i.e. without appealing to ergodic measures) and for particular subadditive families concerned with Lyapunov dimension related results were previously established by A.~Eden \cite{EdenLocalEstimates1990, Eden1989Thesis}. In his thesis \cite{Eden1989Thesis}, he also formulated a problem concerned with establishing whether there exist extreme points which are stationary or periodic ones. Since then, the existence of such extreme points for particular systems is known as the \textit{Eden conjecture}. In the case of Birkhoff averages (i.e. additive families) a similar conjecture in the ergodic optimization (see the survey of J.~Bochi \cite{Bochi2018}) is attributed to the experimental work of B.R.~Hunt and E.~Ott \cite{HuntOtt1996} who also justified generality of such a behavior. Moreover, the authors made the key observation that extreme periodic orbits must be relatively short (i.e. have a small period). Our main heuristic is also related to this circumstance.

So, there is a \textit{Periodicity Conjecture} stating that generically there exist periodic extreme points of small periods. In the remained part of our work, we develop a machinery for justifying this conjecture for particular systems via short-time computations at least in the case of subadditive families given by \eqref{EQ: SubadditiveFamilySingVal}.

Let $\mathcal{K}$ be a compact invariant subset of $(\mathcal{Q},\vartheta)$. Consider a covering of $\mathcal{K}$ by a finite number of sets $\mathcal{U}^{1}, \ldots, \mathcal{U}^{N}$ so
\begin{equation}
	\label{EQ: CoveringInvariantSet}
	\mathcal{U} := \bigcup_{i=1}^{N} \mathcal{U}^{i} \supset \mathcal{K}.
\end{equation}

With such a covering we associate a directed graph $G = \{ \mathcal{V}, \mathcal{E} \}$ with vertices $\mathcal{V} = \{1,\ldots,N\}$ and edges $\mathcal{E} \subset \mathcal{V} \times \mathcal{V}$ constructed as follows.

We fix positive integers (called \textit{transition times}) $\tau_{j}$, where $j \in \mathcal{V}$. Then we suppose that for any $i, j \in \mathcal{V}$ with $\vartheta^{\tau_{i}}(\mathcal{U}^{i}) \cap \mathcal{U}^{j} \not= \emptyset$ there is the corresponding edge $(i,j) \in \mathcal{E}$. We call any such $G$ a \textit{rough symbolic image} associated with the dynamical system $(\mathcal{K},\vartheta)$, the covering $\mathcal{U}$ and the transition times $\tau = \{\tau_{i}\}^{N}_{i=1}$.

If the above edges are the only edges in $G$, then $G$ is called a \textit{rigorous symbolic image}. However, for purposes of robust computations, $G$ need not to contain only such edges and there may be other edges. One may think that they correspond to intersections up to some precision, namely, $(i,j) \in \mathcal{E}$ if a small neighborhood of $\vartheta^{\tau_{i}}(\mathcal{U}^{i})$ intersects with $\mathcal{U}^{j}$. 

\begin{remark}
	There is an obvious choice of transition times as $\tau_{i}=1$. However, there may be regions with slow dynamics (e.g., near equilibria) which lead to undesirable self-loops or long paths (along unstable manifolds) in the graph. To overcome this, it is useful to choose larger transition times for the corresponding covering subsets. 
	
	Moreover, in the case of continuous-time systems (here transition times can be also taken as positive reals), we are additionally allowed to slow down in regions with fast dynamics by choosing $\vartheta$ to be a sufficiently small mapping of the semiflow. Without any slowing in such regions covering subsets may largely spread around the phase space (due to stretching) also leading to highly irrelevant paths and cycles. Theoretically, one may take covering subsets smaller to minimize such effects, although the allowed smallness is limited by practical feasibility. 
	
	Note also that for particular metrics, taking transition times smaller usually leads to rougher estimates (due to discrepancy in the estimate from \eqref{EQ: DGLULEestimatePoint}). So, one should find an appropriate for the problem balance between choosing appropriate transition times and smallness of covering subsets.
	
	However, the consideration of distinct transition times multiplies the time complexity of optimization algorithms (see below).
	\qed
\end{remark}
\begin{remark}
	There are algorithms for constructing symbolic images contained in the GAIO package (see M.~Dellnitz O.~Junge \cite{DellnitzJunge2002}). Some are also discussed in the monograph of G.S.~Osipenko \cite{Osipenko2006}.
\end{remark}

For what follows, we fix a proper subadditive over $(\mathcal{K},\vartheta)$ family $f = \{f^{t}\}_{t \in \mathbb{T}_{+}}$. For any $i \in \mathcal{V}$, we suppose that $f^{\tau_{i}}$ is also defined and bounded on $\mathcal{U}^{i}$ and introduce the values
\begin{equation}
	\label{EQ: AlphaWeightGraphDefinition}
	\alpha^{+}_{f}(i) \coloneq \max_{q \in \mathcal{U}^{i}} f^{\tau_{i}}(q)	
\end{equation}
For example, if $f$ extends to a proper subadditive family over $(\mathcal{Q},\vartheta)$, then $\alpha^{+}_{f}(i) < +\infty$ is guaranteed by \nameref{DESC: SF3} and finiteness of the covering.

We start with the following proposition.
\begin{proposition}
	\label{PROP: ULEestimatePreparatory}
	For any $q \in \mathcal{K}$ choose a sequence\footnote{Clearly, at least one such a sequence always exists due to the invariance of $\mathcal{K}$.} $i_{0}=i_{0}(q), i_{1}=i_{1}(q),i_{2}=i_{2}(q),\ldots$ of vertices in $G$ such that
	\begin{equation}
		q \in \mathcal{U}^{i_{0}} \text{ and } \vartheta^{\tau_{i_{k}}}(q) \in \mathcal{U}^{i_{k+1}} \text{ for any } k=0, 1, \ldots.
	\end{equation}
	Then
	\begin{equation}
		\label{EQ: UniformLEestimateViaGraphWeights}
		\lambda_{f} \leq \limsup_{K \to +\infty} \sup_{q \in \mathcal{K}}\frac{\sum_{k=0}^{K}\alpha^{+}_{f}(i_{k}(q))}{ \sum_{k=0}^{K} \tau_{i_{k}(q)}}.
	\end{equation}
\end{proposition}
\begin{proof}
	Put 
	\begin{equation}
		\label{EQ: UniformGrowthExponentEstimateLemma1}
		\tau^{+} = \max_{i \in \mathcal{V}}\tau_{i} \text{ and } M \coloneq \sup_{t \in [0,\tau^{+}] \cap \mathbb{T}_{+}}\sup_{q \in \mathcal{K}}f^{t}(q).
	\end{equation}
	Note that $M < \infty$ due to \nameref{DESC: SF1} and \nameref{DESC: SF3}.
	
	Let $q \in \mathcal{K}$ be fixed. Clearly, for any $t>0$ there exists an integer $K(t)=K(t;q) \geq 0$ such that $t = T(t) + r(t)$, where $T(t) = T(t;q) \coloneq \sum_{k=0}^{K(t)}\tau_{j_{k}}$ and $r(t)=r(t;q) \in [0,\tau^{+})$. From \nameref{DESC: SF1}, we have
	\begin{equation}
		\label{EQ: DGLULEestimatePoint}
		\begin{split}
			f^{T(t)}(q) \leq\\
			\leq f^{\tau_{j_{K(t)}}}( \vartheta^{\sum_{k=0}^{ K(t)-1 } \tau_{j_{k}}}(q)) + \ldots + f^{\tau_{j_{1}}}(\vartheta^{\tau_{0}}(q)) + f^{\tau_{j_{0}}}(q) \leq \\
			\leq \sum_{k=0}^{K(t)}\alpha^{+}_{f}(i_{k}).
		\end{split}
	\end{equation}

	From \nameref{DESC: SF1} and \eqref{EQ: UniformGrowthExponentEstimateLemma1} we get $f^{t}(q) \leq f^{T(t)}(q) + M$ for any $t \in \mathbb{T}_{+}$. This and \eqref{EQ: DGLULEestimatePoint} give
	\begin{equation}
		\label{EQ: PreparatoryEstimateTimeWeights3}
		\frac{1}{t}f^{t}(q) \leq  \frac{1}{t} \sum_{k=0}^{K(t)}\alpha^{+}_{f}(i_{k}) + \frac{M}{t}.
	\end{equation}
	
	By taking the supremum over $q \in \mathcal{K}$ in \eqref{EQ: PreparatoryEstimateTimeWeights3} we get that for all $t > \tau^{+}$ there exist $q_{0}=q_{0}(t)$, $K_{0}(t)=K(t;q_{0})$ and $r_{0}(t) = r(t;q_{0})$ such that
	\begin{equation}
		\begin{split}
			\frac{1}{t} \sup_{q \in \mathcal{K}}f^{t}(q) \leq \frac{1}{t}\sum_{k=0}^{K_{0}(t)}\alpha^{+}_{f}(i_{k}(q_{0})) + \frac{M}{t} \leq\\
			\leq \left(1 - \frac{r_{0}(t)}{t}\right) \sup_{q \in \mathcal{K}}\frac{\sum_{k=0}^{K_{0}(t)}\alpha^{+}_{f}(i_{k}(q))}{ \sum_{k=0}^{K_{0}(t)} \tau_{i_{k}(q)}} + \frac{M}{t}.
		\end{split}
	\end{equation}
	Since $K_{0}(t) \to +\infty$ as $t \to +\infty$, taking it to the limit superior as $t \to +\infty$ gives the desired inequality \eqref{EQ: UniformLEestimateViaGraphWeights}. The proof is finished.
\end{proof}

By Proposition \ref{PROP: ULEestimatePreparatory}, our problem is reduced to the study of infinite paths\footnote{More precisely, to the study of \textit{admissible} infinite paths in $G$, i.e. the ones which corresponds to a trajectory of the dynamical system. Usually, there are many infinite paths which are not admissible. One can refine such paths by considering refined coverings and right-resolving presentations as in \cite{FroylandJungeochs2001, Osipenko2006}. In our context, this will also lead to optimization with weights defined on edges rather than vertices. Our general approach still generalizes to this case, but it is questionable whether such refinements may help because the main obstacles in our approach are concerned with controlling many local extrema of singular values. We plan to test it in future works.} in $G$. It is convenient to speak about such paths in terms of topological Markov chains. For this, consider the space $\mathfrak{I}^{+}_{\mathcal{V}} \coloneq \mathcal{V}^{\mathbb{Z}_{+}}$ of one-sided sequences $\mathfrak{i} = (i_{0},i_{1},\ldots)$ of symbols from the alphabet $\mathcal{V}$. Define the \textit{left shift} $\sigma$ as the self-map of $\mathfrak{I}^{+}_{\mathcal{V}}$ given by
\begin{equation}
	\sigma(\mathfrak{i}) \coloneq ( i_{1}, i_{2}, \ldots) \text{ for } \mathfrak{i} = (i_{0},i_{1},i_{2},\ldots) \in \mathfrak{I}^{+}_{\mathcal{V}}.
\end{equation}
One can endow $\mathfrak{I}^{+}_{\mathcal{V}}$ with the natural product topology which is metrizable. Therefore $\sigma$ becomes a continuous self-map of $\mathfrak{I}^{+}_{\mathcal{V}}$. 

Now let $\mathfrak{I}_{G}$ be the subset consisting of $\mathfrak{i} = (i_{0},i_{1},\ldots) \in \mathfrak{I}^{+}_{\mathcal{V}}$ satisfying
\begin{equation}
	(i_{s}, i_{s+1}) \in \mathcal{E} \text{ for any } s = 0,1,2,\ldots.
\end{equation}
Clearly, $\mathfrak{I}_{G}$ is a compact subset which is positively invariant w.r.t. $\sigma$. Let $\sigma_{G}$ be the restriction of $\sigma$ to $\mathfrak{I}_{G}$. Then the dynamical system $(\mathfrak{I}_{G}, \sigma_{G})$ is known as a \textit{topological Markov chain} (or a \textit{subshift of finite type}). Clearly, elements of $\mathfrak{I}_{G}$ represent paths (walks) in the graph $G$ with infinite length and $\sigma$ is a shift along such paths.

For any $\mathfrak{i} = (i_{0},i_{1},\ldots) \in \mathfrak{I}_{G}$, we put $\alpha^{+}_{f}(\mathfrak{i}) \coloneq \alpha^{+}_{f}(i_{0})$, where the latter value is given by \eqref{EQ: AlphaWeightGraphDefinition}. We define the \textit{$\tau$-relative time-$t$ weight} $W^{t}_{\tau}(\mathfrak{i})$ of $\mathfrak{i}$ as
\begin{equation}
	\label{EQ: RelativeTweightMarkovChain}
	W^{t}_{\tau}(\mathfrak{i}) \coloneq \frac{\sum_{s=0}^{t-1}\alpha^{+}_{f}(\sigma^{s}(\mathfrak{i}))}{\sum_{s=0}^{t-1}\tau_{i_{s}}}.
\end{equation}

From Proposition \ref{PROP: ULEestimatePreparatory} we immediately get
\begin{equation}
	\label{EQ: UniformExpEstimateViaGraphPath}
	\lambda_{f} \leq \limsup_{t \to +\infty}\sup_{\mathfrak{i} \in \mathfrak{I}_{G}} W^{t}_{\tau}(\mathfrak{i}).
\end{equation}

Let us show that the limit superior in \eqref{EQ: UniformExpEstimateViaGraphPath} is in fact a limit and it has an expression via simple cycles in the graph. For this, take $t \geq 2$ and let $G_{t}$ be the set of all $t$-tuples $\mathfrak{i} = (i_{0},i_{1},\ldots,i_{t-1}) \in \mathcal{V}^{t}$ representing paths of length $t$ in $G$, i.e. such that $(i_{s},i_{s+1}) \in \mathcal{E}$ for any $s \in \{0,\ldots,t-2\}$. Analogously to \eqref{EQ: RelativeTweightMarkovChain} we define the $\tau$-relative weight of $\mathfrak{i}$ as
\begin{equation}
	\label{EQ: RelativeWeightFinitePath}
	W_{\tau}(\mathfrak{i}) = W_{\tau}(\mathfrak{i}; G) = W_{\tau}(\mathfrak{i}; G; f) \coloneq \frac{\sum_{s=0}^{t-1}\alpha^{+}_{f}(i_{s})}{\sum_{s=0}^{t-1}\tau_{i_{s}}}.
\end{equation}

A $t$-tuple $\mathfrak{i}^{c} = (i^{c}_{0},\ldots,i^{c}_{t-1}) \in G_{t}$ such that $(i^{c}_{t-1},i^{c}_{0}) \in \mathcal{E}$ is called \textit{a cycle} and $t$ is said to be the \textit{period} $T(\mathfrak{i}^{c})$ of $\mathfrak{i}^{c}$. If, in addition, all components $i^{c}_{s}$, where $s \in \{0,\ldots,t-1\}$, are distinct, $\mathfrak{i}^{c}$ is called a \textit{simple cycle}. Clearly, for a simple cycle, $T(\mathfrak{i}^{c})$ does not exceed the number of vertices $|\mathcal{V}|$. 

For a cycle $\mathfrak{i}^{c}$ with period $T(\mathfrak{i}^{c})$, the \textit{$\tau$-relative weight} $W_{\tau}(\mathfrak{i}^{c})$ of $\mathfrak{i}^{c}$ is defined via \eqref{EQ: RelativeWeightFinitePath} with $\mathfrak{i} = \mathfrak{i}^{c}$ and $t = T(\mathfrak{i}^{c})$. It is not hard to see that if $\mathfrak{i}^{c}$ is a repetition of a shorter cycle, then their $\tau$-relative weights coincide.
\begin{proposition}
	\label{PROP: GraphPeriodicMeasure}
	 We have
	\begin{equation}
		\label{EQ: GraphCycleExtremalMeasures}
		\lim_{t \to +\infty}\sup_{\mathfrak{i} \in \mathfrak{I}_{G}} W^{t}_{\tau}(\mathfrak{i}) = \max_{ \mathfrak{i}^{c} } W_{\tau}(\mathfrak{i}^{c}),
	\end{equation}
	where the maximum is taken over all simple cycles $\mathfrak{i}^{c} = (i^{c}_{0},\ldots,i^{c}_{T(\mathfrak{i}^{c})-1})$. 
	
	Moreover, this maximum satisfies the inequality
	\begin{equation}
		\label{EQ: LimitBoundViaFinitePath}
		\max_{ \mathfrak{i}^{c} } W_{\tau}(\mathfrak{i}^{c}) \leq \sup_{\mathfrak{i} \in G_{t}} W_{\tau}(\mathfrak{i}) \text{ for any } t \geq 2.
	\end{equation}
\end{proposition}
\begin{proof}
	Clearly, the right-hand side of \eqref{EQ: GraphCycleExtremalMeasures} does not exceed the left one (more precisely, the limit inferior) because repetitions of cycles provide infinite paths in the graph with $\tau$-relative time-$t$ weights being equal to $W_{\tau}(\mathfrak{i}^{c})$ if $t$ is a multiple of $T(\mathfrak{i}^{c})$. For general $t$, a remainder appears which decays as $t \to +\infty$ (analogously to the argument below).
	
	For the contrary, note that, by finiteness of $\mathcal{V}$, for any $t$ there exists $\mathfrak{i}=(i_{0},\ldots,i_{t-1}) \in G_{t}$ such that
	\begin{equation}
		\sup_{\mathfrak{j} \in \mathfrak{I}_{G}} W^{t}_{\tau}(\mathfrak{j}) = W_{\tau}(\mathfrak{i}) =  \frac{\sum_{s=0}^{t-1}\alpha^{+}_{f}(i_{s})}{\sum_{s=0}^{t-1}\tau_{i_{s}}}.
	\end{equation}
	It is not hard to see (by induction) that $\mathfrak{i}$ can be decomposed into an algebraic sum of, say $K$, simple cycles $\mathfrak{i}^{c}_{1},\ldots,\mathfrak{i}^{c}_{K}$ plus a remainder $\mathfrak{i}^{R} \in G_{r}$ for some $r \in \{0, 1, \ldots, |\mathcal{E}| - 1\}$. Thus, $t = \sum_{k=1}^{K}T(\mathfrak{i}^{c}_{k}) + r$ and
	\begin{equation}
		\label{EQ: AlgebraicWalkDecomposition}
		\begin{split}
			&\sum_{s=0}^{t-1}\alpha^{+}_{f}(j_{s}) = \sum_{k=1}^{N} \sum_{s=0}^{T(\mathfrak{i}^{c}_{k})}\alpha^{+}_{f}(i^{c}_{k,s}) + \sum_{s = 0}^{r} \alpha^{+}_{f}(i^{R}_{s}), \\
			&\sum_{s=0}^{t-1}\tau_{j_{s}} = \sum_{k=1}^{N} \sum_{s=0}^{T(\mathfrak{i}^{c}_{k})}\tau_{i^{c}_{k,s}} + \sum_{s = 0}^{r} \tau_{i^{R}_{s}}
		\end{split}
	\end{equation}
	where $\mathfrak{i}^{c}_{k} = ( i^{c}_{k,0},\ldots,i^{c}_{k,T(\mathfrak{i}^{c}_{k})-1})$ and $\mathfrak{i}^{R} = (i^{R}_{0},\ldots,i^{R}_{r-1})$.
	
	Let $W^{+}$ be the right-hand side of \eqref{EQ: GraphCycleExtremalMeasures}. Then
	\begin{equation}
		\sum_{s=0}^{T(\mathfrak{i}^{c}_{k})}\alpha^{+}_{f}(i^{c}_{k,s}) \leq W^{+} \cdot \sum_{s=0}^{T(\mathfrak{i}^{c}_{k})}\tau_{i^{c}_{k,s}}.
	\end{equation}
	for any $k \in \{1,\ldots,K\}$. From this and \eqref{EQ: AlgebraicWalkDecomposition}, we get
	\begin{equation}
		\sup_{\mathfrak{j} \in \mathfrak{I}_{G}} W^{t}_{\tau}(\mathfrak{j}) \leq W^{+}(1 + o(1))
	\end{equation}
	as $t \to +\infty$. This shows the converse inequality and proves \eqref{EQ: GraphCycleExtremalMeasures}.
	
	To show \eqref{EQ: LimitBoundViaFinitePath}, note that we may assume $\alpha^{+}_{f}(i) > 0$ for any $i \in \mathcal{V}$ otherwise exchanging weights $\alpha^{+}_{f}(i)$ with $\alpha^{+}_{f}(i) + \varkappa \tau_{i}$, where $\varkappa$ is any number satisfying $\varkappa > -\alpha^{+}_{f}(i)$ for any $i \in \mathcal{V}$. This exchange only adds $\varkappa$ to all the relative weights.
	
	It is not hard to see that for any $x,a \geq 0$ and $y,b>0$ we have
	\begin{equation}
		\label{EQ: SumQuotientViaMaxQuotients}
		\frac{x+a}{y+b} \leq \max\left\{ \frac{x}{y}, \frac{a}{b} \right\}.
	\end{equation}
	
	Now let $t \geq 2$ be fixed and take $\mathfrak{j}=(j_{0},j_{1},\ldots) \in \mathfrak{I}_{G}$. For any nonnegative integer $k$ consider $\mathfrak{i}^{k} = (j_{k t}, j_{kt + 1}, \ldots, j_{(k+1)t - 1}) \in G_{t}$. From the assumed positivity of $\alpha^{+}_{f}$ and \eqref{EQ: SumQuotientViaMaxQuotients}, we have
	\begin{equation}
		W^{K t}_{\tau}(\mathfrak{j}) \leq \max_{0 \leq k \leq K-1 } W^{t}_{\tau}(\mathfrak{i}^{k}) \leq \max_{\mathfrak{i} \in G_{t}} W_{\tau}(\mathfrak{i})
	\end{equation}
	for any $K=1,2,\ldots$. Taking the limit as $K \to +\infty$ shows \eqref{EQ: LimitBoundViaFinitePath} due to \eqref{EQ: GraphCycleExtremalMeasures}. The proof is finished.
\end{proof}

\begin{remark}
	\label{REM: ErgodicOptimization}
	Proposition \ref{PROP: GraphPeriodicMeasure} can be viewed as a discrete version of the Periodicity Conjecture. Its validness also reflects the validity (by elementary arguments) of many other discrete versions of variational relations in dynamics. For example, a discrete version of the variational principle for topological entropy is also valid (see W.~Parry \cite{Parry1964}). 
	
	Having in mind that paths in symbolic images correspond to pseudo-trajectories, this also motivates validity of the continuous versions if the chain-recurrent set is in a sense well-behaved w.r.t. to its approximations via symbolic images, although we do not know works treating the problem from this perspective. \qed
\end{remark}

Symbolic images of our interest are related to chaotic attractors. In this case, there are too many simple cycles which cannot be enumerated in a reasonable time. For equal transition times, computing the maximum in \eqref{EQ: GraphCycleExtremalMeasures} is known as the \textit{Minimum Mean Cycle Problem} and, for general transition times, this is the \textit{Minimum Cost-to-Time Ratio Cycle Problem} (see Section 5.7 in R.K.~Ahuja, T.L.~Magnanti and J.B.~Orlin \cite{AhujaMagnantiOrlinNF1993}). For the former, there is an algorithm with computational complexity $O(|\mathcal{E}| \cdot |\mathcal{V}|)$ (both in time and memory). For the latter, one uses consecutive applications of the first algorithm and this adds a multiplier to the complexity arising from binary search (see below). 

Our approach will be based on heuristics concerned with the estimate from \eqref{EQ: LimitBoundViaFinitePath} which make computations much faster. Namely, experimental observations show that for sufficiently large values of $t$, the optimal path of length $t$ algebraically is a multiple repetition of a short simple cycle plus a small remainder\footnote{This is an empirical property of symbolic images inherited from the Periodicity Conjecture.}. It is likely that this cycle is the maximal one, but rigorously we can get only the estimate through the remainder. In the case of equal transition times, the maximal path of length $t$ can be found in $O(t \cdot d \cdot |\mathcal{V}|)$, where $d$ is the mean degree of vertices (so $d \cdot |\mathcal{V}| = |\mathcal{E}|$) which is relatively small (say, $\leq 20$) for appropriate symbolic images. For computations related to Iterative Nonlinear Programming discussed in Section \ref{SEC: OptimizatiobViaINP}, $t$ may be taken much smaller than $|\mathcal{V}|$ (say $1000$ versus $10^{6}$).

Let us describe the algorithm. For $k \geq 2$ and $i \in \mathcal{V}$, let $W(k,i)$ be the maximum of $\sum_{s=0}^{k-1} \alpha^{+}_{f}(i_{s})$ taken over all paths $\mathfrak{i} = (i_{0},i_{1},\ldots, i_{k-1}) \in G_{k}$ with $i_{0} = i$. For $k=1$, we put $W(1,i) \coloneq \alpha^{+}_{f}(i)$. It is not hard to see that
\begin{equation}
	\label{EQ: RecurrenceMathPathEqualTimes}
	W(k+1, i) = \alpha^{+}_{f}(i) + \max_{j \colon (i,j) \in \mathcal{E}} W(k,j)
\end{equation}
for any $i \in \mathcal{V}$ and $k \geq 1$. Moreover, we can keep $J(k+1,i)$ as $j$ on which the maximum is achieved in the above formula (for $J(1,i)$ we take any $j$ with $(i,j) \in \mathcal{E}$). By maximizing $W(t,i)$ over $i \in \mathcal{V}$, we obtain a maximal path of length $t$ with the mentioned algorithmic complexity.

\begin{remark}
	Note that \eqref{EQ: RecurrenceMathPathEqualTimes} requires only remembering the previous step to proceed further. If we need only the maximal weight value, this observation allows to lower the memory complexity to $O(\mathcal{V})$. This is significant to obtain more precise estimates from \eqref{EQ: LimitBoundViaFinitePath}.
\end{remark}

For general transition times, the resolution of the problem is concerned with consecutive applications of the above search by appropriate changes of weights similarly to the Minimum Cost-to-Time Ratio Cycle Problem (see Section 5.7 in \cite{AhujaMagnantiOrlinNF1993}). Namely, for $\varkappa \in \mathbb{R}$ define $\alpha_{\varkappa}(i) \coloneq \alpha^{+}_{f}(i) - \varkappa \tau_{i}$ for $i \in \mathcal{V}$. Let $\mathfrak{i}=(i_{0},\ldots,i_{t-1}) \in G_{t}$ be the maximal path w.r.t. weights given by $\alpha_{\varkappa}$. There may be three cases:
\begin{enumerate}
	\item $\sum_{s=0}^{t-1} \alpha_{\varkappa}(i_{s}) > 0$ or, equivalently, $\sum_{s=0}^{t-1} \alpha^{+}_{f}(i_{s}) > \varkappa \sum_{s=0}^{t-1}\tau_{i_{s}}$. Consequently, we get the lower bound
	\begin{equation}
		\sup_{\mathfrak{j} \in G_{t}} W_{\tau}(\mathfrak{j}) \geq W_{\tau}(\mathfrak{i}) = \frac{\sum_{s=0}^{t-1} \alpha^{+}_{f}(i_{s})}{\sum_{s=0}^{t-1}\tau_{i_{s}}} > \varkappa.
	\end{equation}
	\item $\sum_{s=0}^{t-1} \alpha_{\varkappa}(i_{s}) < 0$, where due to the maximality of $\mathfrak{i}$, we must have $\sum_{s=0}^{t-1} \alpha_{\varkappa}(j_{s}) \leq \sum_{s=0}^{t-1} \alpha_{\varkappa}(i_{s}) < 0$ and, consequently, $\sum_{s=0}^{t-1} \alpha^{+}_{f}(j_{s}) < \varkappa \sum_{s=0}^{t-1}\tau_{j_{s}}$ for any $\mathfrak{j}=(j_{0},\ldots,j_{t-1}) \in G_{t}$. Thus, the upper bound is obtained as
	\begin{equation}
		\sup_{\mathfrak{j} \in G_{t}} W_{\tau}(\mathfrak{j}) < \varkappa.
	\end{equation}
	\item $\sum_{s=0}^{t-1} \alpha_{\varkappa}(i_{s}) = 0$ in which case it is easy to see that
	\begin{equation}
		\sup_{\mathfrak{j} \in G_{t}} W_{\tau}(\mathfrak{j}) = W_{\tau}(\mathfrak{i}) = \varkappa.
	\end{equation}
\end{enumerate}
Clearly, $\sup_{\mathfrak{j} \in G_{t}} W_{\tau}(\mathfrak{j})$ can be bounded in terms of minima and maxima of $\alpha^{+}_{f}(i)$ and $\tau_{i}$ over $i \in \mathcal{V}$. Then we can use binary search to sharp the bounds based on the above observations. If the bounding interval is sufficiently small, then any $\mathfrak{i}$ computed for $\varkappa$ within it provides a good approximation (or likely the answer) to the optimization of relative weights.

So, \eqref{EQ: LimitBoundViaFinitePath} delivers a nice approach to get estimates of the limit in \eqref{EQ: GraphCycleExtremalMeasures} for particular metrics and to extract relevant cycles for optimization in families of metrics. We refer to Section \ref{SEC: OptimizatiobViaINP} for further discussions.

Recall that a vertex $i \in \mathcal{V}$ of $G$ is called a \textit{sink} (resp. \textit{source}) if there is no $j \in \mathcal{V}$ such that $(i,j) \in \mathcal{E}$ (resp. $(j,i) \in \mathcal{E}$).
\begin{corollary}
	\label{COR: ReducedGraphEstCorollary}
	Let $G'$ be a subgraph of $G$ obtained by iteratively removing sinks and sources from $G$. In other words, $G'$ is the maximal subgraph of $G$ for which any finite path $(i_{0},i_{1},\ldots,i_{n})$ has a two-sided infinite continuation $(\ldots,i_{-1},i_{0},i_{1},\ldots)$. Then the limit in \eqref{EQ: GraphCycleExtremalMeasures} is the same for $G$ exchanged with $G'$ and the covering elements corresponding to the vertices in $G'$ still covers $\mathcal{K}$.
\end{corollary}
\begin{proof}
	Indeed, $G'$ and $G$ have the same set of infinite paths (in particular, simple cycles). So, the conclusion is obvious.
\end{proof}

Note that the estimate \eqref{EQ: LimitBoundViaFinitePath} becomes exact as $t \to +\infty$ since there are finitely many paths without an infinite prolongation in the future (in the graph $G'$ there are no such paths at all) and, consequently, they do not affect the asymptotics. However, since the time-$t$ estimate is exact up to $O(1/t)$ (see Tab.~\ref{TAB: HenonEstimatesPathLengths}), obtaining estimates with precision above $6-8$ digits even for low-dimensional chaotic systems demands the use of supercomputers.

Let us summarize considerations from the present section into the following \textit{Relative Weights Optimization} algorithm evaluating the right-hand of \eqref{EQ: UniformExpEstimateViaGraphPath}.
\begin{description}[before=\let\makelabel\descriptionlabel]
	\item[\textbf{(RWO1)}\refstepcounter{desccount}\label{DESC: GOA1}] Localize a compact invariant subset $\mathcal{K}$ by a set $\mathcal{U}_{0} \supset \mathcal{K}$;
	\item[\textbf{(RWO2)}\refstepcounter{desccount}\label{DESC: GOA2}] Consider a \textit{preparatory covering} $\mathcal{U}_{0} \subset  \bigcup_{i=1}^{N_{0}} \mathcal{Q}^{i}$ of $\mathcal{U}_{0}$ by small subsets $\mathcal{Q}^{i} \subset \mathcal{Q}$;
	\item[\textbf{(RWO3)}\refstepcounter{desccount}\label{DESC: GOA3}] Put $\mathcal{V}_{0} = \{1,\ldots,N_{0}\}$, choose transition times $\{ \tau_{i} \}_{i \in \mathcal{V}_{0}}$ and compute a \textit{preparatory graph} $G_{0} = (\mathcal{V}_{0},\mathcal{E}_{0})$, where the set of edges $\mathcal{E}_{0}$ at least contains such pairs $(i,j) \in \mathcal{V}_{0} \times \mathcal{V}_{0}$ for which $\vartheta^{\tau_{i}}(\mathcal{Q}^{i}) \cap \mathcal{Q}^{j} \not= \emptyset$;
	\item[\textbf{(RWO4)}\refstepcounter{desccount}\label{DESC: GOA4}] Reduce the graph $G_{0}$ to the \textit{refined graph} (a subgraph) $G=(\mathcal{V},\mathcal{E})$ by removing sinks and sources (see above Corollary \ref{COR: ReducedGraphEstCorollary}). Elements $\mathcal{U}^{1},\ldots,\mathcal{U}^{N}$ of the preparatory covering which correspond to the remained vertices (so $\mathcal{V} = \{1,\ldots,N\}$ and $N=|\mathcal{V}|$) provide the \textit{refined cover} of $\mathcal{K}$. Thus, as in \eqref{EQ: CoveringInvariantSet}, we have $\mathcal{U} = \bigcup_{i=1}^{N}\mathcal{U}^{i} \supset \mathcal{K}$;
	\item[\textbf{(RWO5)}\refstepcounter{desccount}\label{DESC: GOA5}] For a given proper subadditive family $f=\{f^{t}\}_{t \in \mathbb{T}_{+}}$ compute $\alpha^{+}_{f}(i)$ as in \eqref{EQ: AlphaWeightGraphDefinition};
	\item[\textbf{(RWO6)}\refstepcounter{desccount}\label{DESC: GOA6}] Find a path in the refined graph $G$ with a given length $t$ and maximal $\tau$-relative weight (see around \eqref{EQ: RecurrenceMathPathEqualTimes}) which provides an upper estimate for $\lambda_{f}$ according to \eqref{EQ: UniformExpEstimateViaGraphPath}, \eqref{EQ: GraphCycleExtremalMeasures} and \eqref{EQ: LimitBoundViaFinitePath}.
\end{description}

We finalize this section by discussing a reduction of the graph $G$ by capturing some symmetries of the system. Remind that an automorphism of $G$ is a bijection $S \colon \mathcal{V} \to \mathcal{V}$ of vertices which preserves the edges, i.e. for any $(i,j) \in \mathcal{E}$ we have $(Si,Sj) \in \mathcal{E}$. All automorphisms of $G$ form a finite group $\operatorname{Aut}(G)$. Let $\mathbb{S}$ be a subgroup of $\operatorname{Aut}(G)$. Then the quotient graph $G_{\mathbb{S}}$ is the graph with vertices $\mathcal{V}_{\mathbb{S}}$ given by the orbits of $\mathbb{S}$ in $\mathcal{V}$ and the set of its edges $\mathcal{E}_{\mathbb{S}}$ consists of orbit pairs $(i_{\mathbb{S}}, j_{\mathbb{S}} )$ such that there exist $i \in i_{\mathbb{S}}$ and $j \in j_{\mathbb{S}}$ with $(i,j) \in \mathcal{E}$.

\begin{proposition}
	\label{PROP: MaximalRelativeWeightViaQuotientGraph}
	Suppose $\alpha^{+}_{f}$ and $\tau_{i}$ are invariant w.r.t. $\mathbb{S}$, i.e.
	\begin{equation}
		\alpha^{+}_{f}(i) = \alpha^{+}_{f}(Si) \text{ and } \tau_{i} = \tau_{Si} \text{ for any } i \in \mathcal{V} \text{ and } S \in \mathbb{S}.
	\end{equation}
	Then $\alpha^{+}_{f}$ and $\tau_{i}$ correctly define weights and transition times in the quotient graph $G_{\mathbb{S}}$ and (see \eqref{EQ: RelativeWeightFinitePath})
	\begin{equation}
		\label{EQ: CycleRelWeigthMaximaViaQuotient}
		\max_{ \mathfrak{i}^{c} } W_{\tau}(\mathfrak{i}^{c};G) = \max_{ \mathfrak{i}^{c} } W_{\tau}(\mathfrak{i}^{c};G_{\mathbb{S}}),
	\end{equation}
	where the first maximum is taken over simple cycles in $G$ and the second maximum is taken over simple cycles in $G_{\mathbb{S}}$.
\end{proposition}
\begin{proof}
	By definition of $G_{\mathbb{S}}$ and the invariance of functions, any infinite path in $G$ (i.e. the trajectory of $(\mathfrak{I}_{G},\sigma_{G})$) is projected to an infinite path in $G_{\mathbb{S}}$ (i.e. the trajectory of $(\mathfrak{I}_{G_{\mathbb{S}}},\sigma_{G_{\mathbb{S}}})$) with the same $\tau$-relative time-$t$ weights. Conversely, any infinite path in $G_{\mathbb{S}}$ can be lifted to an infinite path in $G$ with the same $\tau$-relative time-$t$ weights (by induction). Then \eqref{EQ: CycleRelWeigthMaximaViaQuotient} follows from Proposition \ref{PROP: GraphPeriodicMeasure}. The proof is finished.
\end{proof}

In virtue of Proposition \ref{PROP: MaximalRelativeWeightViaQuotientGraph}, one may involve discrete symmetries presented in the system in order to reduce the overall computations. For example, in the case of $f$ given by \eqref{EQ: SubadditiveFamilyForULEm} with some metric $\mathfrak{m}$ on $T^{\wedge m}\mathcal{Q}$, this affects the choice of $\mathcal{U}^{i}$, $\tau_{i}$, $\mathfrak{m}$ and computing intersections to respect the symmetries. 

Let us illustrate this in the case when the phase space $\mathcal{Q}$ is a subset of $\mathbb{R}^{n}$ and there exists a subgroup $\mathbb{S}$ of invertible linear transformations (matrices) in $\mathbb{R}^{n}$ commuting with $\vartheta$ on $\mathcal{Q}$, i.e. $\vartheta \circ S = S \circ \vartheta$ for any $S \in \mathbb{S}$. Then for $\mathcal{U}^{i}$ and $\tau_{i}$ to respect the symmetry we must have for any $i \in \mathcal{E}$ and $S \in \mathbb{S}$ that there exists a unique $j$ such that $S\mathcal{U}^{i} = \mathcal{U}^{j}$ and $\tau_{i} = \tau_{j}$. This defines transformations (we keep the same notations) in the graph via $Si \coloneq j$. If for any edge $(i,j) \in \mathcal{E}$ we have $(Si,Sj) \in \mathcal{E}$, then $\mathbb{S}$ is a subgroup of $\operatorname{Aut}(G)$. 

Moreover, we can construct a metric respecting the symmetry from any metric induced by inner products. For simplicity suppose that $m=1$ and we a given with a metric defined via a family of symmetric positive-definite matrices $P_{q}$, where $q \in \mathcal{Q}$, as $\langle \cdot, P_{q} \cdot \rangle_{Euc}$, where $\langle \cdot, \cdot \rangle_{Euc}$ is the standard Euclidean inner product. Then the symmetric positive-definite matrices
\begin{equation}
	\label{EQ: AveragingMetricToRespectS}
	\overline{P}_{q} \coloneq \frac{1}{|\mathbb{S}|} \sum_{S \in \mathbb{S}} S^{T}P_{Sq} S
\end{equation}
define an $\mathbb{S}$-invariant metric over $\mathcal{Q}$.

So, if the symmetries are respected, it is not hard to see that we must have $\alpha^{+}_{f}(i) = \alpha^{+}_{f}(Si)$ for any $i \in \mathcal{V}$. Consequently, Proposition \ref{PROP: MaximalRelativeWeightViaQuotientGraph} is applicable to the symbolic image.

In Section \ref{SEC: RabinovichLargestULE}, the symmetry reduction is illustrated by means of a system with a $\mathbb{Z}_{2}$-symmetry.
\section{Optimization in smooth families of metrics via Iterative Nonlinear Programming}
\label{SEC: OptimizatiobViaINP}

In this section, we will show how the Relative Weights Optimization algorithm (see Section \ref{SEC: RobustAlgoEstimateULE}) can be synthesized with nonlinear programming to effectively resolve the optimization problem for growth exponents in smooth families of metrics.

Let us illustrate the idea by means of an abstract example. Suppose we have a family of real-valued functions $F_{\omega}$ defined on a set $\mathcal{X}$ (finite or infinite) and depending on some $r$-dimensional parameter $\omega \in \Omega \subset \mathbb{R}^{r}$ from a domain $\Omega$. We want to minimize in $\omega$ the maximum value of $F_{\omega}(x)$ over $x \in \mathcal{X}$, i.e. to resolve the optimization problem
\begin{equation}
	\label{EQ: MaximumMinimizationGeneral}
	\max_{x \in \mathcal{X}} F_{\omega}(x) \to \min_{\omega \in \Omega}.
\end{equation}

Even for nice $\mathcal{X}$ and $F_{\omega}$ this is a nonsmooth problem (e.g., take $F_{\omega}(x) = \omega x$ with $\mathcal{X} = [0,1]$ and $\Omega = [-1,1]$), but in some cases one can explore its convexity and develop subgradient optimization algorithms as it is done in \cite{KawanHafsteinGiesl2021, LouzeiroetAlAdaptedMetricsComp2022} for some cases of our interest. However, these results place essential restrictions which are unnatural for practice.

Here we suggest a heuristic algorithm to resolve general problems of the form \eqref{EQ: MaximumMinimizationGeneral} which in its abstract form can be described as follows. Choose an initial parameter $\omega_{0} \in \Omega$; a rectangular (small) neighborhood $\Omega_{bnd}$ of $0$ in $\mathbb{R}^{r}$; a set of $n_{x} \geq 1$ \textit{reference points} points $x_{1}, \ldots, x_{n_{x}} \in \mathcal{X}$ with $F_{\omega_{0}}(x_{1}) = \max_{x \in \mathcal{X}} F_{\omega_{0}}(x)$ and a real number $F_{bnd} < F_{\omega_{0}}(x_{1})$. Then resolve the problem
\begin{equation}
	\begin{split}
		&\max_{1 \leq j \leq n_{x}} F_{\omega}(x_{j}) \to \min_{\omega \in \omega_{0} + \Omega_{bnd}},\\
		&\max_{1 \leq j \leq n_{x}} F_{\omega}(x_{j}) \in [F_{bnd}, F_{\omega_{0}}(x_{1})]
	\end{split}
\end{equation}
which is equivalent to the problem of nonlinear programming with bounds and inequality constraints via introducing an extended parameter $\omega^{r+1} \in \mathbb{R}$. Namely, by putting
\begin{equation}
	\begin{split}
		&\overline{\omega} = (\omega, \omega^{r+1}) = (\omega^{1},\ldots,\omega^{r+1}) \in \mathbb{R}^{r+1},\\
		&\overline{\Omega} = (\omega_{0} + \Omega_{bnd} ) \times [F_{bnd}, F_{\omega_{0}}(x_{1})],
	\end{split}
\end{equation}
we get (note that we denote components of vectors using upper indices)
\begin{equation}
	\label{EQ: NonlinearProgrammingGeneral}
	\begin{split}
		&\omega^{r+1} \to \min_{\overline{\omega} \in \overline{\Omega}},\\
		&F_{\omega}(x_{j}) \leq \omega^{r+1}, \text{ for } j \in \{1,\ldots, n_{x}\}.
	\end{split}
\end{equation}
Under certain regularity of $F_{\omega}(x_{j})$ in $\omega$, the above problem can be solved by the Sequential Quadratic Programming (see below) leading to a parameter $\omega_{1} \in \omega_{0} + \Omega_{bnd}$, where at least a local minimum is achieved. Then we repeat this procedure starting from $\omega_{1}$ therefore resulting in $\omega_{2}$ on the second iteration and so on. We call this process the \textit{Iterative Nonlinear Programming}. Depending on the problem, one can also complement \eqref{EQ: NonlinearProgrammingGeneral} with some additional constraints guiding the optimization (see Section \ref{SEC: ComputationalTips}).

Clearly, effective and efficient applications of the above approach demand a clever use of reference points at each iteration to catch the evolution of shapes as $\omega$ varies (see Remark \ref{REM: ComputingMaximaShapeChanges}). For example, if we know that each $F_{\omega}(x)$ is a polynomial (of fixed degree) of one variable $x$, we can control the evolution just by taking reference points given by the roots of derivatives of polynomials and boundary points of $\mathcal{X}$. For general families, the situation is complicated by the evolving of shapes with parameter varyings. Here appropriate choices\footnote{It is not a priori clear that taking the bounds smaller is always helpful. Sometimes optimizing the objective function by a relatively large value on each iteration, although takes large time and do not lead to immediate improvements on next iterations, results in the better long term optimization, by the way, collecting smaller number of reference cycles.} of $\Omega_{bnd}$ and $F_{bnd}$ may also help to control the evolution. A general recipe may be to take local maxima and some reference points from previous iterations to the next one. Of course, more complex shape changes of $F_{\omega}$ (in practice, this is reflected in the dimension of the parameter $\omega$) require more reference points and make the optimization procedure more complicated.

\begin{remark}
	A significant feature of \eqref{EQ: NonlinearProgrammingGeneral} is that it allows to do precomputations before solving the problem. For example, if the computation of $F_{\omega}$ involves polynomials with parameters only in coefficients, one can precompute monomials for each $x_{j}$. Another case is when $F_{\omega}$ is a correction to some (independent of $\omega$) function, e.g. arising from certain interpolation or a previous optimization that got stuck at minimum. This significantly reduces the time complexity of implementations of the overall algorithm.
\end{remark}

Now we are going to describe the optimization problem in our case. Recall here a symbolic image $G=\{ \mathcal{V}, \mathcal{E} \}$ associated with the covering $\mathcal{U}$ and transition times $\tau = \{ \tau_{i} \}^{N}_{i=1}$ as in Section \ref{SEC: RobustAlgoEstimateULE}. We suppose that there is a family of metrics $\mathfrak{n} = \mathfrak{n}(\omega)$ on $T^{\wedge m}\mathcal{M}$ (at least) over $\mathcal{U} \cup \bigcup_{i \in \mathcal{V}} \vartheta^{\tau_{i}}(\mathcal{U}^{i})$ given by inner products smoothly depending on some $r$-dimensional parameter $\omega \in \Omega \subset \mathbb{R}^{r}$. For each $\omega \in \Omega$ and any $d \geq 0$ we associate the family of subadditive functions $f_{\omega} = \{ f^{t}_{\omega} \}_{t \in \mathbb{T}_{+}}$ given by \eqref{EQ: SubadditiveFamilySingVal} with $\mathfrak{n} = \mathfrak{n}(\omega)$, i.e. $f^{t}_{\omega}(q) = \omega^{(\mathfrak{n}(\omega))}_{d}(\Xi^{t}(q,\cdot))$, where $\omega^{(\mathfrak{n}(\omega))}_{d}$ is the function of singular values in the metric $\mathfrak{n}(\omega)$. For a cycle $\mathfrak{i}^{c}$ in $G$, recall the $\tau$-relative weight $W_{\tau}(\mathfrak{i}^{c};\omega) \coloneq W_{\tau}(\mathfrak{i}^{c};G;f_{\omega})$ of $\mathfrak{i}^{c}$ (see \eqref{EQ: RelativeWeightFinitePath}) computed w.r.t. the family $f_{\omega}$. 

Based on Proposition \ref{PROP: GraphPeriodicMeasure}, we want to resolve the optimization problem
\begin{equation}
	\label{EQ: BasicOptimizationProblemOverCycles}
	\max_{\mathfrak{i}^{c}}W_{\tau}(\mathfrak{i}^{c}; \omega) \to \min_{\omega \in \Omega},
\end{equation}
where the maximum is taken over all simple cycles, or the analogous problem for the quotient graph if some symmetries are presented (see Proposition \ref{PROP: MaximalRelativeWeightViaQuotientGraph}).

Analogously to the procedure above, a step of the algorithm is concerned with choosing an initial parameter $\omega_{0} \in \Omega$; a neighborhood $\Omega_{bnd}$ of $0$ in $\mathbb{R}^{r}$; a small value $W_{h} > 0$; a set of $n_{c} \geq 1$ \textit{reference simple cycles} $\mathfrak{i}^{c}_{1}, \ldots, \mathfrak{i}^{c}_{n_{c}}$ and a set of \textit{reference cycle points} $q_{k, j, s}$ defined for any $k \in \{1,\ldots,n_{c}\}$, $j \in \{1,\ldots, j_{k}\}$ (with some $j_{k} \geq 1$) and $s \in \{0,\ldots, T(\mathfrak{i}^{c}_{k})-1\}$ and satisfying
\begin{equation}
	\label{EQ: INPReferencePoints}
	q_{k, j, s} \in \mathcal{U}^{i_{k, s}}, \text{ where } \mathfrak{i}^{c}_{k} = (i_{k,0},\ldots, i_{k,T(\mathfrak{i}^{c}_{k})-1}).
\end{equation}
For a fixed $k$, we say that $q_{k, j, s}$ is a family over $j$ of reference points corresponding to $\mathfrak{i}^{c}_{k}$.

With such data we associate the \textit{relative reference weights}
\begin{equation}
	\label{EQ: RelWeightRefPoints}
	\widetilde{W}_{k,j}(\omega) \coloneq \frac{\sum_{s=0}^{T(\mathfrak{i}^{c}_{k})-1} \ln \omega^{(\mathfrak{n}(\omega))}_{d}(\Xi^{\tau_{i_{k,s}}}(q_{k,j,s}, \cdot))}{\sum_{s=0}^{T(\mathfrak{i}^{c}_{k})-1} \tau_{i_{k,s}}}
\end{equation}
and, analogously to \eqref{EQ: NonlinearProgrammingGeneral}, the corresponding nonlinear programming problem
\begin{equation}
	\label{EQ: NonlinearProgrammingCycles}
	\begin{split}
		&\omega^{r+1} \to \min_{\overline{\omega} \in \overline{\Omega}},\\
		&\widetilde{W}_{k,j}(\omega) \leq \omega^{r+1}, \text{ for } k \in \{1,\ldots,n_{c}\} \text{ and } j \in \{1,\ldots, j_{k}\},
	\end{split}
\end{equation}
where 
\begin{equation}
	\begin{split}
		&\overline{\omega} = (\omega, \omega^{r+1}) = (\omega^{1},\ldots,\omega^{r+1}) \in \mathbb{R}^{r+1},\\
		&\overline{\Omega} = (\omega_{0} + \Omega_{bnd} ) \times [W^{+}(\omega_{0}) - W_{h}, W^{+}(\omega_{0})]
	\end{split}
\end{equation}
with
\begin{equation}
	W^{+}(\omega_{0}) \coloneq \max_{k,j} \widetilde{W}_{k,j}(\omega_{0}).
\end{equation}

Before we start resolving the problem \eqref{EQ: NonlinearProgrammingCycles} on each step, there are preparatory computations (which we call \textit{updating procedure}) concerned with updating reference cycles and points as follows.
\begin{description}
	\item[\textbf{(UP1)}] Find a path\footnote{Remind that algorithms for finding maximizing paths are discussed around \eqref{EQ: RecurrenceMathPathEqualTimes}.} $\mathfrak{i} \in G_{t}$ of length $t$ maximizing $W_{\tau}(\mathfrak{i};\mathfrak{n}(\omega_{0}))$ for a fixed sufficiently large $t$ (e.g., $t=1000$) and extract from $\mathfrak{i}$ a short cycle $\mathfrak{i}^{c}$ of maximal multiplicity. Keep $\mathfrak{i}^{c}$ as a reference cycle;
	\item[\textbf{(UP2)}] For any reference cycle $\mathfrak{i}^{c}$, let $\mathfrak{i}^{c} = (i_{0}, i_{1}, \ldots,i_{t-1}) \in G_{t}$. Assuming each $\mathcal{U}^{i}$ is compact, there exist points $q_{0},\ldots,q_{t-1}$ such that
	\begin{equation}
		\label{EQ: UP2SingValue}
		\alpha^{+}_{f_{\omega_{0}}}(i_{s}) = \ln \omega^{(\mathfrak{n}(\omega_{0}))}_{d} \left(\Xi^{\tau_{i_{s}}}(q_{s},\cdot) \right) \text{ for } s \in \{0,\ldots, t-1\};
	\end{equation}
	Keep $q_{0},\ldots,q_{t-1}$ as reference points corresponding to $\mathfrak{i}^{c}$.
\end{description}

It may be also convenient to set limits on the maximal number of reference cycles and maximal size $N_{ref} \geq j_{k}$ of the families of reference points related to a given reference cycle $\mathfrak{i}^{c}_{k}$ (see Remark \ref{REM: ComputingMaximaShapeChanges}). In this case, if adding a new cycle or reference points makes us go beyond the limits, we replace the oldest ones by the new ones. Moreover, some old cycles with much smaller relative weights than the maximal one can be taken to be inactive for the current iteration.

There are two crucial experimental observations making the above approach effective (and efficient for appropriate programming realizations). Namely, 
\begin{enumerate}
	\item[1)] In \eqref{EQ: BasicOptimizationProblemOverCycles}, the maximum is achieved on a relatively short cycle (in our experiments the lengths almost never exceeded 50);
	\item[2)] As the algorithm proceeds, reference points become highly overlapped, both within different cycles and within the same cycles.
\end{enumerate}

Thus, based on item 2), the overall algorithm is concerned not just with the optimization of norms (singular values) at scattered points, but rather with their influence in the context of certain cycles in the symbolic image. This adds a guidance by dynamics to the optimization process.

We refer to Section \ref{SEC: ComputationalTips} for further discussions related to the Nonlinear Iterative Programming optimization.

Now let us discuss existing methods and problems concerned with resolution of \eqref{EQ: NonlinearProgrammingCycles}. 

First of all, we theoretically face the problem that \eqref{EQ: UP2SingValue} is not a differentiable function (of the parameter) in general. It is well-known that the arranged eigenvalues are globally Lipschitz continuous in the space of Hermitian matrices. Moreover, at points (matrices) where a chosen eigenvalue is simple, it depends analytically on the matrix. For families of matrices, this gives as much smoothness as the family has. In the absence of simplicity (which can be lost as parameters vary), we may rely only on local Lipschitzity.

For such possibly nonsmooth constrained optimization problems, F.E.~Curtis and M.L.~Overton \cite{CurtisOverton2012} developed a sequential quadratic programming (SQP) method with the search direction based on the gradient sampling (GS). They also established global convergence of the SQP-GS algorithm to minima with probability one. However, for problems with $r$ parameters it requires the evaluation of $O(r)$ gradients per step of the algorithm. This makes the approach inappropriate for problems like our \eqref{EQ: NonlinearProgrammingCycles}, although one can use it to compute weights in the symbolic image (this is also a generally nonsmooth problem).

A more efficient sequential quadratic programming approach is proposed by F.E.~Curtis, T.~Mitchell and M.L.~Overton in \cite{CurtisMitchellOverton2017}. It employs Broyden-Fletcher-Goldfarb-Shanno (BFGS) quasi-Newton Hessian approximations and exact penalty functions with parameters controlled by a steering strategy. Although the BFGS-SQP method does not have a guaranteed convergence theory, it gives effective and efficient results on a set of problems in \cite{CurtisMitchellOverton2017} outperforming existing methods in time-efficiency. This algorithm is implemented in GRANSO package for MATLAB and PyGRANSO package for Python.

In our experiments, we found that the sequential least-squares quadratic programming (SLSQP) method (see D.~Kraft \cite{Kraft1988}) implemented in Fortran and with an interface provided by SciPy package for Python, although being originally developed for smooth problems, significantly outperforms PyGRANSO in convergence and time complexity for the problems of our interest\footnote{This is also known from a comparison of SLSQP with many other existing methods for constrained optimization, where it shows a significant qualitative outperforming in complex-constrained problems. See, e.g., the paper of A.J.~Joshy and J.T.~Hwang \cite{JoshyHwang2024}.}. 

We plan to give details concerned with effective implementations of the algorithm on Python, design the corresponding programming package and provide related discussions in a forthcoming paper. In the present work, we give two illustrative examples in Sections \ref{SEC: HenonLPestimate} and \ref{SEC: RabinovichLargestULE} justifying that the method indeed can be used to reach global minima in both simple and complex optimization problems.
\section{Computation tips and regularizing optimization strategies}
\label{SEC: ComputationalTips}

For simple systems, which particularly have largest uniform Lyapunov exponents $\lambda_{1}(\Xi_{m})$ achieved at some equilibrium, one should not expect serious obstacles for a successful optimization. In Section \ref{SEC: HenonLPestimate}, this is illustrated by means of the H\'{e}non mapping with classical parameters, where only $20$ distinct reference cycles were collected until the success.

However, for systems which are expected to have periodic extreme points at which the Lyapunov exponent of our interest is realized, the optimization process may reach deep stages with hundreds or even thousands of distinct reference cycles because the shapes of adapted metrics may be relatively complex. In this situation, one should involve other heuristics to catch appropriate shapes.

Firstly, on the stage of Iterative Nonlinear Programming optimization, we do not need to compute symbolic images accurately. It is sufficient to take a few long pseudo-trajectories, say $q_{l}(t)$, and construct a graph with only edges $(i,j)$ such that $q_{l}(t) \in \mathcal{U}^{i}$ and $q_{l}(t + \tau_{i}) \in \mathcal{U}^{j}$ for some $t \geq 0$ and $l$. We call such graphs \textit{heuristic symbolic images}. If the pseudo-trajectories are sufficiently long and accurately computed, the graph will contain all relevant paths. Heuristic symbolic images may help to find appropriate parameters (such as the size of covering subsets, transition times, number of parameters) for optimization.


Next, the central problem is choosing a class of metrics. For complex problems, naive polynomial expressions, although providing improved estimates, may be not sufficient to catch proper asymptotics with computationally acceptable degrees of polynomials. In this regard, it is worth to use neural networks models or other multi-parametric families allowing approximations of different kinds. Here for choosing appropriate architectures, it is convenient to provide simple tests. One of such tests for a model is to fit data from Lyapunov-Floquet metrics (see Appendix \ref{SEC: LyapunovFloquetMetrics}). As more advanced tests, one can extract some periodic orbits $q_{l}(\cdot)$ from the attractor and try to optimize over the \textit{heuristic periodic symbolic image} constructed (as above) from the finite periodic skeleton.

\begin{remark}
	\label{REM: StabUnstableOrbits}
	For extracting unstable periodic orbits from chaotic attractors there exist various methods. 
	
	Firstly, there are approaches based on stabilization. Among them the method of delayed-feedback stabilization originated by K.~Pyragas \cite{Pyragas1992} stands out due to its simplicity (see N.V.~Kuznetsov, G.A.~Leonov and M.M.~Shumafov \cite{KuzLeoShum2015}). There exist various developments related to consideration of periodic feedback for continuous- and discrete-time systems (see G.A.~Leonov \cite{Leonov2015PeriodicPyragas}; G.A.~Leonov, K.A.~Zvyagintseva and O.A.~Kuznetsova \cite{LeoZvyagKuz2016}) or allowing adaptive searching for period and feedback gain estimates (see W.~Lin et al. \cite{LinMaFengChen2010}).
	
	Another approach is suggested by E.~Ott, C.~Grebogi and J.A.~Yorke in \cite{OttGrebogiYourke1990} and it is concerned with applying a controlling feedback at (discrete) return times to a Poincar\'{e} section to stabilize unstable periodic orbits.
	
	In \cite{BramburgerBruntonKutz2021}, J.J.~Bramburger, S.L.~Brunton and J.N.~Kutz developed deep learning algorithms for constructing approximate conjugacies between Poincar\'{e} mappings of certain high-dimensional models and explicit simple candidates. This may particularly help to extract unstable periodic orbits.
	
	However, the above methods often do not provide high-precision computation of unstable periodic orbits, rather giving a good guess for initial data. In particular, they may not allow to compute Floquet multipliers with sufficient accuracy (see Section \ref{SEC: RabinovichLargestULE}).
	
	More accurate methods are based on the Newton-Raphson method (see P.~Cvitanovi\'{c} et al. \cite{CvitanovicChaos2005}). Combining with rigorous ODE solvers, it may provide arbitrary precision computations of periodic orbits (see A.~Abad, R.~Barrio and A.~Dena \cite{AbadBarrioDena2011}). If one wishes to extract all cycles up to a given period, the Poincar\'{e} sections may be involved (see R.~Barrio, A.~Dena and W.~Tucker \cite{BarrioDenaTucker2015}). \qed
\end{remark}

Even with appropriate families of metrics, reaching global minima may be a complicated task. One of problems may be caused by a nonregularity related to optimization of relative reference weights \eqref{EQ: RelWeightRefPoints}. Namely, it often happens that optimization chooses the path of least resistance where some of individual weights are further increased (nevertheless lowering the associated relative weights). Since in the symbolic image there are many cycles containing a given set of vertices, this may result in catching new cycles, corresponding to reference points with increased individual weights, on several next iterations until the overlapping between reference cycles prevents the individual weights from increasing again.

To partially overcome the above problem, we suggest several regularizing optimization strategies concerned with complementing constraints in \eqref{EQ: NonlinearProgrammingCycles} by additional inequalities guiding the optimization. Thus, at the cost of growing time complexity for a single iteration, one may hope for qualitative and quantitative improvements of the overall optimization process.

\begin{description}[before=\let\makelabel\descriptionlabel]
	\item[\textbf{(IR)}\refstepcounter{desccount}\label{DESC: IR}] (Individual regularization) Complement constraints with inequalities for individual weights at reference points to not increase them during optimization:
	\begin{equation}
		\label{EQ: IndividualRegularization}
		\begin{split}
			\alpha^{+}_{f_{\omega}}(k,j,s) \leq \alpha^{+}_{f_{\omega_{0}}}(k,j,s) + \varepsilon,\\
			\alpha^{+}_{f_{\omega}}(k,j,s)  \coloneq \ln \omega^{(\mathfrak{n}(\omega))}_{d}(\Xi^{\tau_{i_{k,s}}}(q_{k,j,s}, \cdot))
		\end{split}
	\end{equation}
	where the indices are as in \eqref{EQ: RelWeightRefPoints} and $\varepsilon > 0$ is a small number.
\end{description}
This strategy, if successful, significantly regularizes the optimization process by lowering the number of distinct reference cycles collected per iteration. However, with many reference points or when we are close to the optimal value, it may lead to stucking at local minima which are not local minima for the original problem, especially for $\varepsilon = 0$. In this case, one may try to remove the reference cycles and points and proceed with empty data or to increase $\varepsilon$. This may give only small improvements, in which case it is worth to change the optimization strategy (keeping the possibility to return later).

It seems that \nameref{DESC: IR} is best suited for starting the optimization; in particular, if we wish to transit from a given symbolic image to a finer one with preserving optimized parameters. However, with sufficiently small $\varepsilon$, it will probably not allow to reach global minima for somewhat complex problems.

\begin{description}[before=\let\makelabel\descriptionlabel]
	\item[\textbf{(CR)}\refstepcounter{desccount}\label{DESC: CR}] (Cycles regularization) Complement constraints with $\widetilde{W}_{k,j}(\omega) \leq \widetilde{W}_{k,j}(\omega_{0}) + \varepsilon$ (see \eqref{EQ: RelWeightRefPoints}) for some $\varepsilon > 0$.
\end{description}
Even with $\varepsilon=0$, this strategy seems as a nice general complement to the constraints. It prevents all relative reference weights to be increased during the optimization. This adds qualitative improvements of collected cycles. Our experiments indicated that the strategy may also lead to stucking at irrelevant local minima as in the case of \nameref{DESC: IR} with $\varepsilon=0$, but it happens much more closer to the global minimum. Note that the success of optimization discussed in Section \ref{SEC: RabinovichLargestULE} is significantly based on using this strategy.


\begin{remark}
	\label{REM: ComputingMaximaShapeChanges}
	In our experiments, to compute weights in symbolic images during Iterative Nonlinear Programming optimization, where all covering elements $\mathcal{U}^{j}$ are cubes in $\mathbb{R}^{n}$, we use the SLSQP method from SciPy package for Python to resolve the related maximization problem taking the cube center as an initial point. 
	
	Sometimes, especially when the family of metrics is complicated, this may lead to finding only local minima at some cubes\footnote{A natural mechanism (which was particularly meant under the name ``evolution of shapes'' above) causing this is as follows. Suppose for some parameter we have a unique maximum at some $q$. Then small changes of parameter may cause $q$ to turn into a local minima with appearance of several local maxima in its neighborhood.}. Nevertheless, for the Iterative Nonlinear Programming optimization it is not worth to try obtaining global minima by taking more points as initial due to a significant increase in time complexity. It is better to allow for each reference cycle to have several (up to some $N_{ref} \geq j_{k} \geq 1$) families of reference points (see \eqref{EQ: INPReferencePoints}) iteratively collected as such local minima. As parameters vary, this helps catching relevant extrema and not to lose them during at least $N_{ref}$ iterates. Since reference points are highly overlapped, this adds only a relatively small increase in time complexity. In our experiments, we took $N_{ref} = 10$ and this turned out to be sufficient for a success of the optimization. Note that this approach can be successful only if smallness of the covering subsets respects the number of parameters. \qed
\end{remark}
\begin{remark}
	Of course, accurate computations of symbolic images and applications of the Relative Weights Optimization for a particular adapted metric demand more accurate finding of global minima in each covering element. This requires using batches of points.
	
\end{remark}
\begin{remark}
	\label{REM: OurDevice}
	All our numerical computations reported below were performed on CPU AMD Ryzen 5 5600 OEM with 32GB 3200 MHz RAM, GPU Palit GeForce RTX 4060 Dual and within double-precision floating-point arithmetic.
\end{remark}
\section{H\'{e}non mapping}
\label{SEC: HenonLPestimate}

Here we are going to illustrate the theoretical approach by means of the H\'{e}non mapping\footnote{In the standard form the mapping is given by $(x,y) \mapsto (1+y-ax^{2},bx)$. Clearly, it is conjugated with \eqref{EQ: HennonLeonovForm} via the change of variables $h(x,y)=(ax,ay/b)$.}
\begin{equation}
	\label{EQ: HennonLeonovForm}
	\mathbb{R}^{2} \ni (x,y) \overset{\vartheta}{\mapsto} ( a + b y - x^2, x ) \in \mathbb{R}^{2}
\end{equation}
with the parameters $a=1.4$, $b=0.3$ corresponding to the classical H\'{e}non attractor (see M.A.~H\'{e}non \cite{Henon1976}). Let us however state some general facts about \eqref{EQ: HennonLeonovForm}.

It is easily seen that there are two stationary states $q^{\pm} = (x_{\pm},x_{\pm})$, where
\begin{equation}
	x_{\pm} = \frac{1}{2}\left(b - 1 \pm \sqrt{ (b-1)^{2} + 4 a}\right).
\end{equation}

Note that there may exist points going to infinity under iterations of \eqref{EQ: HennonLeonovForm}. However, the nonwandering set can be trapped into the square (see R.~Devaney and Z.~Nitecki \cite{DevaneyNitecki1979})
\begin{equation}
	\label{EQ: BoundHennonNW}
	\begin{split}
		\mathcal{S}_{R} = \{ (x,y) \in \mathbb{R}^{2} \ | \ |x| \leq R, |y| \leq R \}, \text{ where }\\
		R = \frac{1}{2} \left( 1 + |b| + \sqrt{ (1 + |b|)^{2} + 4 a } \right).
	\end{split}
\end{equation}
For $a=1.4$, $b=0.3$ this gives $R = 1.96244\ldots \leq 2$.

Moreover, for $a=1.4$ and $b=0.3$, M.A.~H\'{e}non \cite{Henon1976} found a positively invariant region $\mathcal{S}$ avoiding the negative stationary state $q^{-}$. Its boundary is given by the quadrilateral with vertices
\begin{equation}
	\label{EQ: HennonQuadrilateral}
	\begin{split}
		&(x_{A}, y_{A}) = \left(-1.33 \cdot a, 0.42 \cdot \frac{a}{b}\right), \ (x_{B}, y_{B}) = \left(1.32 \cdot a, 0.133 \cdot \frac{a}{b}\right), \\
		&(x_{C},y_{C}) = \left(1.245 \cdot a, -0.14 \cdot \frac{a}{b} \right), \ (x_{D},y_{D}) = \left(-1.06 \cdot a, -0.5 \cdot \frac{a}{b} \right).
	\end{split}
\end{equation}

Consequently, there is an attractor in $\mathcal{S}$ given by $\mathcal{A} = \bigcap_{s \in \mathbb{T}_{+}} \vartheta^{s}(\mathcal{S})$. It is however not rigorously established that $\mathcal{A}$ is a chaotic attractor, although there are results indicating this. In particular, in the formation of $\mathcal{A}$ main roles are played by homoclinic intersections and tangencies associated with stable and unstable manifolds of $q^{+}$.

In \cite{Newhouseetal2008}, S.~Newhouse et al. obtained a lower bound for the topological entropy of \eqref{EQ: HennonLeonovForm} with $a=1.4$ and $b=0.3$ as
\begin{equation}
	\label{EQ: TopologicalEntropyHenonBelow}
	h_{\operatorname{top}}(\vartheta;\mathcal{A}) \geq 0.46469.
\end{equation}
Their method is based on finding subshifts carrying large entropy (in accordance with the celebrated A.~Katok result) embedded into $(\mathcal{A},\vartheta)$. This is achieved via computing transversal homoclinic intersections concerned with the positive stationary state $q^{+}$. Then one can extract a certain structure (called a trellis) associated with such intersections which gives rise to a topological Markov chain bounding the entropy from below. Related computations are done with the aid of computer algorithms which are robust and allow rigorous control of errors.

Note that \eqref{EQ: TopologicalEntropyHenonBelow} agrees well with the expected numerical evaluation $h_{\operatorname{top}}(\vartheta;\mathcal{A}) \approx 0.4651$ done by G.~D'Alessandro et al. \cite{Alessandroelal1990} using a long sequence of homoclinic tangencies and the associated grammars leading to Markov-like partitions. We relate such a relevance of the lower bound \eqref{EQ: TopologicalEntropyHenonBelow} to the lower-semicontinuity of the topological entropy in the space of $C^{r}$-surface mappings with any $r > 1$ \cite{Newhouse1989ContEnt}.

\begin{remark}
	Regardless of the lower-semicontinuity, there is a transient chaos in \eqref{EQ: HennonLeonovForm} with parameters close to the classical ones. Namely, for $a=1.39945219$ and $b=0.3$, the attractor of \eqref{EQ: HennonLeonovForm} is a period-13 cycle (see p.~78 and Exercise 6.3 in P.~Cvitanovi\'{c} et al. \cite{CvitanovicChaos2005}), although typical trajectories for a long time (about $10^{4}$ iterates) resemble the classical H\'{e}non attractor. So, all the entropy vanishes by a small perturbation of $a$.
\end{remark}

There is a general upper estimate $h_{\operatorname{top}}(\vartheta;\mathcal{K}) \leq n \ln d$ for a compact invariant set $\mathcal{K}$ of a polynomial mapping $\vartheta$ of degree $d$ in $\mathbb{R}^{n}$ due to M.~Gromov. In \cite{Newhouse1988EntropyVol}, S.~Newhouse sharpened this estimate to $h_{\operatorname{top}}(\vartheta;\mathcal{K}) \leq (n-1) \ln d$ provided that $\vartheta$ is injective in a neighborhood of $\mathcal{K}$ (see Theorem 4 therein). For \eqref{EQ: HennonLeonovForm} this gives
\begin{equation}
	\label{EQ: HenonNewhouseAboveEstimate}
	h_{\operatorname{top}}(\vartheta;\mathcal{A}) \leq \ln 2 = 0.6931\ldots
\end{equation}

Now we are going to our approach. Since \eqref{EQ: HennonLeonovForm} contracts areas, \eqref{EQ: TopologicalEntropyViaUnExps} gives
\begin{equation}
	\label{EQ: EntropyHenonViaLyapExp}
	h_{\operatorname{top}}(\vartheta;\mathcal{A}) \leq \lambda_{1}(\Xi;\mathcal{A}).
\end{equation}

There are previous investigations related to the estimate \eqref{EQ: EntropyHenonViaLyapExp}. Namely, G.A.~Leonov \cite{LeonovForHenonLor2001} constructed a metric with the aid of Lyapunov-like functions which can be used to show that (see, for example, Theorem 16 in A.~Matveev and A.~Pogromsky \cite{MatveevPogromsky2016})
\begin{equation}
	\lambda_{1}(\Xi;\mathcal{A}) \leq \lambda_{1}(\Xi;q^{-}) = \ln\left( \sqrt{x^{2}_{-} + b} - x_{-}\right)
\end{equation}
which for $a=1.4$ and $b = 0.3$ gives $\lambda_{1}(\Xi;\mathcal{A}) \leq 1.1816\ldots$. However, $q^{-}$ does not belong to $\mathcal{A}$ and it is much more unstable comparing with $q^{+}$.

In \cite{KawanHafsteinGiesl2021}, C.~Kawan, S.~Hafstein and P.~Giesl developed a projected subgradient method for the resolution of related optimization problems in the conformal class of a constant metric (inspired by the method of G.A.~Leonov \cite{Kuznetsov2016}). This allowed them to construct a metric giving the estimate for $a=1.4$ and $b=0.3$ (see Section 6.1 in \cite{KawanHafsteinGiesl2021})
\begin{equation}
	\label{EQ: HenonEntEstKawan1}
	\lambda_{1}(\Xi;\mathcal{A}) \leq 0.9907\ldots
\end{equation}
In \cite{LouzeiroetAlAdaptedMetricsComp2022}, M.~Louzeiro et al. by the same method (but with different step choices) improved the estimate to (see Section 5.2 therein)
\begin{equation}
	\label{EQ: HenonEntEstKawan2}
		\lambda_{1}(\Xi;\mathcal{A}) \leq 0.9045\ldots
\end{equation}

We are aimed to justify for $a=1.4$ and $b=0.3$ that
\begin{equation}
	\label{EQ: HenonLargestLPConjecture}
	\lambda_{1}(\Xi;\mathcal{A}) = \lambda_{1}(\Xi;q^{+}) = \ln \left( \sqrt{x^{2}_{+} + b} + x_{+}\right) = 0.6542706\ldots
\end{equation}
This improves all the estimates, including \eqref{EQ: HenonNewhouseAboveEstimate}, and gives the maximum that can be achieved via \eqref{EQ: EntropyHenonViaLyapExp}. More precisely, we will provide a robust upper bound for $\lambda_{1}(\Xi;\mathcal{A}) \geq \lambda_{1}(\Xi;q^{+})$ which is close to $\lambda_{1}(\Xi;q^{+})$ by constructing an adapted metric $\mathfrak{n}$ on $T \mathbb{R}^{2} \cong \mathbb{R}^{2}$ via the Iterative Nonlinear Programming discussed in Section \ref{SEC: OptimizatiobViaINP} (see Tab.~\ref{TAB: HenonMetricValues}) and applying the Relative Weights Optimization for finding maximal paths of increased lengths providing upper estimates for $\lambda_{1}(\Xi;\mathcal{A})$ (see Tab.~\ref{TAB: HenonEstimatesPathLengths}).

Let us firstly describe nuances related to the Relative Weights Optimization algorithm.

For step \nameref{DESC: GOA1}, let us take $\mathcal{U}_{0} = \mathcal{S} \cap \mathcal{S}_{2}$ which contains\footnote{Rigorously speaking, this (experimentally evident) statement requires to show that every point from $\mathcal{A}$ is nonwandering and this may be a complicated task. However, what we really need for our theory is what at least one of extreme points for $\lambda_{1}(\Xi;\mathcal{A})$ lies in the nonwandering set that is true since there always exist extreme points which are recurrent. This can be seen in a purely topological context (see Corollary A.8 in \cite{Morris2013}) or via ergodic theory machinery (see Theorem A.3 in \cite{Morris2013}).} $\mathcal{A}$ as it is discussed around \eqref{EQ: BoundHennonNW} and \eqref{EQ: HennonQuadrilateral}. We fix $\varepsilon>0$ and cover $\mathcal{S}_{2} = [-2,2]^{2}$ by the union of cubes
\begin{equation}
	\label{EQ: HenonCubesPartition}
	\mathcal{Q}_{k,l} = [-2, -2 + \varepsilon]^{2} + \varepsilon \cdot (k, l)
\end{equation}
where $k,l \in \{0,\ldots, N_{\varepsilon} \}$ and $N_{\varepsilon} = \lceil 4/\varepsilon \rceil$.
According to step \nameref{DESC: GOA2}, as a preparatory covering of $\mathcal{U}_{0}$ we take only cubes which intersect the quadrilateral $\mathcal{S}$.

To identify intersections and construct $G_{0}$, we consider a uniform grid of points on each cube by dividing its sides by $N_{g} = 100$. Moreover, we found it convenient to work with the second iteration $\vartheta^{2}$ of \eqref{EQ: HennonLeonovForm} (and equal transition times $\tau_{i} = 1$) because there may be cubes close to the negative stationary state $q^{-}$ having self-intersections under just the first iteration. This leads to carrying wrong asymptotics around $q^{-}$ which does not belong to $\mathcal{A}$.

For each $\mathcal{Q}_{k,l}$, we compute the maximum\footnote{This is done as in Remark \ref{REM: ComputingMaximaShapeChanges}. Even if there are several extrema at some cubes and some $\alpha_{k,l}$ are accidentally computed as local maxima by the method, we do not expect this to cause any differences for the reported results.}, say $\alpha_{k,l}$, over $q \in \mathcal{Q}_{k,l}$ of the largest singular value of the differential $D_{q}\vartheta^{2}$ in the standard Euclidean metric. For each point $p$ in the grid we identify $(k_{0},l_{0})$ such that $\vartheta^{2}(p) \in \mathcal{Q}_{k_{0},l_{0}}$ and register the corresponding intersection (i.e. we add the edge to the graph). Moreover, we register a rough intersection with any adjacent to $\mathcal{Q}_{k_{0},l_{0}}$ cube to which $\vartheta^{2}(p)$ is close up to $2 \alpha_{k,l} \varepsilon/{N_{g}}$. In other words, we let $p$ carry a disk of radius $\sqrt{2}\varepsilon/N_{g}$ around it (which contains the cube  $p + [0, \frac{\varepsilon}{N_{g}}]^{2}$) and assume the worst situation that it is uniformly stretched under $\vartheta^{2}$ with the maximal coefficient $\alpha_{k,l}$. Then the estimate $2 \alpha_{k,l} \varepsilon/{N_{g}}$ also carries possible round-off errors.

Therefore we compute a preparatory graph $G_{0}$ according to step \nameref{DESC: GOA3}. Then we remove sinks and sources from it, getting a more relevant localization of the attractor $\mathcal{A}$ given by the refined covering $\mathcal{U} = \bigcup_{j=1}^{N}\mathcal{U}_{j}$ by the remained cubes $\mathcal{U}_{j}$ according to step \nameref{DESC: GOA4}. Examples of such refined coverings obtained for $\varepsilon = 0.1$ and $\varepsilon = 0.05$ are shown in Fig. \ref{Fig: ExampleHenonCovering}.

\begin{figure}
	\begin{minipage}{.5\textwidth}
		\includegraphics[width=\textwidth]{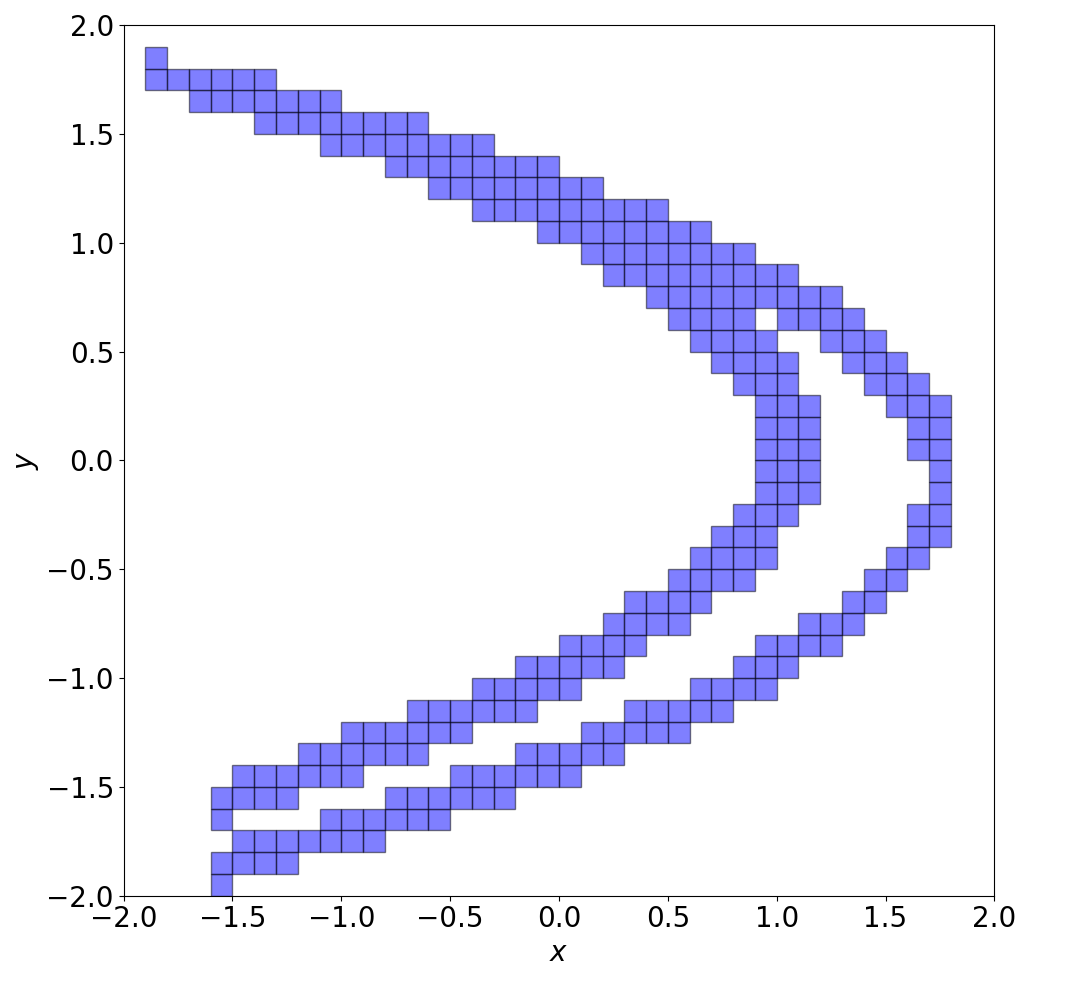}
	\end{minipage}%
	\begin{minipage}{.514\textwidth}
		\includegraphics[width=\textwidth]{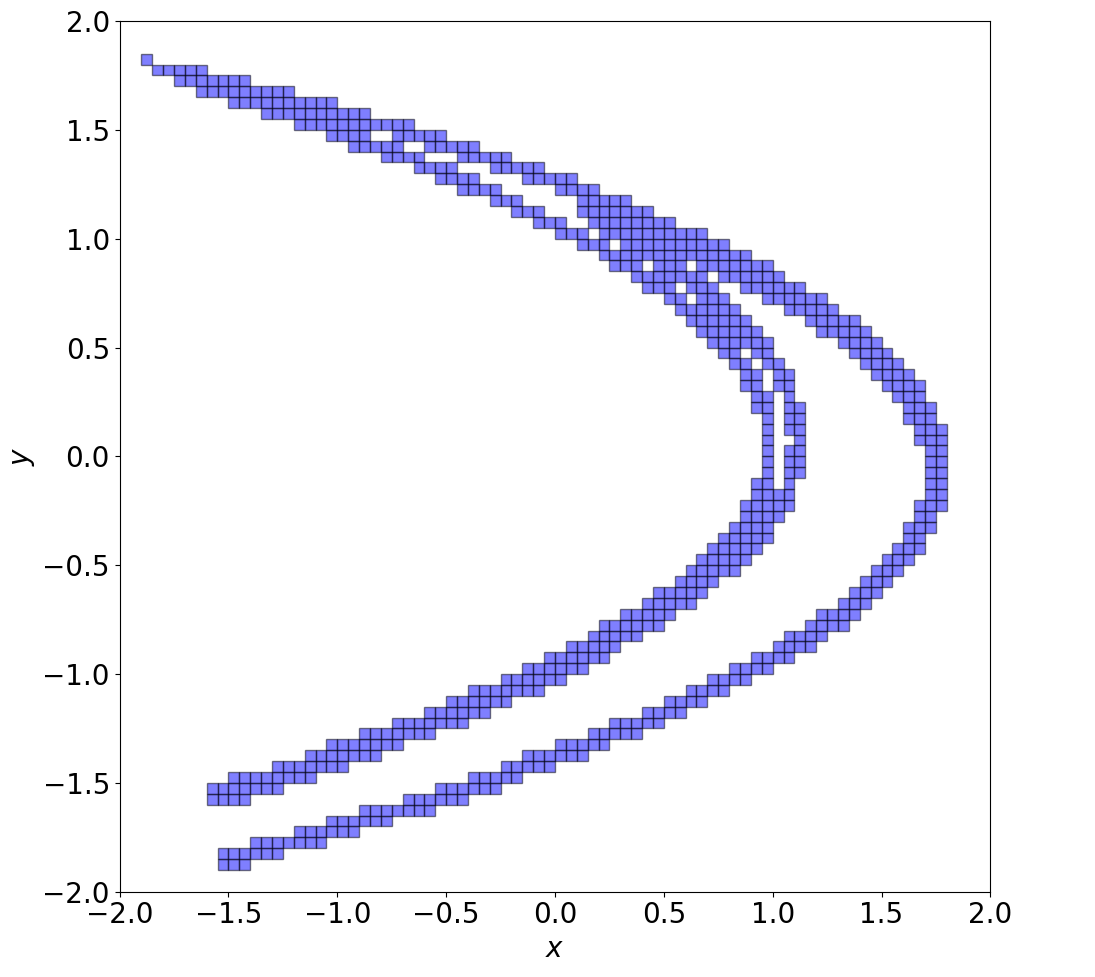}
	\end{minipage}%
	\caption{Examples of refined coverings for $\varepsilon = 0.1$ (left) and $\varepsilon = 0.05$ (right).}
	\label{Fig: ExampleHenonCovering}
\end{figure}

Let $\langle \cdot, \cdot \rangle_{\operatorname{Euc}}$ be the standard Euclidean inner product in $\mathbb{R}^{2}$. We consider a family of inner products $\langle \cdot, \cdot\rangle_{q}$ in $T_{q}\mathbb{R}^{2}$ over $q \in \mathcal{U}$. After natural identification of tangent spaces with $\mathbb{R}^{2}$, $\langle \cdot, \cdot\rangle_{q}$ is given by
\begin{equation}
	\label{EQ: HenonGeneralAbstractMetric}
	\langle \xi_{1}, \xi_{2}\rangle_{q} \coloneq \langle \xi_{1} , P_{q} \xi_{2}\rangle_{\operatorname{Euc}},
\end{equation}
where $q = (x,y) \in \mathcal{U}$, $P_{q}$ is a real symmetric positive-definite matrix and $\xi_{1}, \xi_{2} \in T_{q} \mathbb{R}^{2} = \mathbb{R}^{2}$.

Let $\mathfrak{m}$ be the metric (a family of norms) associated with \eqref{EQ: HenonGeneralAbstractMetric}. Consider the derivative $\Xi^{2}(q,\cdot) = D_{q}\vartheta^{2}$ of $\vartheta^{2}$ at $q=(x,y)$. It is not hard to see that it is given by
\begin{equation}
	\Xi^{2}(q,\cdot) = \begin{pmatrix}
		&b + 4x (a + b y - x^{2}) \ &&-2b (a + b y - x^{2}) \\
		&-2 x \ &&b
	\end{pmatrix}.
\end{equation}
Then it is well-known that the singular values of $\Xi^{2}(q,\cdot)$ w.r.t. $\mathfrak{m}$ are given by the singular values of the matrix (see Remark \ref{REM: SingularValuesComputation})
\begin{equation}
	\label{EQ: HenonMaximizedGrowthAsSingularValue}
	\sqrt{P_{\vartheta^{2}(q)}} \Xi^{2}(q,\cdot) \sqrt{P_{q}}^{-1}.
\end{equation}

Let us start with the family $f^{t}(q) = \ln \| \Xi^{2 t}(q,\cdot) \|_{\mathfrak{m}}$. According to \nameref{DESC: GOA5}, the weights $\alpha^{+}_{f}(j)$ in the symbolic image are given by maximizing $\alpha^{+}_{\mathfrak{m}}(q) = f^{1}(q)$ over $q \in \mathcal{U}^{j}$.

As an illustration to the last step \nameref{DESC: GOA6}, let us choose $P_{q}$ in \eqref{EQ: HenonGeneralAbstractMetric} as the identity matrix. Then for $\varepsilon = 0.01$, the estimate via a maximal path of length $t=10^{6}$ gives (we shall not discuss details) the estimate $\lambda_{1}(\Xi;\mathcal{A}) \leq 0.74309\ldots\ldots$ which already improves \eqref{EQ: HenonEntEstKawan2}.

Next, we apply the Iterative Nonlinear Programming optimization (see around \eqref{EQ: NonlinearProgrammingCycles}) to find a metric $P_{q}$ of the form
\begin{equation}
	\label{EQ: HenonMetricGeneralForm}
	P_{q} = e^{V(x,y)} \cdot \left[\begin{pmatrix}
		a_{11}(x,y) \ a_{12}(x,y)\\
		a_{12}(x,y) \ a_{22}(x,y)
	\end{pmatrix}^{2} + \begin{pmatrix}
	1 \ 0\\
	0 \ 1
	\end{pmatrix} \right],
\end{equation}
where $a_{11}$, $a_{12}$, $a_{22}$ and $V$ are real-valued polynomials with degrees $1$ for the matrix entries and $5$ for $V$. Coefficients of the polynomials are taken as parameters. On each iteration, variations of the coefficients $\omega_{j}$ corresponding to monomials with degree $k$ were bounded by the interval $[-0.025 \cdot 2^{k}, 0.025 \cdot 2^{k}]$ (therefore determining $\Omega_{bnd}$) and variations of $\omega^{r+1}$ were bounded by $W_{h} = 0.005$ (see \eqref{EQ: NonlinearProgrammingCycles}). For each reference cycle we allowed to have up to $j_{k} \leq 10$ families of reference points. Optimization was performed on the symbolic image computed for $\varepsilon = 0.01$.

\begin{table}
	\begin{minipage}{1.\linewidth}
		\centering
		\begin{tabular}{||c | c||} 
			\hline
			\text{Polynomial} & \text{Values} \\ [1.ex] 
			\hline\hline
			$a_{11}(x,y)$ & $0.45416294331290125 + 0.020513510515392696y$\\
			&$+0.5825112360363504x$ \\
			\hline
			$a_{12}(x,y)$ & $-0.004713193356337495 + 0.2922817455833501y$\\ &$-0.007977225865509469x$ \\
			\hline
			$a_{22}(x,y)$ & $0.0676334142588755 -0.030301622546198646 y$\\ 
			&$+ 0.015984077263301436 x$ \\ 
			\hline
			$V(x,y)$ &$0.026333660960733175 y - 0.03801620581510509 y^2$\\
			&$+ 0.08406821440555853 y^3 - 0.0434679060595814 y^4$\\ 
			&$+ 0.06314183074749603 y^5 + 0.011864879010530308 x$\\
			&$+ 0.045124703949557816 xy - 0.06153909040839601 xy^2$\\ 
			&$+ 0.09637541819232207 xy^3 - 0.06348892062303364 xy^4$\\ 
			&$+ 0.019348395756412246 x^2 + 0.039301442261707224 x^2y$\\ 
			&$- 0.07968948274531053 x^2 y^2 + 0.11255056942758461 x^2 y^3$\\ 
			&$+ 0.03232142219661241 x^3 + 0.021643240656063066 x^3y$\\ 
			&$- 0.08000812774170378 x^3y^2 + 0.02115259866492413 x^4$\\ 
			&$+ 0.025994133897055023 x^4 y + 0.014034588529305288 x^5$ \\
			\hline
		\end{tabular}
	\end{minipage}%
	\caption{Polynomials defining the metric in \eqref{EQ: HenonMetricGeneralForm} as a result of the Iterative Nonlinear Programming optimization.}
	\label{TAB: HenonMetricValues}
\end{table}

It took 50 iterations of the method to converge providing an estimate close to the expected optimal value at $q^{+}$ (see below). On our device (see Remark \ref{REM: OurDevice}) searching for maximal paths of length $t=1000$ took about $3$ seconds per iteration. Recomputing weights in the graph took about $5$ seconds per iteration. Resolving the nonlinear programming problems via SLSQP from SciPy took $0.1$ seconds per iteration and showed convergence results. Resulted polynomials are collected in Tab.~\ref{TAB: HenonMetricValues}. Fig. \ref{Fig: HenonIterativeNonlinearProgrammingOptimizationWeights} illustrates convergence of the optimization procedure and Fig. \ref{Fig: HenonIterativeNonlinearProgrammingOptimizationDataGrowth} shows growing of the corresponding data. In particular, at the end, the algorithm collected only 20 distinct cycles. For 655 collected reference points (counted with multiplicities) there were only 127 distinct points. This justifies the high overlapping discussed in Section \ref{SEC: OptimizatiobViaINP} even for such small data.

\begin{figure}
	\includegraphics[width=\textwidth]{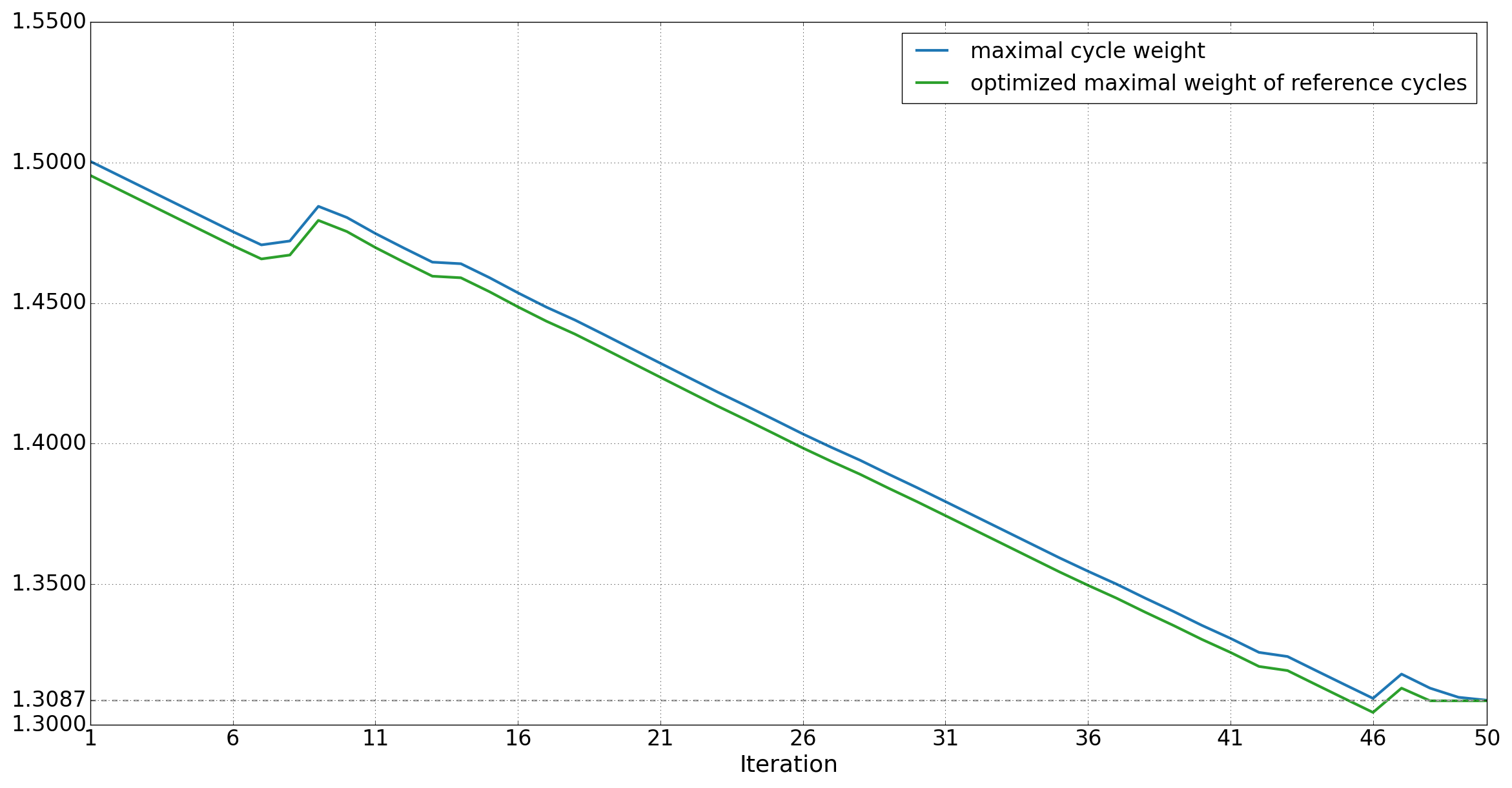}
	\caption{Graphs of the maximal cycle weight (extracted from the maximal path of length $1000$; blue) and the optimized maximal weight over reference cycles w.r.t. reference points (green) versus iteration of the Iterative Nonlinear Programming optimization described below \eqref{EQ: HenonMetricGeneralForm}.}
	\label{Fig: HenonIterativeNonlinearProgrammingOptimizationWeights}
\end{figure}

\begin{figure}
	\includegraphics[width=\textwidth]{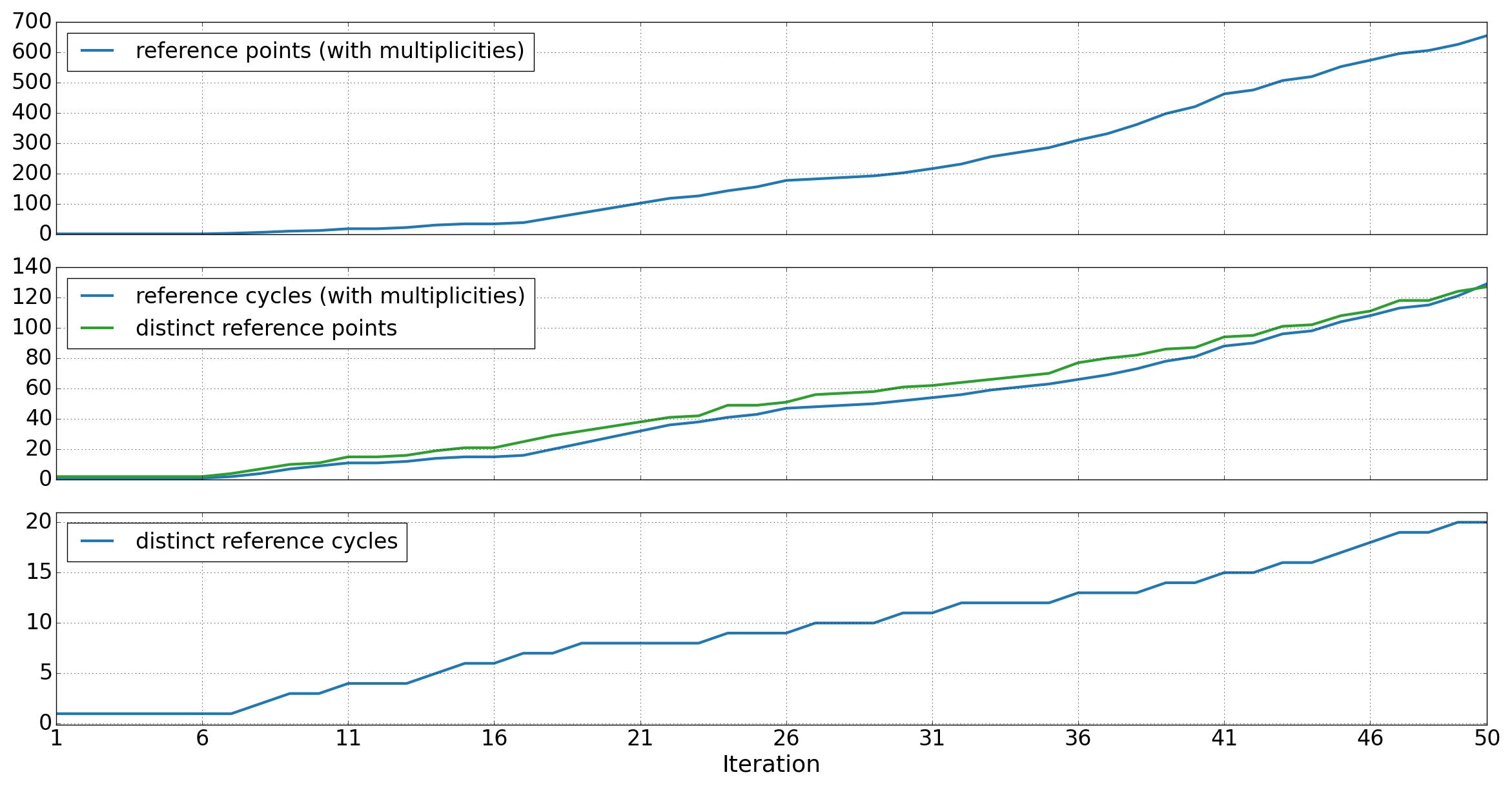}
	\caption{Graphs of the data growing (number of reference cycles and reference points) versus iteration of the Iterative Nonlinear Programming optimization described below \eqref{EQ: HenonMetricGeneralForm}.}
	\label{Fig: HenonIterativeNonlinearProgrammingOptimizationDataGrowth}
\end{figure}

Following the last step \nameref{DESC: GOA6}, we need to find a path with a given length $t$ having maximal mean weight in the refined graph $G$. We computed such weights for $t = 10^{k}$ and $k \in \{1,\ldots,6\}$. This gives a sequence of improving estimates which are collected in Tab.~\ref{TAB: HenonEstimatesPathLengths}. Note that for $k=6$, the estimate agrees with \eqref{EQ: HenonLargestLPConjecture} up to $5$ digits.
\begin{table}
	\begin{minipage}{1.\linewidth}
		\centering
		\begin{tabular}{||c | c||} 
			\hline
			\text{Path length t} & \text{Upper estimate for $\lambda_{1}(\Xi;\mathcal{A})$} \\ [1.ex] 
			\hline\hline
			$10$ & $0.7466752468429976$\\
			\hline
			$10^{2}$ & $0.6635115149479631$\\
			\hline
			$10^{3}$ & $0.6551951417584647$\\
			\hline
			$10^{4}$ & $0.6543635044395801$\\
			\hline
			$10^{5}$ & $0.6542803407063347$\\ 
			\hline
			$10^{6}$ & $0.6542720243392837$\\
			\hline
		\end{tabular}
	\end{minipage}%
	\caption{Estimates for the largest uniform Lyapunov exponent $\lambda_{1}(\Xi;\mathcal{A})$ for the classical H\'{e}non attractor $\mathcal{A}$ ($a=1.4$, $b=0.3$) via maximal paths of length $t$ in the symbolic image with weights corresponding to the metric from \eqref{EQ: HenonMetricGeneralForm} with values in Table \ref{TAB: HenonMetricValues}.}
	\label{TAB: HenonEstimatesPathLengths}
\end{table}

Moreover, for $k \in \{1,\ldots,4\}$ we extracted the simple cycle with maximal multiplicity from the optimal path. In all cases, it was a self-loop $\mathfrak{i}^{c}$ corresponding to the cube $\mathcal{Q}_{288,288}$ containing the positive equilibrium $q^{+}$ with 
\begin{equation}
	W_{\tau}(\mathfrak{i}^{c}) = 2 \cdot 0.6542711002929601\ldots.
\end{equation}
In particular, one cannot obtain more precise to \eqref{EQ: HenonLargestLPConjecture} estimates without further improving the metric.

Moreover, we can also justify that the Lyapunov dimension $\dim_{\operatorname{L}}(\Xi;\mathcal{A})$ of $\Xi$ on $\mathcal{A}$ is achieved\footnote{In fact, if \eqref{EQ: HenonLargestLPConjecture} is true, then \eqref{EQ: HenonLDAtPosState} must be also satisfied, since the H\'{e}non mapping has a constant determinant.} at $q^{+}$, i.e.
\begin{equation}
	\label{EQ: HenonLDAtPosState}
	\dim_{\operatorname{L}}(\Xi;\mathcal{A}) = \dim_{\operatorname{L}}(\Xi;q^{+}) = 1 + \frac{\lambda_{1}(\Xi;q^{+})}{\lambda_{1}(\Xi;q^{+}) - \ln b} \overset{b = 0.3}{=} 1.3520909\ldots
\end{equation}
Namely, in the same metric we may get the estimate $\dim_{\operatorname{L}}(\Xi;\mathcal{A}) \leq 1.35361$ and after 10 iterates of further corrections, we get (via the path of length $t=10^{6}$)
\begin{equation}
	\label{EQ: HenonLyapDimAttractor}
	\dim_{\operatorname{L}}(\Xi;\mathcal{A}) \leq 1.352095.
\end{equation}
In particular, we obtain the same estimate for the fractal dimension of $\mathcal{A}$ via \eqref{EQ: FractalDimEstViaLD}. This also improves the corresponding estimates from \cite{KawanHafsteinGiesl2021,LouzeiroetAlAdaptedMetricsComp2022} and establishes the maximum that can be achieved via the volume contraction method.

In the repository, we present Python programs for constructing symbolic image for \eqref{EQ: HennonLeonovForm} with $a=1.4$ and $b=0.4$ and verification of all the estimates from Tab.~\ref{TAB: HenonEstimatesPathLengths} and \eqref{EQ: HenonLyapDimAttractor}.
\section{Rabinovich system}
\label{SEC: RabinovichLargestULE}

Let us consider the following Lorenz-like system
\begin{equation}
	\label{EQ: LorenzLikeSytem}
	\begin{split}
		&\dot{x} = -\sigma (x - y) - a y z,\\
		&\dot{y} = r x - y - x z,\\
		&\dot{z} = -b z + xy,
	\end{split}
\end{equation}
where $\sigma, r, b > 0$ and $a \in \mathbb{R}$ are real parameters.

For $a = 0$, \eqref{EQ: LorenzLikeSytem} is just the classical Lorenz system. For the interesting to us case of $a < 0$, \eqref{EQ: LorenzLikeSytem} can be transformed into the system studied by M.I.~Rabinovich \cite{Rabinovich1978} by properly scaling space and time variables.

Depending on the parameters, \eqref{EQ: LorenzLikeSytem} may have $1$, $3$ or $5$ equilibria (see Section 2 in \cite{LeonovBoi1992}). For us it is important that \eqref{EQ: LorenzLikeSytem} is dissipative and all solutions enter the ellipsoid $\mathcal{E}_{\delta}$ given by (see p.~9 in \cite{LeonovBoi1992})
\begin{equation}
	\label{EQ: RabinovichEllipsoidsLocalization}
	x^{2} + \delta y^{2} + (a + \delta)	\left( z - \frac{d + \delta r}{a + \delta} \right)^{2} \leq \frac{b(d + \delta r)^{2}}{2 c (a + \delta)},
\end{equation}
where $c = \min\{ \sigma, 1, b/2  \}$ and $\delta > 0$ is arbitrary such that $a + \delta > 0$. Therefore there exists a global attractor $\mathcal{A}$ of \eqref{EQ: LorenzLikeSytem}.

In the Master thesis\footnote{Supervised by T.N.~Mokaev who also communicated the result to us.} of E.G.~Fedorov entitled ``The Eden hypothesis for generalised Lorenz
system'' defended at St. Petersburg University in 2018, for the parameters
\begin{equation}
	\label{EQ: RabinovichSystemFedorovParameters}
	\sigma = 2.5, r = 1.25, b=1 \text{ and } a = -40
\end{equation}
it was justified that the (uniform) Lyapunov dimension (and the largest uniform Lyapunov exponent) over $\mathcal{A}$ is not achieved at equilibria. Namely, with the aid of numerical experiments it was found and stabilized a periodic orbit on $\mathcal{A}$ which produce larger values.

More precisely, under \eqref{EQ: RabinovichSystemFedorovParameters} there are three equilibria: $q^{0} = (0,0,0)$ and
\begin{equation}
	\label{EQ: RabonovichEquilibria}
	\begin{split}
		q^{\pm} = \sqrt{ \theta } \cdot \left( \pm 1, \pm \frac{rb}{\theta + b} , r \frac{\sqrt{ \theta }}{\theta + b}  \right), \text{ where }\\
		\theta = \frac{1}{2A}\left( -B - \sqrt{ B^{2} - 4AC } \right),\\
		A = -\sigma, \ B = b\sigma (r - 2) - ab r^{2} \text{ and } C = b^{2} \sigma (r-1).
	\end{split}
\end{equation}
As for the classical Lorenz system, we have that $q^{0}$ is a saddle, and $q^{\pm}$ are unstable foci. Numerical experiments justify the existence of a chaotic physical attractor $\mathcal{A}_{\varphi} \subset \mathcal{A}$ which is a Lorenz-like (i.e. singular) attractor, i.e. the saddle $q^{0} \in \mathcal{A}_{\varphi}$ plays a key role in its formation (see Fig. \ref{Fig: RabinovichAttractorLocalization}).

Let $\vartheta$ be the time-$1$ (below we will use different times) mapping of the semiflow generated by \eqref{EQ: LorenzLikeSytem} and $\Xi$ be the derivative cocycle over $(\mathcal{A};\vartheta)$. Straightforward computations give largest local Lyapunov exponents
\begin{equation}
	\lambda_{1}(\Xi;q^{0}) = 0.1702\ldots\ \text{ and } \lambda_{1}(\Xi;q^{\pm}) = 0.02551\ldots.
\end{equation}
and local Lyapunov dimensions
\begin{equation}
	\dim_{\operatorname{L}}(\Xi;q^{0}) = 1.1703\ldots\ \text{ and } \dim_{\operatorname{L}}(\Xi;q^{\pm}) = 2.011\ldots.
\end{equation}

In the Master thesis of E.G.~Fedorov it was found a short periodic orbit $\Gamma$ such that\footnote{In the original work it was $\lambda_{1}(\Xi;\Gamma) \approx 0.4848$ and $\dim_{\operatorname{L}}(\Xi;\Gamma) \approx 2.09726$ obtained from the delayed feedback stabilization. We present corrected values via the Newton-Raphson method.} $\lambda_{1}(\Xi;\Gamma) = 0.48265\ldots$ and $\dim_{\operatorname{L}}(\Xi;\Gamma) = 2.09686\ldots$ that provides evidence that both $\lambda_{1}(\Xi;\mathcal{A})$ and $\dim_{\operatorname{L}}(\Xi;\mathcal{A})$ are not achieved at equilibria. He also stabilized some larger periodic orbits and noted that the local characteristics over them decay as the period increases. Thus, the shortest periodic orbit $\Gamma$ seems to be a candidate for the extreme trajectory.

We are going to justify that
\begin{equation}
	\label{EQ: RabinovichULEConj}
	\lambda_{1}(\Xi;\mathcal{A}) = \lambda_{1}(\Xi;\Gamma) \approx 0.48265.
\end{equation}

It can be verified that under \eqref{EQ: RabinovichSystemFedorovParameters} the rectangular parallelepiped $[-20,20] \times [-2,2] \times [-1, 4]$ contains the global attractor of \eqref{EQ: LorenzLikeSytem}. For example, the ellipsoid $\mathcal{E}_{\delta}$ for $\delta = 125$ lies within the parallelepiped (see Fig. \ref{Fig: RabinovichAttractorLocalization}). It will be convenient to study the system \eqref{EQ: LorenzLikeSytem} after the change of variables
\begin{equation}
	\label{EQ: RabLorenzLikeChange}
	x \mapsto \frac{x}{20}, \ y \mapsto \frac{y}{2}, \ z \mapsto \frac{z - 1.5}{2.5}
\end{equation}
so that the cube $[-1,1]^{3}$ in these new coordinates will contain the global attractor. We will abuse notation and use $q^{0}, q^{\pm}, \mathcal{A}, \Gamma$ and $\mathcal{E}_{\delta}$ to denote the above defined objects after the change \eqref{EQ: RabLorenzLikeChange}.

It is not hard to see that under \eqref{EQ: RabLorenzLikeChange} the system transforms into
\begin{equation}
	\label{EQ: RabinovichAfterChange}
	\begin{split}
		&\dot{x} = -\sigma x + \frac{1}{10}(\sigma - 1.5 a) y - \frac{a}{4} yz,\\
		&\dot{y} = (10r - 15) x - y - 25 xz,\\
		&\dot{z} = -b z + 16 xy - \frac{3 b}{5}.
	\end{split}
\end{equation}

\begin{figure}
	\includegraphics[width=1\textwidth]{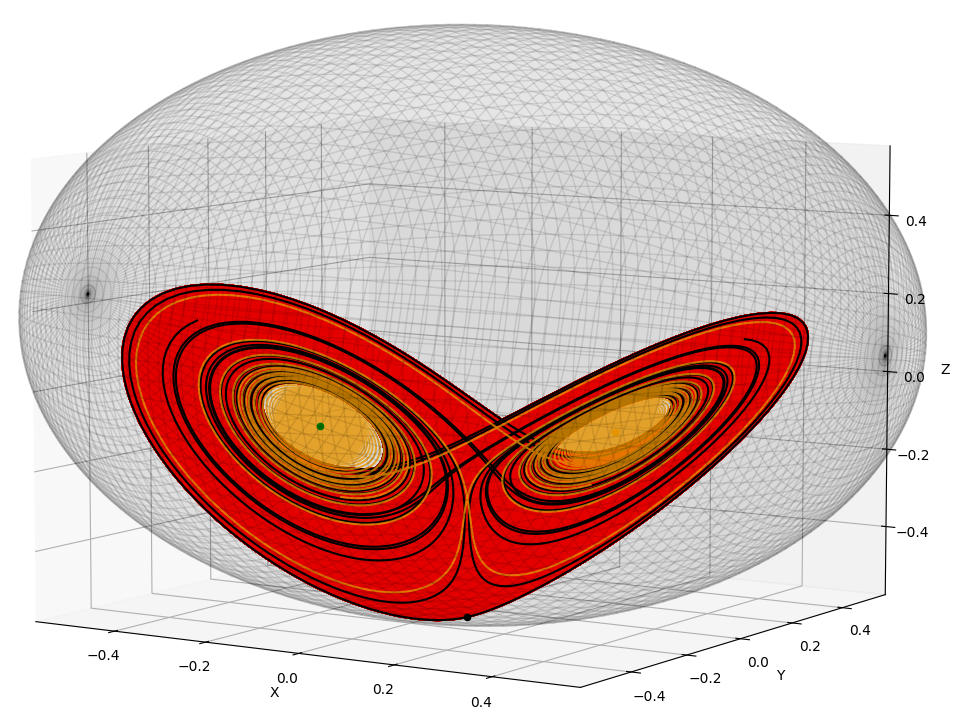}
	\caption{Numerical simulation of the attractor of \eqref{EQ: LorenzLikeSytem} with parameters \eqref{EQ: RabinovichSystemFedorovParameters} under the change of variables \eqref{EQ: RabLorenzLikeChange}. It is localized by a trajectory (red) coming from outside (transient time is removed) and trajectories (black and orange) starting respectively in small neighborhoods of $q^{0}$ (black bold dot), $q^{+}$ (orange bold dot) and $q^{-}$ (green bold dot). In agreement with rigorous results, all trajectories are contained in the ellipsoid $\mathcal{E}_{\delta}$ from \eqref{EQ: RabinovichEllipsoidsLocalization} with $\delta = 125$ (black transparent mesh).}
	\label{Fig: RabinovichAttractorLocalization}
\end{figure}

It is important that \eqref{EQ: RabinovichAfterChange} (as well as \eqref{EQ: LorenzLikeSytem}) has the $\mathbb{Z}_{2}$-symmetry given by
\begin{equation}
	\label{EQ: RabinovichSymmetry}
	S = \begin{pmatrix}
		-1 & 0 & 0\\
		0 & -1 & 0\\
		0 & 0 & 1
	\end{pmatrix},
\end{equation}
i.e. for any solution triple $(x(t), y(t), z(t))$, the triple $(-x(t),-y(t),z(t))$ is also a solution.

\begin{remark}
	\label{REM: NumericalIntegrationODEs}
	For numerical integration of ODEs with individual initial conditions we used JiTCODE package for Python (see G.~Ansmann \cite{AnsmannGJitCDDE2018}) and the Dormand-Prince method (dopri5) implemented therein with tolerance parameters taken as $\operatorname{atol} = 10^{-8}$ and $\operatorname{rtol} = 0$. 
\end{remark}


Since the attractor $\mathcal{A}$ of \eqref{EQ: RabinovichAfterChange} lies in $[-1,1]^{3}$, we partition the cube according to a given $\varepsilon > 0$ into smaller cubes
\begin{equation}
	\mathcal{Q}_{i,j,k} = [-1, -1 + \varepsilon]^{3} + \varepsilon \cdot (i,j,k),
\end{equation}
where $i,j,k \in \{0, \ldots, N_{\varepsilon}\}$ and $N_{\varepsilon} = \lceil 2/\varepsilon \rceil$.

We computed heuristic symbolic images (see Section \ref{SEC: ComputationalTips}) respecting the symmetry \eqref{EQ: RabinovichSymmetry} for $\varepsilon \in \{ 0.01, 0.005, 0.002, 0.001 \}$ and equal transition times $\tau_{i} = T_{0}/k_{\tau}$ with $k_{\tau} \in \{ 3, 5, 7 \}$ and $T_{0} = 3.52324$. More precisely, we integrated \eqref{EQ: RabinovichAfterChange} (see Remark \ref{REM: NumericalIntegrationODEs}) over the time interval $[0, 3 \cdot 10^{5}]$, obtaining values on the uniform grid of times with step $10^{-3}$, for three initial data $(x_{e},y_{e},z_{e}) = (0.1, 1, -1)$, $(x_{0},y_{0},z_{0} = q^{0} + (0.0001, 0, 0)$, $(x_{+},y_{+},z_{+}) = q^{+} + (0.0001, 0, 0)$ therefore obtaining $3$ trajectories $q_{e}(t)$, $q_{0}(t)$ and $q_{+}(t)$, where, abusing the notation, $q^{0}$ and $q^{\pm}$ denote the corresponding equilibria from \eqref{EQ: RabonovichEquilibria} after the change \eqref{EQ: RabLorenzLikeChange} and for $q_{e}(t)$ we remove the time segment $[0,500]$ as transient. Then the heuristic symbolic image $G=G^{h}_{\varepsilon,k_{\tau}}$ is constructed from 6 pseudo-trajectories $\{q_{e}(t), q_{0}(t), q_{+}(t), Sq_{e}(t), Sq_{0}(t), Sq_{+}(t) \}$. For $\varepsilon=0.001$ and $k_{\tau} = 3$, $G^{h}_{\varepsilon,k_{\tau}}$ contains about $10^{6}$ vertices and $11 \cdot 10^{6}$ edges.
 
For $\varepsilon = 0.01$, the Relative Weights Optimization applied in the standard metric gives the estimate $\lambda_{1}(\Xi;\mathcal{A}) \leq 2.6299\ldots$. This indicates that it is a challenging task to construct a metric justifying \eqref{EQ: RabinovichULEConj}. 

We searched for inner products $\langle \cdot, \cdot \rangle_{q} = \langle \cdot, P(q) \cdot \rangle_{\operatorname{Euc}}$ in $T_{q} \mathbb{R}^{3} \cong \mathbb{R}^{3}$, where $q = (x,y,z) \in \mathbb{R}^{3}$, with symmetric positive-definite $P(q)=P(x,y,z)$ given by (of course $S^{T} = S$)
\begin{equation}
	P(x,y,z) = \frac{1}{4}\left[ P_{N}(x,y,z) + S^{T} P_{N}(-x,-y,z) S  \right]^{2} + \operatorname{Id}_{3},
\end{equation}
where $\operatorname{Id}_{3}$ is the unit $(3\times 3)$-matrix and
\begin{equation}
	\label{EQ: RabinovichMetric}
	P_{N}(x,y,z) = \begin{bmatrix}
		N_{1}(x,y,z) \ N_{2}(x,y,z) \ N_{3}(x,y,z)\\
		N_{2}(x,y,z) \ N_{4}(x,y,z) \ N_{5}(x,y,z)\\
		N_{3}(x,y,z) \ N_{5}(x,y,z) N_{6}(x,y,z)
	\end{bmatrix} 
\end{equation}
with functions $N_{1}(x,y,z), \ldots, N_{6}(x,y,z)$ depending on some parameters. Clearly, such a metric respects the symmetry $S$ (see around \eqref{EQ: AveragingMetricToRespectS}).

We construct $P_{N}$ via a neural network model with architecture $(3,200,6)$ and activation function $\sigma(s) = 1/(1+s^{2})$ for $s \in \mathbb{R}$. More precisely, as model parameters we take $(200 \times 3)$- and $(6 \times 200)$-matrices $M_{1}$ and $M_{2}$ respectively and $200$- and $6$-vectors $b_{1}$ and $b_{2}$ respectively. Then the model is given by
\begin{equation}
	N(x,y,z) = (N_{1}(x,y,z),\ldots,N_{6}(x,y,z)) \coloneq M_{2}\sigma\left(M_{1}(x,y,z)^{T} + b_{1}\right) + b_{2} \in \mathbb{R}^{6},
\end{equation}
where the application of $\sigma$ to a vector should be understood componentwise. This model\footnote{We also tested several deep architectures, but they did not give better results.} showed a good fit to data from Lyapunov-Floquet metrics over the candidate periodic orbit and got to deep stages of optimization on the heuristic periodic graph constructed from $11$ shortest periodic orbits.

\begin{table}
	\begin{minipage}{1.\linewidth}
		\centering
		\begin{tabular}{||c | c||} 
			\hline
			\text{Path length t} & \text{Upper estimate for $\lambda_{1}(\Xi;\mathcal{A})$} \\ [1.ex] 
			\hline\hline
			$10$ & $0.9257762213273285$\\
			\hline
			$10^{2}$ & $0.5340507381091755$\\
			\hline
			$10^{3}$ & $0.4881811626497006$\\
			\hline
			$10^{4}$ & $0.4835925539521204$\\
			\hline
			$10^{5}$ & $0.4831336089521827$\\ 
			\hline
			$10^{6}$ & $0.4830877144522192$\\
			\hline
		\end{tabular}
	\end{minipage}%
	\caption{Estimates for the largest uniform Lyapunov exponent $\lambda_{1}(\Xi;\mathcal{A})$ over the global attractor $\mathcal{A}$ of the Rabinovich system with parameters \eqref{EQ: RabinovichSystemFedorovParameters} via maximal paths of length $t$ in the heurstic symbolic image $G_{\varepsilon,k_{\tau}}$ with $\varepsilon = 0.001$ and $k_{\tau} = 3$ and the weights corresponding to a metric from \eqref{EQ: RabinovichMetric} obtained after optimization.}
	\label{TAB: RabinovichEstimatesPathLengths}
\end{table}

Finding appropriate parameters of the neural network took about 2 months on our device (see Remark \ref{REM: OurDevice}). At the final stage, we used the heuristic symbolic image $G_{\varepsilon,k_{\tau}}$ with $\varepsilon = 0.001$ and $k_{\tau} = 3$. It took about 2000 iterations of the Iterative Nonlinear Programming optimization to converge resulting in a metric providing the estimate for $\lambda_{1}(\Xi;\mathcal{A})$ collected in Tab. \ref{TAB: RabinovichEstimatesPathLengths}. At the end, the algorithm collected 1408 distinct reference cycles, with 36411 distinct reference points (with about $7\cdot 10^{5}$ multiple reference points). Moreover, the maximal path of length $t=1000$ contains a simple cycle of length 106 with almost all vertices corresponding to covering elements of $\Gamma$.

For the final metric, we used a uniform grip of 27 points in each covering element to compute the corresponding global maxima. Even though the heuristic symbolic images $G_{\varepsilon,k_{\tau}}$ may be not too reliable, the obtained estimates are in a good agreement with the expectations, Namely, the estimate for $t=10^{6}$ agrees with \eqref{EQ: RabinovichULEConj} up to $5 \cdot 10^{-4}$.

In particular, this shows that the Iterative Nonlinear Programming optimization may succeed in complex problems.

\appendix
\section{Lyapunov-Floquet metrics for periodic cocycles}
\label{SEC: LyapunovFloquetMetrics}
In this section we start with a linear continuous-time cocycle $\Xi$ in a complex (see Remark \ref{REM: LyapunovFloquetRealSpace}) $n$-dimensional vector space $\mathbb{E}$ which is $\sigma$-periodic, i.e. the driving system $(\mathcal{Q},\vartheta)$ is a single $\sigma$-periodic orbit. It will be convenient to identify $\mathcal{Q}$ with the circle $\mathcal{S}^{1}_{\sigma} := \mathbb{R} / \sigma\mathbb{Z}$ of length $\sigma$ and $\vartheta$ with the translation flow on $\mathcal{S}^{1}_{\sigma}$, i.e. $\vartheta^{t}(q) = q + t \mod \sigma$.

We suppose that any trajectory $\xi(t) = \Xi^{t}(q,\xi_{0})$ of $\Xi$ with $\xi_{0} \in \mathbb{E}$ satisfies the linear differential equation
\begin{equation}
	\dot{\xi}(t) = A(t + q) \xi(t)
\end{equation}
with a $\sigma$-periodic continuous operator-valued function $A(\cdot)$.

Recall that $\Xi^{\sigma}(q,\cdot)$ is called the monodromy operator over $q \in \mathcal{Q}$. By the cocycle property, monodromy operators over different $q$ are conjugated and, consequently, they share the same set of eigenvalues $\mu_{1}, \mu_{2}, \ldots, \mu_{n}$ known as the Floquet multipliers. It is convenient to arrange them by nonincreasing of modules. Then $\sigma^{-1}|\mu_{1}|$, \ldots, $\sigma^{-1}|\mu_{n}|$ are known to coincide with the Lyapunov exponents of $\Xi$ (see below).

By the Floquet theorem, there exists a $\sigma$-periodic Lyapunov change of variables $C(t) \colon \mathbb{E} \to \mathbb{E}$ such that
\begin{equation}
	\label{EQ: FloquetChangeEquation}
	\Xi^{t}(q,\cdot) = C(q+t)  e^{tB},
\end{equation}
where $e^{tB}$ is the linear semigroup generated by the stationary equation
\begin{equation}
	\dot{\eta}(t) = B\eta(t)
\end{equation}
with $B$ given by the relation \eqref{EQ: FloquetChangeEquation} with $q = 0$ and $t = \sigma$, i.e.
\begin{equation}
	e^{\sigma B} = \Xi^{\sigma}(0,\cdot) \text{ or } B = \frac{1}{\sigma} \log \Xi^{\sigma}(0,\cdot).
\end{equation} 
Note that the logarithm exists since  $\Xi^{\sigma}(0,\cdot)$ is nondegenerate, but it may be not unique. Clearly, the eigenvalues (also called characteristic exponents) $\lambda_{1},\ldots,\lambda_{m}$ of $B$ arranged by nonincreasing of real parts are related to the Floquet multiplies by $|\mu_{j}| = e^{\sigma \operatorname{Re}\lambda_{j}}$ for any $j \in \{1,\ldots,n\}$.

\begin{remark}
	\label{REM: LyapunovFloquetRealSpace}
	If the space $\mathbb{E}$ is real (what is natural for applications), there are well-known spectral limitations concerned with the existence of real logarithms (there must be an even number of Jordan blocks corresponding to any negative eigenvalue). In general, there exists only a $2\sigma$-periodic change of variables $C(t)$ determined by putting $q = 0$ and $t=2\sigma$ in \eqref{EQ: FloquetChangeEquation}. This is not appropriate for the construction of Lyapunov-Floquet metrics below.
\end{remark}

Assume that there exists an eigenbasis $e_{1},\ldots,e_{n}$ of $B$ and let $\langle \cdot, \cdot \rangle_{B}$ be the (Hermitian\footnote{We use the agreement that a Hermitian form is linear in its first argument and conjugate-linear in the second.}) inner product in $\mathbb{E}$ making the eigenbasis orthonormal. Then the family of inner products
\begin{equation}
	\langle \xi_{1}, \xi_{2} \rangle_{q} \coloneq \left\langle C^{-1}(q)\xi_{1}, C^{-1}(q)\xi_{2} \right\rangle_{B}, \text{ where } \xi_{1},\xi_{2} \in \mathbb{E},
\end{equation}
defines a metric over $\mathcal{Q}$ which is called by us the \textit{Lyapunov-Floquet metric}.

Let $e_{1}(q),\ldots,e_{n}(q)$ be the images of $e_{1},\ldots,e_{n}$ respectively under $C(q)$. Then, by definition, $e_{1}(q),\ldots,e_{n}(q)$ is an orthonormal basis w.r.t. $\langle \cdot, \cdot \rangle_{q}$. By \eqref{EQ: FloquetChangeEquation}, in these $\sigma$-periodic coordinates, the cocycle $\Xi$ is diagonal as 
\begin{equation}
	\label{EQ: FloquetBasisProperties}
	\Xi^{t}(q,e_{j}(q)) = e^{\lambda_{j} t} e_{j}(\vartheta^{t}(q)) \text{ for any } t \geq 0,  q \in \mathcal{Q} \text{ and } j \in \{1,\ldots,n\}.
\end{equation}
In particular, the singular values of $\Xi^{t}(q,\cdot)$ in the metric are given by $\exp\left(\frac{t}{\sigma}\log|\mu_{1}|\right)$, \ldots, $\exp\left(\frac{t}{\sigma}\log|\mu_{n}|\right)$.

Thus, the Lyapunov-Floquet metric allows to capture asymptotic (spectral) characteristics of the periodic cocycle $\Xi$ on the infinitesimal level. However, here we are interested in short-time (rather than infinitesimal) computations having in mind that the periodic cocycle is the derivative cocycle over a periodic trajectory embedded into a chaotic attractor.

In this regard, it should be noted that Lyapunov-Floquet metrics are computationally relevant only for the shortest periodic orbits. For example, for the Rabinovich system studied in Section \ref{SEC: RabinovichLargestULE} this limits to only first 3 periodic orbits. This is caused by the following two reasons. Firstly, basis vectors $e_{j}(q)$ corresponding to contracting (expanding) directions are usually very short (large) w.r.t. to a standard inner product (in which our calculations are done) to be referred as Euclidean. Secondly, the angle between some basis vectors may be close to a multiple of $\pi$ (the usual scenario for non-uniformly hyperbolic attractors). This gives drastically small or, conversely, large eigenvalues for metrics and destroys numerical computations.

However, experiments suggest that nontrivial Lyapunov exponents over periodic orbits decay as the period increases. Thus, orbits which are extreme for uniform Lyapunov exponents over the attractor are expected to be the shortest ones (in accordance with observations of B.R.~Hunt and E.~Ott \cite{HuntOtt1996}; see also Section \ref{SEC: RobustAlgoEstimateULE}). Here, Lyapunov-Floquet metrics may serve as a base for constructing metrics in a neighborhood of the attractor providing a test for finding appropriate architectures of neural networks for further learning. This is illustrated in Section \ref{SEC: RabinovichLargestULE}.

Another important feature of Lyapunov-Floquet metrics is that they capture symmetries of the system. We illustrate this by the following proposition.
\begin{proposition}
	Let $S$ be a diffeomorphism of $\mathbb{R}^{n}$ which is a symmetry for a smooth vector field $F$ in $\mathbb{R}^{n}$, i.e.
	\begin{equation}
		\label{EQ: SymmetryODE}
		D_{v}SF(v) = F(S(v)) \text{ for any } v \in \mathbb{R}^{n}
	\end{equation}
	or, equivalently, $v(\cdot)$ is a solution of $\dot{v} = F(v)$ iff $(Sv)(\cdot) \coloneq S(v(\cdot))$ is also a solution.
	
	Suppose $v(t)$ is a $\sigma$-periodic solution of $\dot{v}(t)=F(v(t))$.
	
	Let $\langle \cdot, \cdot \rangle$ be an inner product in $\mathbb{E} = \mathbb{C}^{n}$ (or in $\mathbb{E} = \mathbb{R}^{n}$ if the real logarithm exists; see Remark \ref{REM: LyapunovFloquetRealSpace}) and $Q(t)$ be the operator\footnote{That is the above construction applied to $A(t) \coloneq D_{S(v(t))}F$ (more rigorously, $A(t)$ is the complexification of $D_{S(v(t))}F$).} of a Lyapunov-Floquet metric over $Sv(t)$ w.r.t. $\langle \cdot, \cdot \rangle$ (see \eqref{EQ: OperatorLFQ}). Then (the adjoint is taken w.r.t. $\langle \cdot, \cdot \rangle$ and we omit mentioning complexifications)
	\begin{equation}
		\label{EQ: LyapunovFloquetMetricSymmetry}
		P(t) = (D_{v(t)} S)^{*} Q(t) D_{v(t)} S
	\end{equation}
	is the operator of a Lyapunov-Floquet metric over $v(t)$ for $t \in [0,\sigma]$.
\end{proposition}
\begin{proof}
	Let $\dot{\xi}(t) = D_{v(t)}F \xi(t)$ and $\xi(0) = \xi_{0} \in \mathbb{E}$. Then \eqref{EQ: SymmetryODE} gives
	\begin{equation}
		\frac{d}{dt}\left[D_{v(t)}S\xi(t)\right] = D_{(Sv)(t)}F D_{v(t)}S \xi(t)
	\end{equation}
	or, in terms of the cocycle we have $\Xi^{t}(q,\xi_{0}) = \xi(t)$, $q = v(0)$ and
	\begin{equation}
		\label{EQ: CocycleSymmetryIntegral}
		\Xi^{t}(q,\xi_{0}) = (D_{\vartheta^{t}(q)} S)^{-1} \Xi^{t}(S(q), D_{q}S \xi_{0}).
	\end{equation}
	Let $e_{1}((Sv)(t)), \ldots, e_{n}((Sv)(t))$ be a periodic basis as in \eqref{EQ: FloquetBasisProperties} over $(Sv)(t)$ and $\langle \cdot, \cdot \rangle_{(Sv)(t)}$ be the associated Lyapunov-Floquet metric. Then $Q(t)=Q^{*}(t)$ is determined via $\langle \cdot, \cdot \rangle$ as 
	\begin{equation}
		\label{EQ: OperatorLFQ}
		\langle \cdot, \cdot \rangle_{(Sv)(t)} =  \langle \cdot, Q(t) \cdot \rangle.
	\end{equation}
	According to \eqref{EQ: CocycleSymmetryIntegral}, $e_{j}(v(t)) \coloneq (D_{v(t)}S)^{-1}e_{j}((Sv)(t))$ taken over $j \in \{1,\ldots,n\}$ form a periodic basis diagonalizing $\Xi$ over $v(t)$. From this, \eqref{EQ: LyapunovFloquetMetricSymmetry} and \eqref{EQ: OperatorLFQ} for any $j,k \in \{1,\ldots,n\}$ we have
	\begin{equation}
		\begin{split}
			\left\langle e_{j}(v(t)), P(t) e_{k}(v(t))  \right\rangle = \left\langle e_{j}(v(t)), (D_{v(t)}S)^{*} Q(t) D_{v(t)}S e_{k}(v(t)) \right\rangle =\\
			= \left\langle D_{v(t)}S e_{j}(v(t)),  Q(t) D_{v(t)}S e_{k}(v(t)) \right\rangle = \left\langle e_{j}((Sv)(t)), e_{k}((Sv)(t)) \right\rangle_{(Sv)(t)}.
		\end{split}
	\end{equation}
	Thus the inner product $\langle \cdot, P(t) \cdot \rangle$ makes the basis $\{ e_{j}(v(t)) \}_{j = 1}^{n}$ orthonormal and, consequently, it is the associated Lyapunov-Floquet metric over $v(t)$.
\end{proof}
\section{Operators on exterior powers and singular values}
\label{SEC: OperatorsOnExteriorProducts}
In this appendix we briefly recall some definitions from exterior algebra just to fix notations for the main text.

Let $\mathbb{E}$ be a (finite-dimensional) vector space over $\mathbb{R}$ or $\mathbb{C}$. For any integer $m \geq 1$, we denote by $\mathbb{E}^{\wedge m}$ the $m$-fold exterior power of $\mathbb{E}$. Recall that it is spanned by antisymmetric tensors $v_{1} \wedge \ldots \wedge v_{m}$, where $v_{1},\ldots,v_{m} \in \mathbb{E}$. We have $v_{1} \wedge \ldots \wedge v_{m} = 0$ iff the vectors are linearly dependent. Moreover, for two linearly independent sets of vectors the equality $v_{1} \wedge \ldots \wedge v_{m} = w_{1} \wedge \ldots \wedge w_{m}$ holds iff both $m$-tuples span the same $m$-dimensional subspace $\mathbb{L} \subset \mathbb{E}$ and there exists a special linear transformation of $\mathbb{L}$ taking $v_{1},\ldots,v_{m}$ into $w_{1},\ldots,w_{m}$ respectively. Because of this one says that $v_{1} \wedge \ldots \wedge v_{m}$ is an oriented $m$-volume in $\mathbb{L}$ represented by the parallelepiped with edges $v_{1},\ldots,v_{m}$.

If $\mathbb{E}$ is endowed with an inner product $\langle \cdot, \cdot \rangle$, there is an associated inner product in $\mathbb{E}^{\wedge m}$ given by (the quadratic form is the Gram determinant)
\begin{equation}
	\label{EQ: InnerProductExteriorPower}
	\langle v_{1} \wedge \ldots \wedge v_{m}, w_{1} \wedge \ldots \wedge w_{m} \rangle \coloneq \det( \langle v_{i}, w_{j} \rangle )_{i,j=1}^{m}
\end{equation}
and extended to entire $\mathbb{E}^{\wedge m}$ by linearity in each argument.

With any linear operator $L \colon \mathbb{E} \to \mathbb{F}$ between vector spaces $\mathbb{E}$ and $\mathbb{F}$, one can associate the operator $L^{\wedge m} \colon \mathbb{E}^{\wedge m} \to \mathbb{F}^{\wedge m}$ given by
\begin{equation}
	L^{\wedge m}(v_{1} \wedge \ldots \wedge v_{m}) = Lv_{1} \wedge \ldots \wedge L v_{m}
\end{equation}
for $v_{1},\ldots,v_{m} \in \mathbb{E}$ and extended by linearity to entire $\mathbb{E}^{\wedge m}$.

Let $\langle \cdot, \cdot \rangle_{\mathbb{E}}$ and $\langle \cdot, \cdot \rangle_{\mathbb{F}}$ be inner products in $\mathbb{E}$ and $\mathbb{F}$ respectively. Then the adjoint $L^{*} \colon \mathbb{F} \to \mathbb{E}$ of $L$ is given by
\begin{equation}
	\langle Lv, f \rangle_{\mathbb{F}} = \langle v, L^{*}f \rangle_{\mathbb{E}} \text{ for } v \in \mathbb{E}, \ f \in \mathbb{F}. 
\end{equation}
Let $n = \dim \mathbb{E}$. Then \textit{singular values} $\alpha_{1}(L) \geq \alpha_{2}(L) \geq \ldots \geq \alpha_{n}(L)$ of $L$ are defined as the arranged eigenvalues of the operator $\sqrt{L^{*}L}$.

\begin{remark}
	\label{REM: SingularValuesComputation}
	In the above situation, choose bases in $\mathbb{E}$ and $\mathbb{F}$. Then there are matrices $P_{\mathbb{E}}$, $P_{\mathbb{F}}$ and $M_{L}$ representing the quadratic forms $\langle \cdot, \cdot \rangle_{\mathbb{E}}$ and $\langle \cdot, \cdot \rangle_{\mathbb{F}}$ and the operator $L$ respectively in these bases. Then the singular values of $L$ are given by the singular values of the matrix 
	\begin{equation}
	 \sqrt{ P_{\mathbb{F}} } M_{L} \sqrt{P_{\mathbb{E}}}^{-1}.
	\end{equation}
\end{remark}

A key result\footnote{To prove it, take the orthonormal basis $e_{1},\ldots,e_{n}$ of eigenvectors of $L^{*}L$ and look at the action of $L^{\wedge m}$ in terms of the associated basis $e_{j_{1}} \wedge \ldots \wedge e_{j_{m}}$, where $1 \leq j_{1} < \ldots < j_{m} \leq n$, on $\mathbb{E}^{\wedge m}$.} states that
\begin{equation}
	\label{EQ: KeyResultExteriorOperators}
	\| L^{\wedge m} \| = \sup_{\substack{v_{1},\ldots,v_{m} \in \mathbb{E}\\ |v_{1} \wedge \ldots \wedge v_{m}| = 1}} |Lv_{1} \wedge \ldots \wedge L v_{m}| = \prod_{j=1}^{m} \alpha_{j}(L),
\end{equation}
where $|\cdot|$ denote appropriate norms in $\mathbb{E}^{\wedge m}$ and $\mathbb{F}^{\wedge m}$ induced by the associated with $\langle \cdot, \cdot \rangle_{\mathbb{E}}$ and $\langle \cdot, \cdot \rangle_{\mathbb{F}}$ inner products  respectively and $\| \cdot \|$ is the associated operator norm.

At the end let us introduce the \textit{function of singular values} (of order $d$) $\omega_{d}(L)$ of $L$ (as above) defined for $d = m + s$ with integer $m \geq 0$ and $s \in (0,1]$ by
\begin{equation}
	\label{EQ: SingularValuesFunctionDef}
	\omega_{d}(L) = \prod_{j=1}^{m} \alpha_{j}(L) \cdot \alpha^{s}_{m+1}(L)
\end{equation}
and with the convention that $\omega_{0}(L) \coloneq 1$. Note that $\omega_{d}(L) = \omega^{1-s}_{m}(L) \omega^{s}_{m+1}(L)$. Then from \eqref{EQ: KeyResultExteriorOperators} and \eqref{EQ: SingularValuesFunctionDef} it is clear that $\omega_{d}(L_{2} L_{1}) \leq \omega_{d}(L_{1}) \omega_{d}(L_{2})$ for any operators $L_{1}$ and $L_{2}$ defined on appropriate spaces.

Let $\mathbb{E}$ be a vector bundle over a complete metric space $\mathcal{Q}$. If $\mathfrak{n} = \{ \mathfrak{n}_{q}\}_{q \in \mathcal{Q}}$ is a family of norms (a \textit{metric} on $\mathbb{E}$) $\mathfrak{n}_{q}$ defined in fibers $\mathbb{E}_{q}$ of $\mathbb{E}$ over $q \in \mathcal{Q}$ and $L \colon \mathbb{E} \to \mathbb{E}$ is a bundle morphism given by fiber mappings $L_{q} \colon \mathbb{E}_{q} \to \mathbb{E}_{\vartheta(q)}$ for some transformation $\vartheta \colon \mathcal{Q} \to \mathcal{Q}$, by $\| L_{q} \|_{\mathfrak{n}}$ we denote the norm of $L_{q} \colon (\mathbb{E}_{q},\mathfrak{n}_{q}) \to (\mathbb{E}_{\vartheta(q)},\mathfrak{n}_{\vartheta(q)})$. Analogously, if each $\mathfrak{n}_{q}$ is induced by an inner product $\langle \cdot, \rangle_{q}$ in $\mathbb{E}_{q}$, by $\omega^{(\mathfrak{n})}_{d}(L_{q})$ we denote the function of singular values of $L_{q} \colon (\mathbb{E}_{q}, \langle\cdot,\cdot\rangle_{q}) \to (\mathbb{E}_{\vartheta(q)}, \langle\cdot,\cdot\rangle_{\vartheta(q)})$.



\begin{funding}
The reported study was funded by the Russian Science Foundation (Project 22-11-00172).
\end{funding}

\section*{Data availability}
The data that support the findings of this study can be generated using the scripts in the repository:

\section*{Conflict of interest}
The author has no conflicts of interest to declare that are relevant
to the content of this article.



\begin{thebibliography}{99}







\bibitem{AbadBarrioDena2011}
Abad~A., Barrio~R., Dena~A. Computing periodic orbits with arbitrary precision. \textit{Phys. Rev. E Stat. Nonlin. Soft Matter Phys.}, \textbf{84}(1), 016701 (2011)

\bibitem{AhujaMagnantiOrlinNF1993}
Ahuja~R.K., Magnanti~T.L., Orlin~J.B. \textit{Network Flows: Theory, Algorithms,
	and Applications}. Prentice Hall (1993)

\bibitem{Anikushin2023LyapExp}
Anikushin~M.M. Variational description of uniform Lyapunov exponents via adapted metrics on exterior products. \textit{arXiv preprint}, arXiv:2304.05713 (2025)

\bibitem{AnikushinRomanov2024EffEst}
Anikushin~M.M., Romanov~A.O. On the recent progress in effective dimension
estimates for delay equations. \textit{Differ. Uravn. Protsessy Upravl.}, 1, 22--46 (2024)

\bibitem{AnikushinRomanov2023FreqConds}
Anikushin~M.M., Romanov~A.O. Frequency conditions for the global stability of nonlinear delay equations with several equilibria. \textit{arXiv preprint}, \textit{arXiv:2306.04716} (2024)

\bibitem{Anikushin2023Comp}
Anikushin~M.M. Spectral comparison of compound cocycles generated by delay equations in Hilbert spaces. \textit{arXiv preprint}, arXiv:2302.02537 (2024)

\bibitem{Anikushin2020FreqDelay}
Anikushin~M.M. Frequency theorem and inertial manifolds for neutral delay equations. \textit{J. Evol. Equ.}, \textbf{23}, 66 (2023)

\bibitem{AnikushinRom2023SS}
Anikushin~M.M., Romanov~A.O. Hidden and unstable periodic orbits as a result of homoclinic bifurcations in the {S}uarez-{S}chopf delayed oscillator and the irregularity of {ENSO}. \textit{Phys. D: Nonlinear Phenom.}, \textbf{445}, 133653 (2023)

\bibitem{Anikushin2022Semigroups}
Anikushin~M.M. Nonlinear semigroups for delay equations in Hilbert spaces, inertial manifolds and dimension estimates. \textit{Differ. Uravn. Protsessy Upravl.}, 4, (2022)

\bibitem{Anikushin2020Geom}
Anikushin~M.M. Inertial manifolds and foliations for asymptotically compact cocycles in Banach spaces. \textit{arXiv preprint}, arXiv:2012.03821v2 (2022)

\bibitem{Anikushin2020FreqParab}
Anikushin M.M. Frequency theorem for parabolic equations and its relation to inertial manifolds theory. \textit{J. Math. Anal. Appl.}, \textbf{505}(1), 125454 (2021)

\bibitem{AnikushinAADyn2021}
Anikushin M.M. Almost automorphic dynamics in almost periodic cocycles with one-dimensional inertial manifold. \textit{Differ. Uravn. Protsessy Upravl.}, 2, (2021), in Russian

\bibitem{Anikushin2019Liouv}
Anikushin~M.M. On the Liouville phenomenon in estimates of fractal dimensions of forced quasi-periodic oscillations. \textit{Vestn. St. Petersbg. Univ., Math.}, \textbf{52}(3), 234--243 (2019)

\bibitem{AnikushinReitRom2019}
Anikushin~M.M., Reitmann V., Romanov A.O. Analytical and numerical estimates of the fractal dimension of forced quasiperiodic oscillations in control systems. \textit{Differ. Uravn. Protsessy Upravl.}, 2 (2019), in Russian

\bibitem{AnsmannGJitCDDE2018}
Ansmann~G. Efficiently and easily integrating differential equations with JiTCODE, JiTCDDE, and JiTCSDE. \textit{Chaos}, \textbf{28}(4), 043116 (2018)

\bibitem{Alessandroelal1990}
D'Alessandro~G., Grassberger~P., Isola~S., Politi~A. On the topology of the H\'{e}non map. \textit{J. Phys. A Math. Gen.}, \textbf{23}(22), 5285 (1990)

\bibitem{BarrioDenaTucker2015}
Barrio~R., Dena~A., Tucker~W. A database of rigorous and high-precision periodic orbits of the Lorenz model. \textit{Comput. Phys. Commun.}, \textbf{194}, 76--83 (2015)

\bibitem{Bochi2018}
Bochi~J. Ergodic optimization of Birkhoff averages and Lyapunov exponents. \textit{Proceedings of the International Congress of Mathematicians} (2018)

\bibitem{BoronskiStimac2023}
Boro\'{n}ski~J.P., \v{S}timac~S. The pruning front conjecture, folding patterns and classification of H\'enon maps in the presence of strange attractors. \textit{arXiv preprint}, arXiv:2302.12568 (2023)

\bibitem{BramburgerBruntonKutz2021}
Bramburger~J.J., Brunton~S.L., Kutz~J.N. Deep learning of conjugate mappings.
 \textit{Phys. D: Nonlinear Phenom.}, \textbf{427}, 133008 (2021)

\bibitem{ChenPesin2010}
Chen~J., Pesin~Y. Dimension of non-conformal repellers: a survey. \textit{Nonlinearity}, \textbf{23}(4), R93 (2010)


\bibitem{ChepyzhovIlyin2004}
Chepyzhov~V.V., Ilyin~A.A. On the fractal dimension of invariant sets; applications to Navier-Stokes equations. \textit{Discrete Contin. Dyn. Syst.}, \textbf{10}(1\&2) 117--136 (2004)

\bibitem{CurtisMitchellOverton2017}
Curtis~F.E., Mitchell~T., Overton~M.L. A BFGS-SQP method for nonsmooth, nonconvex, constrained optimization and its evaluation using relative minimization profiles. \textit{Optim. Methods Softw.}, \textbf{32}(1), 148--181  (2017)

\bibitem{CurtisOverton2012}
Curtis~F.E., Overton~M.L. A sequential quadratic programming algorithm for nonconvex, nonsmooth constrained optimization. \textit{SIAM J. Optim.}, \textbf{22}(2), 474--500 (2012)

\bibitem{CvitanovicChaos2005}
Cvitanovi\'{c}~P., Artuso~R., Mainieri~R., Tanner~G., Vattay~G., Whelan~N. and Wirzba~A. \textit{Chaos: Classical and Quantum}. ChaosBook.org edition17.6.4, May 6 2023 (2023)

\bibitem{DellnitzJunge2002}
Dellnitz~M., Junge~O. Set oriented numerical methods for dynamical systems. \textit{Handbook of dynamical systems}, 2, 221--264 (2002)

\bibitem{DevaneyNitecki1979}
Devaney~R., Nitecki~Z. Shift automorphisms in the H\'{e}non mapping. \textit{Comm. Math. Phys.}, \textbf{67}, 137--146 (1979)

\bibitem{DouadyOesterle1980}
Douady~A., Oesterl\'{e}~J.. Dimension de Hausdorff des attracteurs. \textit{C. R. Acad. Sci. Paris, Ser. A}, \textbf{290}, 1135--1138 (1980)

\bibitem{EdenLocalEstimates1990}
Eden~A. Local estimates for the Hausdorff dimension of an attractor. \textit{J. Math. Anal. Appl.}, \textbf{150}(1), 100--119 (1990)

\bibitem{Eden1989Thesis}
Eden~A. An abstract theory of L-exponents with applications to dimension analysis. \textit{Ph.D. thesis}, Indiana University (1989)


\bibitem{FroylandJungeochs2001}
Froyland~G., Junge~O., Ochs~G. Rigorous computation of topological entropy with respect to a finite partition. \textit{Phys. D: Nonlinear Phenom.}, \textbf{154}(1-2), 68--84  (2001)

\bibitem{Henon1976}
H\'{e}non~M.A. A two-dimensional mapping with a strange attractor. \textit{Commun. Math. Phys.}, \textbf{50}, 69--77 (1976)

\bibitem{HuntOtt1996}
Hunt~B.R., Ott~E. Optimal periodic orbits of chaotic systems occur at low period. \textit{Phys. Rev. E}, \textbf{54}(1), 328 (1996)

\bibitem{JoshyHwang2024}
Joshy~A.J., Hwang~J.T. PySLSQP: A transparent Python package for the SLSQP optimization algorithm modernized with utilities for visualization and post-processing. \textit{arXiv preprint}, arXiv:2408.13420 (2024)

\bibitem{KatokSB1980}
Katok~S.B. The estimation from above for the topological entropy of a diffeomorphism. In Global theory of dynamical systems Lecture Notes in Math., \textbf{819}, Springer, Berlin, 258--264 (1980)

\bibitem{KawanHafsteinGiesl2021}
Kawan~C., Hafstein~S., Giesl~P. A subgradient algorithm for data-rate optimization in the remote state estimation problem. \textit{SIAM J. Appl. Dyn. Syst.}, \textbf{20}(4), 2142--2173 (2021)

\bibitem{KawanPogromsky2021}
Kawan~C., Matveev~A.S., Pogromsky~A.Yu. Remote state estimation problem: Towards the data-rate limit along the avenue of the second Lyapunov method. \textit{Automatica}, \textbf{125}, 109467 (2021)

\bibitem{Kozlovski1998}
Kozlovski~O.S. An integral formula for topological entropy of maps. \textit{Ergod. Theory Dyn. Syst.}, \textbf{18}(2), 405--424 (1998)

\bibitem{Kraft1988}
Kraft~D. A software package for sequential quadratic programming. \textit{Tech. Rep.
	DFVLR-FB 88-28}, DLR German Aerospace Center (1988)

\bibitem{KuzMokKuzKud2020}
Kuznetsov N.V., Mokaev T.N., Kuznetsova O.A., Kudryashova E.V. The Lorenz
system: hidden boundary of practical stability and the Lyapunov dimension. \textit{Nonlinear Dyn.}, \textbf{102}, 713--732 (2020)

\bibitem{KuzReit2020}
Kuznetsov~N.V., Reitmann~V. \textit{Attractor Dimension Estimates for Dynamical Systems: Theory and Computation}. Switzerland: Springer International Publishing AG (2021)

\bibitem{Kuznetsov2016}
Kuznetsov~N.V. The Lyapunov dimension and its estimation via the Leonov method. \textit{Phys. Lett. A}, \textbf{380}(25--26) 2142--2149 (2016)

\bibitem{KuzLeoShum2015}
Kuznetsov~N.V., Leonov~G.A., Shumafov~M.M. A short survey on Pyragas time-delay feedback stabilization and odd number limitation. \textit{IFAC-PapersOnLine}, \textbf{48}(11), 706--709 (2015)

\bibitem{LeoKuzKorKusakin2016}
Leonov~G.A., Kuznetsov~N.V., Korzhemanova~N.A., Kusakin~D.V. Lyapunov dimension formula for the global attractor of the Lorenz system. \textit{Commun. Nonlinear Sci. Numer. Simul.}, \textbf{41}, 84--103 (2016)

\bibitem{LeoZvyagKuz2016}
Leonov~G.A., Zvyagintseva~K.A., Kuznetsova~O.A. Pyragas stabilization of discrete systems via delayed feedback with periodic control gain. \textit{IFAC-PapersOnLine}, \textbf{49}(14), 56--61 (2016)

\bibitem{Leonov2016LorenzLike}
Leonov~G.A. Lyapunov dimension formulas for Lorenz-like systems. \textit{Doklady Mathematics}, \textbf{93}(3) (2016)

\bibitem{Leonov2015PeriodicPyragas}
Leonov~G.A. Pyragas stabilizability via feedback with periodic control gain. \textit{Doklady Mathematics}, \textbf{92}, 519--520 (2015)

\bibitem{LeonovBurkinShep2012}
Leonov~G.A., Burkin~I.M., Shepeljavyi~A.I. \textit{Frequency Methods in Oscillation Theory}. Springer Science \& Business Media (2012)

\bibitem{LeonovForHenonLor2001}
Leonov~G.A. Lyapunov dimensions formulas for H\'{e}non and Lorenz attractors. \textit{Alg. Anal.}, \textbf{13}, 155--170 (2001) (Russian); English transl. \textit{St. Petersburg Math. J.}, \textbf{13}(3), 453--464 (2002)

\bibitem{LeonovPonomarenkoSmirnova1996}
Leonov~G.A., Ponomarenko~D.V., Smirnova~V.B. \textit{Frequency-Domain Methods for Nonlinear Analysis: Theory and Applications}. World Scientific (1996)

\bibitem{LeonovBoi1992}
Leonov~G.A., Boichenko~V.A. Lyapunov's direct method in the estimation of the Hausdorff dimension of attractors. \textit{Acta Appl. Math.}, \textbf{26}(1), 1--60 (1992)

\bibitem{LiMuldowney1996SIAMGlobStab}
Li~M.Y., Muldowney~J.S. A geometric approach to global-stability problems. \textit{SIAM J. Math. Anal.}, \textbf{27}(4), 1070--1083 (1996)

\bibitem{LiMuldowney1995LowBounds}
Li~M.Y., Muldowney~J.S. Lower bounds for the Hausdorff dimension of attractors. \textit{J. Dyn. Differ. Equ.}, \textbf{7}(3), 457--469 (1995)

\bibitem{LinMaFengChen2010}
Lin~W., Ma~H., Feng~J., Chen~G. Locating unstable periodic orbits: When adaptation integrates into delayed feedback control. \textit{Phys. Rev. E Stat. Nonlin. Soft Matter Phys.}, \textbf{82}(4), 046214 (2010)

\bibitem{LouzeiroetAlAdaptedMetricsComp2022}
Louzeiro~M., Kawan~C., Hafstein~S., Giesl~P., Yuan~J. A projected subgradient method for the computation of adapted metrics for dynamical systems. \textit{SIAM J. Appl. Dyn. Syst.}, \textbf{21}(4), 2610--2641 (2022)

\bibitem{MatveevPogromsky2016}
Matveev~A.S., Pogromsky~A.Yu. Observation of nonlinear systems via finite capacity channels: Constructive data rate limits. \textit{Automatica}, \textbf{70}, 217--229 (2016)

\bibitem{Mischaikow2002}
Mischaikow~K. Topological techniques for efficient rigorous computation in dynamics. \textit{Acta Numer.}, 11, 435--477 (2002)

\bibitem{Morris2013}
Morris~I.D. Mather sets for sequences of matrices and applications to the study of joint spectral radii. \textit{Proc. Lond. Math. Soc.}, \textbf{107}(1), 121--150 (2013)

\bibitem{Newhouseetal2008}
Newhouse~S., Berz~M., Grote~J., Makino~K. On the estimation of topological entropy on surfaces. \textit{Contemp. Math.}, \textbf{469}, 243--270 (2008)

\bibitem{NewhousePignataro1993}
Newhouse~S., Pignataro~T. On the estimation of topological entropy. \textit{J. Stat. Phys.}, 72, 1331--1351 (1993)

\bibitem{Newhouse1989ContEnt}
Newhouse~S.E. Continuity properties of entropy. \textit{Ann. Math.}, \textbf{129}(1), 215--235 (1989)

\bibitem{Newhouse1988EntropyVol}
Newhouse~S. Entropy and volume. \textit{Ergod. Th. \& Dynam. Sys.}, \textbf{8}, 283--299 (1988)

\bibitem{Osipenko2023}
Osipenko~G. Computer-oriented tests for hyperbolicity and structural stability of dynamical system. \textit{Differ. Uravn. Protsessy Upravl.}, 3, 22--50 (2023)

\bibitem{OsipenkoAmpilova2019}
Osipenko~G.S., Ampilova~N.B. On the entropy of symbolic image of a dynamical system. \textit{Dinamicheskie Sistemy}, \textbf{9}(2), 116--132 (2019)

\bibitem{Osipenko2006}
Osipenko~G. \textit{Dynamical Systems, Graphs, and Algorithms}. Springer (2006)

\bibitem{OttGrebogiYourke1990}
Ott~E., Grebogi~C., Yorke~J.A. Controlling chaos. \textit{Phys. Rev. Lett.}, \textbf{64}(11), 1196--1199 (1990)

\bibitem{Parry1964}
Parry~W. Intrinsic markov chains. \textit{Trans. Am. Math. Soc.}, \textbf{112}(1), 55--66 (1964)

\bibitem{Pincus1992}
Pincus~S.M. Approximating Markov chains. \textit{Proc. Natl. Acad. Sci. U. S. A.}, \textbf{89}(10), 4432--4436 (1992)

\bibitem{Pincus1991ApEn}
Pincus~S.M. Approximate entropy as a measure of system complexity. \textit{Proc. Natl. Acad. Sci. U. S. A.}, \textbf{88}(6), 2297--2301 (1991)

\bibitem{Przytycki1980}
Przytycki~F. An upper estimation for topological entropy of diffeomorphisms. \textit{Invent. Math.}, \textbf{59}(3), 205--213 (1980)

\bibitem{Pyragas1992}
Pyragas~K. Continuous control of chaos by self-controlling feedback. \textit{Phys. Lett. A}, \textbf{170}(6), 421--428 (1992)

\bibitem{Rabinovich1978}
Rabinovich~M.I. Stochastic self-oscillations and turbulence. \textit{Physics-Uspekhi}, \textbf{125}(1), 123--168 (1978) (Russian) English transl. \textit{Sov. Phys. Usp.}, \textbf{21}(5), 443--469 (1978)


\bibitem{Smith1986HD}
Smith~R.A. Some applications of Hausdorff dimension inequalities for ordinary differential equations. \textit{P. Roy. Soc. Edinb. A}, \textbf{104}(3-4), 235--259 (1986)

\bibitem{Temam1997}
Temam~R. \textit{Infinite-Dimensional Dynamical Systems in Mechanics and Physics}. Springer (1997)

\bibitem{TomarKawanZamani2022}
Tomar~M.S., Kawan~C., Zamani~M. Numerical over-approximation of invariance entropy via finite abstractions. \textit{Syst. Control Lett.}, \textbf{170}, 105395 (2022)


\bibitem{Yomdin1987}
Yomdin~Y. Volume growth and entropy. \textit{Isr. J. Math.}, \textbf{57}, 285--300  (1987)

\bibitem{ZelikAttractors2022}
Zelik~S. Attractors. Then and now. \textit{Uspekhi Mat. Nauk}, \textbf{78}(4), 53--198 (2023)

\end{thebibliography}
\end{document}